\documentclass{amsart}

\usepackage{graphicx}
\usepackage[centertags]{amsmath}
\usepackage{amsfonts}
\usepackage{amssymb}
\usepackage{amsthm}
\usepackage{newlfont}
\usepackage{hyperref}
\usepackage{tikz}
\usepackage{subcaption}
\usepackage{enumitem}

\usepackage{mathtools} 

\newtheorem{thm}{Theorem}[section]
\newtheorem{cor}[thm]{Corollary}
\newtheorem{lem}[thm]{Lemma}

\theoremstyle{definition}
\newtheorem{defn}[thm]{Definition}

\theoremstyle{remark}
\newtheorem{rem}[thm]{Remark}

\numberwithin{equation}{section}
\newtheorem{fact}[thm]{Fact}
\newtheorem{claim}[thm]{Claim}

\newcommand{\R}{\mathbb{R}}
\newcommand{\N}{\mathbb{N}}
\newcommand{\Z}{\mathbb{Z}}

\newcommand{\dom}{\mathrm{dom}}

\newcommand{\res}{\restriction}

\newcommand{\Gf}{\Gamma}
\newcommand{\pathd}{\rho}

\newcommand{\fzn}{F(2^{\Z^n})}
\newcommand{\fzs}{F(2^{\Z^2})}
\newcommand{\sR}{\mathcal{R}}
\newcommand{\sS}{\mathcal{S}}
\newcommand{\sQ}{\mathcal{Q}}
\newcommand{\sA}{\mathcal{A}}

\newcommand{\sP}{\mathcal{P}}

\newcommand{\sN}{\mathcal{N}}

\newcommand{\tsQ}{\tilde{\sQ}}
\newcommand{\tsP}{\tilde{\sP}}
\newcommand{\tsR}{\tilde{\mathcal{R}}}

\newcommand{\sm}{\setminus}
\newcommand{\sgn}{\text{sgn}}

\newcommand{\tR}{\tilde{R}}
\newcommand{\stR}{\tilde{\mathcal{R}}}

\newcommand{\actson}{\curvearrowright}

\makeatletter
\def\Ddots{\mathinner{\mkern1mu\raise\p@
\vbox{\kern7\p@\hbox{.}}\mkern2mu
\raise4\p@\hbox{.}\mkern2mu\raise7\p@\hbox{.}\mkern1mu}}
\makeatother

\newcounter{mycounter}

\begin{document}

\title{Borel Combinatorics of Abelian Group Actions}
\author{Su Gao}
\address{School of Mathematical Sciences and LPMC, Nankai University, Tianjin 300071, P.R. China}
\email{sgao@nankai.edu.cn}
\thanks{The first author acknowledges the partial support of his research by the National Natural Science Foundation of China (NSFC) grants 12250710128 and 12271263. Research of the second author was partially supported by the U.S. NSF grant DMS-1800323.}

\author{Steve Jackson}
\address{Department of Mathematics, University of North Texas, 1155 Union Circle \#311430, Denton, TX 76203, U.S.A.}
\email{jackson@unt.edu}
\thanks{}

\author{Edward Krohne}
\email{krohneew@gmail.com}

\author{Brandon Seward}
\address{Department of Mathematics, University of California San Diego, San Diego, CA 92093, U.S.A.}
\email{b.m.seward@gmail.com}


\begin{abstract} We study the free part of the Bernoulli action of $\mathbb{Z}^n$ for $n\geq 2$ and the Borel combinatorics of the associated Schreier graphs. We construct orthogonal decompositions of the spaces into marker sets with various additional properties. In general, for Borel graphs $\Gamma$ admitting weakly orthogonal decompositions, we show that $\chi_B(\Gamma)\leq 2\chi(\Gamma)-1$ under some mild assumptions. As a consequence, we deduce that the Borel chromatic number for $F(2^{\mathbb{Z}^n})$ is $3$ for all $n\geq 2$. Weakly orthogonal decompositions also give rise to Borel unlayered toast structures. We also construct orthogonal decompositions of $F(2^{\mathbb{Z}^2})$ with strong topological regularity, in particular with all atoms homeomorphic to a disk. This allows us to show that there is a Borel perfect matching for $F(2^{\mathbb{Z}^n})$ for all $n\geq 2$ and that there is a Borel lining of $F(2^{\mathbb{Z}^2})$.
\end{abstract}

\date{\today}

\maketitle

\tableofcontents

\section{Introduction}

This paper is a contribution to the fast developing field of Borel combinatorics, which studies combinatorics of definable graphs and other structures on Polish spaces. For an overview of the entire field we refer the reader to the surveys \cite{KM} and \cite{Pikhurko2021}. 

The object of study for Borel combinatorics is a Borel graph $G$ on a Polish space $X$. Two main concepts that have been extensively explored are proper colorings and perfect matchings. A {\em proper coloring} of $G$ is a map $c: V(G)\to \kappa$, where $\kappa$ is a cardinal, such that for any $x, y\in V(G)$ with $(x,y)\in E(G)$, $c(x)\neq c(y)$. A {\em perfect matching} of $G$ is a map $m: V(G)\to V(G)$ such that for all $x\in V(G)$, $(x,m(x))\in E(G)$ and $m^2(x)=x$. Since these concepts are all functions, we may impose definability conditions on them and consider, for instance, continuous or Borel proper colorings (equipping the cardinal $\kappa$ with the discrete topology) and continuous or Borel perfect matchings, etc. 

The Borel graphs considered in this paper are Schreier graphs of marked group actions. A {\em marked group} is a pair $(\Gamma, S)$, where $\Gamma$ is a group and $S$ is a finite generating set of $\Gamma$. Usually we also require $S=S^{-1}$ and $1_\Gamma\not\in S$. When the generating set is standard or otherwise understood, we omit specifying the set $S$ and say that $\Gamma$ is a marked group. The {\em Cayley graph} $G(\Gamma, S)$ of a marked group $(\Gamma, S)$ is defined by
$$ V(G(\Gamma, S))=\Gamma $$
and
$$ E(G(\Gamma, S))=\{(g,h)\in \Gamma^2\,:\, \exists s\in S\ h=sg\}. $$
When there is a Borel action of a marked group $\Gamma$ on a Polish space $X$, the {\em Schreier graph} of the action $\Gamma\actson X$ on $X$, denoted $\Sigma(\Gamma, S)$, is defined by
$$ V(\Sigma(\Gamma, S))=X $$
and
$$ E(\Sigma(\Gamma, S))=\{(x,y)\in X^2\,:\, \exists s\in S\ y=s\cdot x\}. $$
The Schreier graph will be particularly nice when the action $\Gamma\actson X$ is free; in this case the Schreirer graph on each orbit of the action will be a copy of the Cayley graph $G(\Gamma,S)$.

In this paper the particular marked groups that we consider are $\mathbb{Z}^n$ with their standard sets of generators $\{\pm e_1,\dots, \pm e_n\}$ where for $1\leq i\leq n$, the $i$-th coordinate of $e_i$ is $1$ and the other coordinates are $0$. Most of the time we will be considering the {\em Bernoulli shift} action of $\mathbb{Z}^n$ on $2^{\Z^n}=\{0,1\}^{\Z^n}$, where for $g\in \Z^n$ and $x\in 2^{\Z^n}$, 
$$ (g\cdot x)(h)=x(g+h) $$
for all $h\in \Z^n$. In addition, since we are mostly interested in free actions, we consider the {\em free part} of the Bernoulli shift action:
$$ F(2^{\Z^n})=\{x\in 2^{\Z^n}\,:\, \forall 0\neq g\in \Z^n\ g\cdot x\neq x\}. $$

The systematic study of Borel combinatorics started with \cite{KST1999}, in which Kechris, Solecki, and Todorcevic introduced and studied the notion of Borel chromatic numbers. When there exists a Borel proper coloring $c: V(G)\to \kappa$ for a countable cardinal $\kappa$, we call the least such $\kappa$ the {\em Borel chromatic number} for $G$. A fundamental result proved in \cite{KST1999} states that, if in a Borel graph $G$ every vertex has degree $\leq k$ for some finite $k$, then the Borel chromatic number of $G$ is $\leq k+1$. Marks \cite{Marks2016} showed that this bound is optimal using Borel determinacy methods. He constructed marked group actions which give rise to Borel acyclic $k$-regular graphs with their Borel chromatic numbers achieving any number from $2$ to $k+1$. More recently, Conley, Jackson, Marks, Seward, and Tucker-Drob \cite{CJMST2020} improved these examples to hyperfinite Borel graphs, more specifically, Borel graphs which come from marked group actions where the orbit equivalence relation is hyperfinite. 

In the case of $F(2^{\Z^n})$, it is well known by an ergodicity argument that the Borel chromatic number is at least $3$, although the chromatic number of the Cayley graph of $\Z^n$ is $2$. For $F(2^{\Z})$ it is also well known that the Borel chromatic number is exactly $3$. For $F(2^{\Z^n})$ where $n\geq 2$, Gao and Jackson \cite{GJ2015} constructed a continuous proper coloring with $4$ colors. In \cite{GJKSinf} the authors of the present paper proved that the continuous chromatic number of $F(2^{\Z^n})$, where $n\geq 2$, is exactly $4$. We also announced that the Borel chromatic number of $F(2^{\Z^n})$ is $3$. In this paper we present a proof.

\begin{thm}\label{thm:main1} Suppose $n\geq 2$ and $\Z^n\actson X$ is a free Borel action of $\Z^n$ on a Polish space $X$. Then the Borel chromatic number of the Schreier graph on $X$ is $3$.
\end{thm}

A classical theorem of K\"onig in graph theory (which is a special case of Hall's marriage theorem) asserts that if $G$ is a $k$-regular bipartite graph then $G$ has a perfect matching. The Borel version of the statement turns out to be false in general. Marks \cite{Marks2016} constructed, for every $k\geq 2$, a $k$-regular acyclic Borel bipartite graph with no Borel perfect matching. These Borel graphs come from actions of free groups with finitely many generators. In this paper we prove that there is a Borel matching for $F(2^{\Z^n})$ for any $n\geq 2$.

\begin{thm}\label{thm:main2} Suppose $n\geq 2$ and $\Z^n\actson X$ is a free Borel action of $\Z^n$ on a Polish space $X$. Then there is a Borel perfect matching for the Schreier graph on $X$.
\end{thm}

A related concept in combinatoric  is {\em proper edge-coloring} for a graph $G$, which is a map $k: E(G)\to \kappa$, where $\kappa$ is a cardinal, such that for any $e, f\in E(G)$ such that $e$ and $f$ share exactly one vertex, $k(e)\neq k(f)$. The least cardinal $\kappa$ that admits a proper edge-coloring $k:E(G)\to \kappa$ is called the {\em edge chromatic number} of $G$. A classical theorem of Vizing states that if in a graph $G$ every vertex has degree $\leq k$, then the edge chromatic number of $G$ is $\leq k+1$. By a theorem of K\"onig, a $k$-regular bipartite graph $G$ has edge chromatic number $k$. The Borel versions of these results turn out all to be false in general. Marks \cite{Marks2016} constructed $k$-regular acyclic Borel graphs which all have Borel chromatic number $2$ but the Borel edge chromatic numbers vary arbitrarily from $k$ to $2k-1$. 

Recently, Greb\'ik--Rozho\v{n} \cite{GRinf} and Weilacher \cite{Weilacherinf} independently showed that the Borel edge chromatic number for $F(2^{\Z^n})$, where $n\geq 2$, is $2n$. This implies our Theorem~\ref{thm:main2}. Bencs--Hr\v{u}skov\'{a}--T\'{o}th \cite{BHTinf} showed the same theorem for $n=2$, which is a weaker result but sufficient to deduce our theorem. Of course, the methods used in these proofs are all different from the one presented in this paper. In contrast, in \cite{GJKSinf} the authors of the present paper showed that the continuous edge chromatic number of $F(2^{\Z^2})$ is $5$. More recently, Gao--Wang--Wang--Yan \cite{GWWY} showed that the continuous edge chromatic number of $F(2^{\mathbb{Z}^n})$ for $n\geq 2$ is exactly $2n+1$.

In this paper we also consider the combinatorial concepts of line section and lining in a Schreier graph of a marked group action. Following \cite{GJKSinf}, a {\em line section}  of $\Sigma(\Gamma, S)$ is a subgraph $G$ where each vertex in $G$ has degree $2$. If $G$ is a line section, we call each connected component of $G$ a {\em $G$-line}. A line section $G$ is {\em complete} if $G$ meets every orbit of the action. A line section $G$ is {\em single} if for each $x\in X$, the intersection of $G$ with the orbit of $x$ is a nonempty single $G$-line. A {\em lining} $G$ is a single line section where the vertex set of $G$ is the entire space, and in particular it is a complete line section. A line section $G$ on $X$ is {\em Borel} ({\em clopen}, resp.) if for each generator $e\in S$, the set $\{x\in X\,:\, (x, e\cdot x)\in G\}$ is Borel (clopen, resp.).

The authors of the present paper studied the existence of clopen line sections in \cite{GJKSinf}. We showed that there do not exist clopen single line sections of $F(2^{\Z^2})$; in particular, there are no clopen linings of $F(2^{\Z^2})$. In this paper we prove that there exists a Borel lining of $F(2^{\Z^2})$.

\begin{thm} \label{thm:main3}
Suppose $\Z^2\actson X$ is a free Borel action of $\Z^2$ on a Polish space $X$. Then there is a Borel lining for the Schreier graph on $X$. 
\end{thm}

Recently, Chandgotia--Unger \cite{CU2024} has proved the existence of Borel linings in higher dimensions with a different method.

The method we use to prove all these combinatorial results can be collectively called the {\em orthgonal marker method}. The simplest form of this method was first developed by Gao and Jackson in \cite{GJ2015} to tackle the hyperfiniteness problem for countable abelian group actions. Later Schneider and Seward \cite{SSinf} extended the orthogonal marker method to countable locally nilpotent groups and used it to show the hyperfiniteness of the orbit equivalence relations of their actions. It was already clear in \cite{GJ2015} that the orthognal marker regions we constructed not only can facilitate a proof of hyperfiniteness but also allows us to prove results of combinatorial nature about the Schreier graphs. In fact, it was shown in \cite{GJ2015} that for any $n\geq 2$, there is a continuous proper coloring of $F(2^{\Z^n})$ with $4$ colors.

In this paper we further develop the orthogonal marker method and use it to prove our theorems stated above. In Section 2 we first present a review of the orthogonal marker method and prove some improvements. In general, an {\em orthogonal marker structure} is a sequence $\{\sR_k\}_{k\geq 1}$ of partitions of the phase space $F(2^{\Z^n})$ such that each element in $\sR_k$ (which we refer to as a {\em marker region}) is a finite subset of an orbit with certain properties. We regard each $\sR_n$ as a {\em layer} of the orthogonal marker structure, and {\em orthogonality} refers to the relationship between  marker regions from different layers. The main improvement we explore in Section 2 is that, by controlling the geometric parameters used in the construction of $\sR_k$, we are able to obtain orthogonal marker regions that are connected. It turns out that connectedness is a key consideration in the  construction of combinatorial objects such as proper colorings, perfect matchings, and linings. 

Starting from Section 3 we adopt a slightly different point of view when we consider an orthogonal marker structure $\{\sR_k\}_{k\geq 1}$. Here, for each $k\geq 1$ we consider all the partitions $\sR_m$ for $m\geq k$ and the coarsest common refinement of all of them, denoted $\sA_k$. The elements of $\sA_k$ are called {\em $k$-atoms} of the orthogonal marker structure. We also consider the decomposition of $\sA_k$ into finite connected components and define notions of strong and weak orthogonality.  In Section 4 we work with weakly orthogonal decompositions with polynomial bound and some additional boundedness condition on the orthogonality constant and prove some results about their Borel chromatic numbers. These results imply Theorem~\ref{thm:main1}. We also show that weak orthogonal decompositions give rise to a combinatorial structure known as {\em toast} in the literature (defined in \cite{GJKS2022} and \cite{GJKSinf}). In particular, we show that there exists a Borel unlayered toast on $F(2^{\mathbb{Z}^n})$ for all $n\geq 1$.

In Section 5 we prove the main technical theorem of this paper about orthgonal decompositions. We show that for $n=2$, the orthogonal marker structure $\{\sR_k\}_{k\geq 1}$ can be constructred so that all $k$-atoms are homeomorphic to disks. This has consequences about the exact structures of $k$-atoms as well as their relationship with the $(k-1)$-atoms, which allow us to perform constructions in Section 6 of a Borel perfect matching (Theorem~\ref{thm:main2}) and in Section 7 of a Borel lining (Theorem~\ref{thm:main3}).

\section{Orthogonal markers and structures for $\fzn$} \label{sec:omzn}

We review the orthogonal marker structures for the equivalence relation $F_n$
given by the left shift action of the group $\Z^n$ on the free part $\fzn$
of the space $2^{\Z^n}$. These arguments are presented in detail in \cite{GJ2015}.
Here we review the construction, summarize the main results, and
present some improvements. In \S\ref{sec:odzn} we will introduce a
slightly different point of view which we refer to as an {\em orthogonal decomposition}.
The notion of an orthogonal decomposition will be a key component of many of our arguments.

\subsection{Review of clopen rectangular partition construction}

A frequent way to construct clopen rectangular partitions is to start with
$d$-marker sets. We recall this definition.

\begin{defn} \label{def:ms}
Let $d>0$ be an integer. A \textit{$d$-marker set} $M_d \subseteq \fzn$ is a set satisfying
the following:
\begin{enumerate}
\item \label{def:ms1}
$\forall x\neq y \in M_d\ \rho(x,y)>d$.
\item \label{def:ms2}
$\forall x \in \fzn\ \exists y \in M_d\ \rho(x,y)\leq d$. 
\end{enumerate}
\end{defn}

We often refer to $d$ as a \textit{marker distance}. We say the $d$-marker set $M_d$ is clopen if the set $M_d$ is a relatively clopen set in $\fzn$.

For each marker distance $d$ there is a clopen partition $\sR_d$ of
$\fzn$ into rectangles with side lengths in $\{ d, d+1\}$. More
precisely, $\sR_d$ is a subequivalence relation of $F_n$ with the
property that on each $F_n$ class $[x]$, $\sR_d \res [x]$ is a
partition of $[x]$ into finite sets, each of which is a rectangle in
the sense that it is of the form $R\cdot y$ for some $y \in [x]$ where
$R\subseteq \Z^n$ is a rectangle (i.e., $R=\Z^n \cap \prod_i
[a_i,b_i]$).  Also, $\sR_d$ is clopen in the sense that we have the
relative clopenness in $\fzn$ of the
set $C$ of $x \in \fzn$ such that $[x]_{\sR_d}=R\cdot x$ for some $R$
of the form $R=\prod_i [0,c_i]$. Intuitively, elements of the set $C$
are the ``bottom-left" corners of all rectangles in the partition
given by $\sR_d$. We will, with a slight abuse of terminolgy, refer to $\sR_d$
as a ``clopen partition'' or ``clopen finite subequivalence relation'' though
$\sR_d$ is not relatively clopen as a subset of $\fzn\times \fzn$.

In \cite{GJ2015} a clopen rectangular  partition $\sR_d$ with side lengths in
$\{ d, d+1\}$ is constructed by first starting with a $d'$-marker set $M_{d'}$ with $d' \gg d$.
About each point $x \in M_{d'}$ we consider the  ``cube''  $C=C_x$ in $[x]$ centered at
$x$ with side lengths $d'$. These cubes clearly cover $\fzn$ by property (\ref{def:ms2}) of Definition~\ref{def:ms}. We refer to this collection of cubes as a \textit{clopen covering} of $\fzn$. For each of these
cubes $C$ we ``adjust'' it by moving each of its faces outward with a distance bounded by a
small fraction of $d'$ (say no more than $\frac{1}{10}d'$). When doing the adjustment we ensure that
each of the faces $f$ of $C$ stays a certain distance, bounded below by a
fixed fraction $\epsilon d'$ of $d'$ (where $\epsilon$ is a fixed constant depending only
on $n$), away from any parallel face $f'$ of a cube $C'$ within a Hausdorff distance of $5d'$ from $C$.
To do these adjustments we use a \textit{big-marker-little-marker} procedure. More precisely, we use an auxiliary distance $d'' \gg d'$ and a
$d''$-marker set $M_{d''} \subseteq M_d$ (actually, to make $M_{d''}\subseteq M_{d'}$
we must relax slightly the inequalities in Definition~\ref{def:ms} by replacing
$d$ with $(1-\eta)d$ for property (\ref{def:ms1}) and $(1+\eta)d$ for property (\ref{def:ms2}) for some small $\eta$). We use $M_{d''}$
to organize the adjustment process, first adjusting the points of $M_{d''}$,
then at step $i$ adjusting the cubes centered at points in
$g_i\cdot M_{d'}$ not already considered, where $\{g_i\}$ is an enumeration of all $g\in \Z^n$ within distance $d''/2$ to the origin. As $d'' \gg d'$, the cubes being adjusted
at any given step do not interfere with each other. After finitely many steps
(roughly $(d'')^n$ many) we will have adjusted all of the cubes. This results in
a clopen partition $\sR_{d'}$ of $\fzn$ into regions each of which is the disjoint union of finitely many rectangles $R$
of side lengths between $\epsilon d'$ and $\frac{6}{5}d'$. If $d'$
is large enough so that $\epsilon d' > d^2$, then we proceed with a \textit{finite subdivision algorithm} to further subdivide
the $\sR_{d'}$ regions into rectangles with side lengths in $\{d, d+1\}$.
This produces the desired clopen rectangular partition $\sR_d$. The full details of this construction can be found in \cite{GJ2015}.

In summary, the key steps of the clopen rectangular partition construction for a marker distance $d$ are the following:
\begin{enumerate}
\item Find a clopen covering of $\fzn$ by rectangles of side lengths $d'\gg d$;
\item Using a big-marker-little-marker procedure, adjust the rectangles in the clopen covering so that any parallel faces of rectangles are at least $\epsilon d'$ apart, for some constant $\epsilon$ depending only on $n$; the adjusted clopen covering gives rise to a clopen partition of $\fzn$ into rectangular polygonal regions;
\item \label{step3} Use a finite subdivision algorithm to divide each of the rectangular polygonal regions in the clopen partition into rectangles of side lengths $d$ or $d+1$.
\end{enumerate}

In the following we give a variant of the finite subdivision algorithm in dimension $n=2$ so that in the final clopen rectangular partition, no points will be close to 4 different rectangles. 

\begin{lem}\label{lem:no4corners}
Let $0<\epsilon\leq 1/2$ and $d>0$ be a sufficiently large integer. Let
$R$ be a rectangular polygonal region in $\Z^2$ such that any parallel
sides of $R$ have a perpendicular distance at least $12 d$. Let $S$ be a
subset of the boundary of $R$ such that each point in $S$ is at least
$\frac{1}{60}\epsilon d$ apart from any corner of $R$ and any two
points of $S$ are at least $\epsilon d$ apart. Then there is a
partition $\sR$ of $R$ into rectangles such that
\begin{enumerate}
\item[\rm (i)] any rectangle in $\sR$ has side lengths between
  $\frac{1}{2} d$ and $d$;
\item[\rm (ii)] for any rectangle $\tilde{R}$ in $\sR$, any corner of
  $\tilde{R}$ is at least $\frac{1}{60}\epsilon d$ apart from any
  point in $S$;
\item[\rm (iii)] no point of $R$ is within $\frac{1}{60}\epsilon d$ of
  four different rectangles in $\sR$.
\end{enumerate}
\end{lem}

\begin{proof}
We first note that if $p$ is any point on the boundary of $R$ then
within $\frac{1}{10}\epsilon d$ of $p$ there is a point $q$ on the
boundary of $R$ such that $q$ is at least $\frac{1}{60}\epsilon d$
apart from any corner of $R$ or any point of $S$. In fact, if $p$ is
at least $\frac{1}{60}\epsilon d$ apart from any corner of $R$ or any
point of $S$, then we let $q=p$. Otherwise, some corner of $R$ or
point of $S$ is within $\frac{1}{60}\epsilon d$ of $p$. Consider all
the points on the boundary of $R$ that are within
$\frac{1}{20}\epsilon d$ of $p$. If there are no other corner of $R$
or point of $S$ in this set, then some point $q$ can be selected from
this set to satisfy the requirement. Otherwise, it must be the case
that both a corner of $R$ and a point of $S$ are within
$\frac{1}{20}\epsilon d$ of $p$. Now consider the set of all points on
the boundary of $R$ that are within $\frac{1}{10}\epsilon d$ of
$p$. By our assumption about $R$ and $S$, there are no other points in
this set that is either a corner of $R$ or a point of $S$. A point $q$
can be selected from this set to satisfy the requirement.

It is worth noting that the point $q$ can be chosen from either side
of $p$ on the boundary of $R$ as long as its choice is not blocked by
the existence of a corner on one side. This observation will be useful
in our argument below.

Next, we claim that if $p$ is any point on the boundary of $R$ that is
not itself a corner and $l$ is the unique line of the boundary of $R$
that contains $p$, then there is a point $q$ on $l$ within
$\frac{1}{5}\epsilon d$ of $p$ such that $q$ is at least
$\frac{1}{60}\epsilon d$ apart from any corner of $R$ or any point of
$S$, and, letting $s$ be the ray starting from $q$ that is
perpendicular to $l$ and pointing to the interior of $R$, then the
very first intersection of $s$ with another part of the boundary of
$R$ is a point $r$ that is also at least $\frac{1}{60}\epsilon d$
apart from any corner of $R$ or any point of $S$. 

To prove this claim
we need to consider a number of cases. However, the argument for all
these cases are similar. We only give the argument for one
case. Suppose $p$ is within $\frac{1}{60}\epsilon d$ of a corner of
$R$ and within $\frac{1}{30}\epsilon d$ of a point in $S$. This is illustrated in Figure~\ref{fig:lem1.2}.
\begin{figure}[h] 
\begin{tikzpicture}[scale=0.05]

\pgfmathsetmacro{\a}{0.5}
\pgfmathsetmacro{\b}{0.1}

\draw[thick] (0,0) to (0,70) to (35,70);
\draw[thick] (100,0) to (100,70);



\draw[fill=yellow, opacity=\b] (0,0) rectangle (100,70);




\node at (-5,65) {$p$};
\node at (50,50) {$R$};
\node at (-5, 55) {$S$};
\node at (-5, 35) {$q'$};
\node at (-5, 25) {$q$};
\node at (-5, 10) {$l$};
\node at (105, 10) {$m$};
\node at (105, 25) {$r$};
\node at (105, 35) {$r'$};
\node at (30, 39) {$s'$};
\node at (30, 22) {$s$};

\draw[->] (0,25) to (100, 25);
\draw[->] (0,35) to (100, 35);
\draw[->] (0,65) to (100, 65);

\draw[fill=black] (0,65) circle (0.8);
\draw[fill=black] (0, 55) circle (0.8);
\draw[fill=black] (0, 35) circle (0.8);
\draw[fill=black] (0, 25) circle (0.8);
\draw[fill=black] (100, 25) circle (0.8);
\draw[fill=black] (100, 35) circle (0.8);

\end{tikzpicture}
\caption{Avoiding corners of $R$ and points of $S$} \label{fig:lem1.2}
\end{figure}
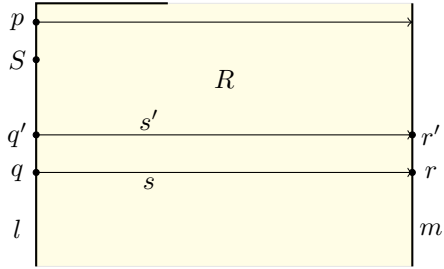
The
above mentioned corner of $R$ and the point in $S$ must be on two
opposite sides of $p$ on $l$. In this case we still can pick a point
$q'$ on $l$ within $\frac{1}{10}\epsilon d$ of $p$ such that $q'$ is at
least $\frac{1}{60}\epsilon d$ apart from any point of $S$. Obviously
$q'$ is on the opposite side of $p$ as the corner of $R$ is. Note that
by our assumptions the next point of $S$ is at least $\epsilon d$
apart from $p$, and therefore at least $\frac{9}{10}\epsilon d$ apart
from $q'$. Now let $s'$ be the ray through $q'$ that is perpendicular
to $l$ and pointing to the interior of $R$. Let $r'$ be the first
intersection of $s'$ with a boundary part of $R$, which we denote by
$m$. $m$ and $l$ are parallel. Draw a ray through $p$ that is
perpendicular to $l$ and pointing to the interior of $R$. Then on $m$,
on the opposite side of $r'$, there is not a corner of $R$ that is
within $\epsilon d$ of $r'$. This is because, otherwise, by comparing
the corner close to $p$ and such a corner, we obtain two parallel
sides of $R$ that are within perpendicular distance $d$ of each other,
contradicting our assumption. Thus, on this side of $r'$ of $m$ there
is a point $r$ within $\frac{1}{10}\epsilon d$ of $r'$ such that $r$
is at least $\frac{1}{60}\epsilon d$ apart from any point of $S$. Now
let $q$ be on $l$ and $s$ be the ray to give $r$ as the intersection
of $s$ with $m$. Then $p$ and $q$ are on opposite sides of $q'$ on
$l$, and $q$ is within $\frac{1}{10}\epsilon d$ of $q'$. Thus $q$ is
within $\frac{1}{5}\epsilon d$ of $p$ and is as required.

We use the above claim tacitly in the following finite subdivision
algorithm. The algorithm consists of four steps. In the first step,
for each horizontal boundary side $l$ of $R$ draw a horizontal line
$l'$ within $R$ so that the distance between $l$ and $l'$ is in
between $3 d$ and $4d$, and that the two intersections of $l'$ with
the boundary of $R$ are at least $\frac{1}{60}\epsilon d$ away from
any corner of $R$ or point of $S$. In the second step, for each
vertical boundary line $l$ of $R$ whose extension intersects the
interior of $R$, extend $l$ until it intersects the first horizontal
line that is created in the first step of the algorithm. The first two steps of the algorithm is illustrated in Figure~\ref{fig:lem1.22}.
\begin{figure}[h] 
\begin{tikzpicture}[scale=0.05]

\pgfmathsetmacro{\a}{0.5}
\pgfmathsetmacro{\b}{0.1}
  
\foreach \j in {0,1,2}   
{\draw[thick] (0+\j*90,0) to (50+\j*90,0) to (50+\j*90,30) to (70+\j*90,30) to (70+\j*90,80) to (30+\j*90,80) to (30+\j*90,60) to (0+\j*90,60) to (0+\j*90,0);}

\draw[->] (75,40) to (85,40);
\draw[->] (165,40) to (175, 40);

\foreach \j in {1,2}
{\draw (0+\j*90,10) to (50+\j*90,10);
\draw (0+\j*90, 40) to (70+\j*90, 40);
\draw (0+\j*90,50) to (70+\j*90, 50);
\draw (30+\j*90, 70) to (70+\j*90, 70);
}

\draw (50+180,30) to (50+180,40);
\draw (30+180,60) to (30+180,50);
\end{tikzpicture}
\caption{The first two steps of the subdivision algorithm} \label{fig:lem1.22}
\end{figure}
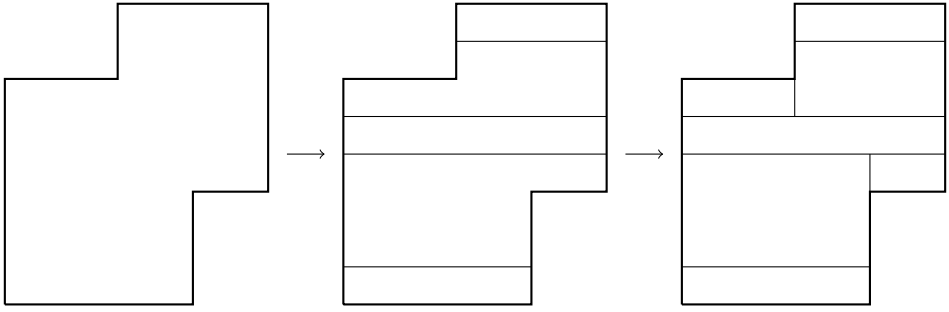
The first two
steps give a subdivision of $R$ into rectangles of side lengths at
least $3 d$. Also, no point is within $4d$ of four distinct rectangles.
For each of these rectangles, say $R'$, let $S'$
be the set of its non-corner boundary points that are either in $S$ or
next to a corner of another rectangle created in the first two
steps. Then any point of $S'$ is at least $\frac{1}{60}\epsilon d$
apart from a corner of $R'$ and any two points of $S'$ are at least
$\epsilon d$ apart if their distance is $>1$. Now in the third step of
the algorithm, we introduce vertical lines to each rectangle $R'$ to
subdivide it into smaller rectangles. For each $R'$, the distances
between any two vertical lines will be between $\frac{1}{2} d$ and
$d$, and the intersections of these new lines with the
boundary of $R'$ will be at least $\frac{1}{60}\epsilon d$ apart from any point
of $S'$. This can clearly be done to any $R'$ with the assumed
property of $S'$, using the fact that the edges of $R'$
have lengths at least $3d$. It takes a moment of reflection to see that after
this is done to each rectangle, we add to the $S'$ for a neighboring
rectangle, but we still maintain the property that any two points of
$S'$ is at least $\epsilon d$ apart if their distance is $>1$. Thus
this can be done to all rectangles one by one. Moreover, after all
these divisions by vertical lines are completed, each of the resulting
rectangles also satisfies the property that any two points of the new
$S'$ is at least $\epsilon d$ apart if their distance is $>1$ and at
least $\frac{1}{60} \epsilon d$ apart from a corner. In addition, the
resulting rectangles from the third step all have width between
$\frac{1}{2} d$ and $d$. Finally, in the fourth step of the
algorithm we introduce only horizontal lines to further subdivide the
rectangles resulting from the third step. The final resulting
rectangles will have side lengths between $\frac{1}{2} d$ and $d$, and
all of their corners will be at least $\frac{1}{60}\epsilon d$ apart
from other corners from neighboring rectangles or from a point of
$S$. Thus the proof of the lemma is complete.
\end{proof}

We can apply Lemma~\ref{lem:no4corners} to obtain clopen rectangular
partitions of $F(2^{\Z^2})$ in which no four rectangles come
together. We state this in the next lemma.

\begin{lem} \label{lem:n4c}
For all sufficiently large $d$ there is a clopen rectangular partition
on $\fzs$ having the following properties:
\begin{enumerate}
\item
All of the rectangles in the partition have side lengths between
$\frac{1}{2}d$ and $d$.
\item
No point is within distance $\frac{1}{120} d$ of four distinct rectangles of the
partition.
\end{enumerate}
\end{lem}

\begin{proof}

In order to get this we only need to modify step
(\ref{step3}) of the standard construction as sketched before Lemma~\ref{lem:no4corners}.
We also assume that $d' >d$ is sufficiently large so that
$\epsilon d' > 12 d$, where $\epsilon$ is as in $(2)$ of the
sketch before Lemma~\ref{lem:no4corners}. 
Using the big-marker-little-marker procedure
one produces a clopen covering of $\fzs$ by rectangles of side lengths
at least $d'$, and such that distinct parallel edges are at least $\epsilon d'$ apart. 
At each stage $i$ (as determined by the big-marker-little-marker procedure)
we subtract from each stage $i$ rectangle the rectangles that intersect it from earlier
stages. This produces a rectangular polyhedral region $R$ with the property
that parallel edges at least $\epsilon d' >12 d$ apart.
We simultaneously subdivide the rectangular polyhedron $R$ into smaller rectangular regions
using Lemma~\ref{lem:no4corners}. We apply the lemma using $\epsilon=\frac{1}{2}$
and the set $S$ being the points on the boundary of $R$ 
of distance $1$ from a perpendicular edge of a (smaller) rectangle produced at an
earlier stage together with the points on the boundary of $R$ which are on the
edge of a (larger, not yet subdivided) rectangle corresponding to a
later stage $j>i$ in the big-marker-little-marker procedure.
The resulting rectangular partition produced at the end has the desired property.
\end{proof}

\subsection{Review of orthogonal marker region construction} \label{subsec:omreview}
In \cite{GJ2015} we also gave a construction of the so-called orthogonal marker regions on $\fzn$.
For the orthogonal marker region construction we have a sequence $d_0<d_1<\cdots$
of marker distances. We refer to $d_k$ as the \textit{$k$-th marker distance}.
For our arguments we need the sequence of marker distances to grow reasonably fast.
We will implicitly assume in the following that there is a fixed constant $K=K(n)$ (depending on the dimension $n$)
such that $d_{k+1} >K d_k$ for all $k$. 

The orthogonal marker regions will be a sequence $\sR_m$ of clopen partitions of $\fzn$
into polygonal regions each of which is a union of rectangles.
We inductively define the $m$th clopen partition $\sR_m$. To begin, we let
$\sR^m_m$ be a clopen rectangular partition of $\fzn$ with side lengths in $\{ d_m, d_m+1\}$.
We may construct this partition so that parallel edges which are distance $> 1$ apart are
distance $> \epsilon_0 d_m$ apart, for some fixed $\epsilon_0 > 0$
(this was implicit in the argument of Lemma~\ref{lem:n4c} except here we do not
need Lemma~\ref{lem:no4corners} but rather can use a simpler
subdivision algorithm obtained by just extending the faces of the earlier
regions which intersect a rectangular region at some stage of the construction).

We successively modify the regions to produce clopen partitions
$\sR^m_{m-1}, \dots, \sR^m_1$, and we will let
$\sR_m=\sR^m_1$. Suppose we have defined $\sR^m_{k+1}$.  We again use
a big-marker-little-marker construction by working with $d''\gg d'\gg
d_k$ to produce rectangular marker regions $\tilde{\sR}$ of side
lengths between $\frac{1}{2} d_k$ and $d_k$. More specifically, let $\sR$
be a clopen covering of $\fzn$ by rectangles with side lengths
approximately $d'$. We then adjust the faces of rectangles in $\sR$ to
produce another clopen covering $\sR'$ such that parallel faces of
distinct regions in the covering are at least $3 \epsilon' d'$
apart, for some fixed $\epsilon'>0$. We may assume $\epsilon' d' >12 d_k$.
We then do a secondary adjustment to produce $\sR''$ by moving
the faces of the regions in $\sR'$ no more than $\epsilon' d'$ so
that the perpendicular distance between parallel faces of a region in
$\sR''$ and a nearby region of $\sR^j_k$ for all $k \leq j <m$ are at
least $\epsilon_1 d_k$ apart, for some fixed $\epsilon_1 >0$
which also depends only on $n$. 
We still have that parallel faces of distinct regions in
$\sR''$ are at least $\epsilon' d'>12 d_k$ apart.  We then do a subdivision
algorithm to produce a clopen partition $\tilde \sR$ so that each
rectangular region in the partition has side lengths between
$\frac{1}{2} d_k$ and $d_k$, and each face of a region in $\tilde \sR$
is at least $\epsilon_2 d_k$ from any parallel face of a nearby region
in $\sR^j_k$ for $k\leq j<m$, and at least $\epsilon_2 d_k$ from a parallel
face of a region in $\tilde \sR$ if it has distance $>1$ from that face.

We then get a clopen assignment to each
region $\tilde R$ of $\tilde \sR$ one of the regions of $\sR^m_{k+1}$
which intersect it. We then replace each region $R$ of $\sR^m_{k+1}$
with the union of the $\tilde R$ which are assigned to $R$. Thus we
have obtained $\sR^m_k$.  Note that each region $R$ of the partition
$\sR^m_{k}$ is a finite union of rectangular regions from the $\tilde
\sR$ partition.

If we let $\epsilon=\epsilon(n)=\min \{\epsilon_1,\epsilon_2\}$, then we have the
following properties of the $\sR^m_k$.

\begin{defn} \label{def:op}
The {\em orthogonality properties} of the $\sR^m_k$ are the following.
\begin{enumerate}
\item (bounded geometry) \label{op_a}
$\sR^m_k$ is a clopen  partition of $\fzn$, and each 
region $R$ in $\sR^m_k$ is a finite disjoint union of rectangles in $\tilde R$
with side lengths between $\frac{1}{2} d_k$ and $d_k$. 
\item (orthogonality to previous regions) \label{op_b}
Any face $F$ of a region in $\sR^m_k$ is at least $\epsilon d_k$ away from any parallel
face $F'$ of a nearby region in $\sR^j_k$ for any $k\leq j <m$.
\item (orthogonality to same regions) \label{op_c}
Any face $F$ of a region in $\sR^m_k$ is at least $\epsilon d_k$ away from any parallel
face $F'$ of a nearby region in $\sR^m_k$ unless $\rho(F,F')\leq 1$.
\setcounter{mycounter}{\value{enumi}}
\end{enumerate}
\end{defn}


As we said above, we set $\sR_m=\sR^m_1$. This completes the review of the construction
of the basic orthogonal marker regions, for a given set of distances 
$d_0<d_1\cdots$.  We note explicitly
for later purposes that for each dimension $n$, there is a fixed constant 
$K=K(n)$ such that if $d_{k+1}> K d_k$ for all $k$, then the orthogonal marker 
construction can be carried out.

It is not clear that the orthogonal marker regions that we constructed
above are connected and have connected boundaries. These connectedness
properties will be useful later, so we note below that they can be
arranged with a refinement or a variation of the above
construction.

A subset $S\subseteq \Z^n$ is \textit{connected} if it is connected in
the Cayley graph of $\Z^n$ with the standard generators. For any set
$S\subseteq \Z^n$ we define the \textit{boundary} of $S$, denoted by
$\partial S$, to be the set of points $x\in S$ such that $e\cdot
x\not\in S$ for one of the standard generators $e$. Since the regions
we consider will be finite unions of rectangles with reasonably big
side lengths, we will regard such $S\subseteq \Z^n$ as a subset of
$\R^n$ in the obvious way and speak of their connectedness, simple
connectedness, and their boundaries as subsets of $\R^{n-1}$. This
convention will make our discussion easier.

\begin{thm} \label{thm:om}
For some fixed $\epsilon=\epsilon_n>0$ there are clopen partitions
$\sR_m$ of $\fzn$ satisfying the orthogonality properties so that each
region $R$ of an $\sR_m$ partition is connected and also its boundary
$\partial R$ is connected. In fact each region $R$ is homeomorphic to
the closed unit ball in $\R^n$.
\end{thm}

The rest of this subsection is devoted to a proof of
Lemma~\ref{thm:om} for the dimension $n=2$ case. A different proof for
the general case will be given in the next subsection. We first give a
proof for the dimension $n=2$ case as the argument in this case
requires less modification of the construction, and the $n=2$ case
suffices for the applications we will have later in the paper.


We follow the general outline of the standard construction of
orthogonal marker regions but make several modifications. First, in
constructing $\sR^m_m$ we use Lemma~\ref{lem:n4c} and ensure
that no point is close to 4 rectangular marker regions. In the
subsequent steps, we will maintain this property at each step of the
construction.  Suppose we have defined $\sR^m_{k+1}$ and for each
region $R$ of $\sR^m_{k+1}$ we have that $R$ and $\partial R$ are
connected. In fact, we will inductively assume that the $\sR^m_{k+1}$
regions are homeomorphic to the closed unit disk via a homeomorphism
which maps the boundary to the unit circle (and so induces a homeomorphism between
the boundary and the unit circle). As in the standard orthogonal marker
regions construction, we define an auxiliary clopen rectangular
partition $\tilde \sR$ in the step of the construction of $\sR^m_{k}$
from $\sR^m_{k+1}$. In defining the $\tilde \sR$ partition,
we use scale $d_k$ and Lemma~\ref{lem:n4c} to get the partition $\tilde \sR$.
Thus, every point $x$ of $\fzs$ is close to at most $3$ distinct regions of
$\tilde \sR$ (i.e., no $4$ regions of $\tilde{\sR}$ come together).

The next step is to assign the regions $\tilde R$ of $\tilde \sR$
to the regions $R$ of $\sR^m_{k+1}$. Instead of doing the assignment arbitrarily as in the standard construction, we make the following modification. Call a region $\tilde R$ \textit{exceptional} if 
it intersects $3$ distinct regions of $\tilde \sR^m_{k+1}$. First, we assign the 
exceptional regions $\tilde R$ to regions $R$ of $\sR^m_{k+1}$ as follows. Let 
$R_1, R_2, R_3$ be the three regions of $\sR^m_{k+1}$ which intersect $\tilde R$. 
Label them so that $R_3$ is the region which contains an entire face of $\tilde R$
(see Figure~\ref{fig:twodconn}). Assign $\tilde R$ to $R_1$. Then, for the regions 
$\tilde R'$ in $\tilde \sR$ that are within $2d_k$ distance of the exceptional $\tilde R$, assign $\tilde{R}'$ to $R_1$ if $\tilde R'$ intersects $R_1$. 
Next, for those regions $\tilde R'$ in $\tilde \sR$ that are within $2 d_k$ distance to $\tilde{R}$, if they have not been assigned, assign them to $R_2$. Figure~\ref{fig:twodconn}
illustrates this assignment process near the exceptional rectangles 
of $\tilde R$.

\begin{figure}[h] 
\begin{tikzpicture}[scale=0.05]

\draw[thick] (0,50) to (0,0) to (100,0);
\draw[thick] (-1,50) to (-1,0) to (-100,0);
\draw[thick] (-100,-1) to (100,-1);

\draw[fill=red, opacity=0.5] (-10,-10) rectangle (8,8);
\draw[fill=red, opacity=0.5] (-15,8) rectangle (15,12);
\draw[fill=red, opacity=0.5] (-6,12) rectangle (9,18);
\draw[fill=red, opacity=0.5] (-10,18) rectangle (15,25);
\draw[fill=red, opacity=0.5] (-17,-5) rectangle (-10,8);
\draw[fill=red, opacity=0.5] (-30,-9) rectangle (-17,4);

\draw[fill=blue, opacity=0.5] (8,-4) rectangle (15,8);
\draw[fill=blue, opacity=0.5] (15,-8) rectangle (25,23);
\draw[fill=blue, opacity=0.5] (25,-3) rectangle (32,9);

\node at (-50,25) {$R_1$};
\node at (50,25) {$R_2$};
\node at (0,-20) {$R_3$};


\end{tikzpicture}
\caption{The assignment algorithm in dimension 2. The red rectangles are assigned to
$R_1$, and the blue ones to $R_2$.} \label{fig:twodconn} 
\end{figure}
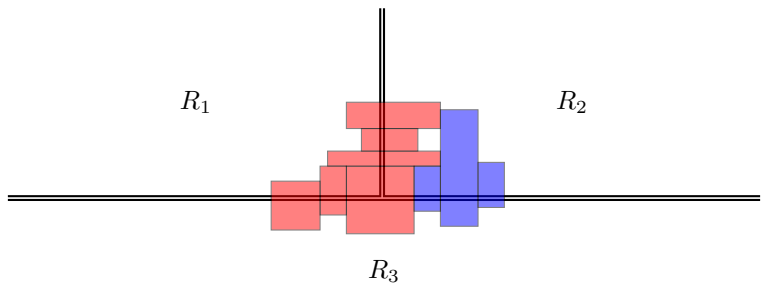

The rectangles $\tilde R'$ of $\tilde \sR$ which are not within $2d_k$
of an exceptional rectangle are assigned arbitrarily as before. The
regions of $\sR^m_k$ are obtained from this assignment as in the
standard construction. It is clear that the orthogonality properties
are satisfied as before. We only verify that the regions are homemorphic to
the unit disk (with boundaries mapping to the unit circle).

A region $R''$ of $\sR^m_k$ is obtained from a region $R$ of $\sR^m_{k+1}$
in two steps. First, for each exceptional rectangle $\tilde R$
which intersects $R$, we add or subtract from $R$ the rectangles of $\tilde \sR$
within $2d_k$ of $\tilde R$ according to the above algorithm (the region $R$
corresponds to one of the regions $R_1$, $R_2$, $R_3$ in the above description,
and which one may depend on $\tilde R$). Let $R'$ be the region obtained
from $R$ in this first step. In the second step we produce $R''$ from $R'$ by
adding or subtracting from $R'$ all of the rectangles $\tilde R$ in $\tilde \sR$
which intersect $R$ (or equivalently $R'$) not within $2d_k$ of an exceptional rectangle
intersecting $R$. We first show that $R'$ is homeomorphic to $R$ (and thus by induction
to the unit ball). Consider a single exceptional rectangle $\tilde R$ of $\tilde \sR$
which intersects $R$. With respect to $\tilde R$, that is,
in considereing the rectangles of $\tilde \sR$ within $2d_k$ of $\tilde R$,
$R$ may be considered as $R_1$, $R_2$, or $R_3$ in the algorithm.
If $R$ is considered as $R_1$, then all of the rectangles of
$\tilde \sR$ within $2d_k$ of $\tilde R$
which intersect $R$ are added to $R$ in forming $R'$. There is clearly a homeomophism
of this new region with $R$ which preserves boundaries, and is obtained by contracting
each of these rectangles in $\tilde \sR$ back to the boundary of $R$. These homeomorphism
can be taken to be the identity outside of $B(\tilde R,3d_k)$, the $3d_k$
neighborhood of $\tilde R$.
Suppose $R$ is considered as an $R_2$ region with respect to $\tilde R$. In this case
we obtain $R'$ within a $3d_k$ neighborhood of $\tilde R$ by
first removing the rectangles within $2d_k$ of $\tilde R$ which intersect
both $R_1$ and $R_2$, and then adding the rectangles which intersect $R_2$ but not $R_1$.
After removing the first set of rectangles the result is homeomorphic to $R_2$
by the same contraction argument used in the $R_1$ case. There is also a homeomorphism
which contracts the added rectangles back to the boundary of $R$. Both of these
homeomorphisms will be the identity outside of $B(\tilde R,3d_k)$.
Combining these two homeomorphisms gives a homeomorphism
between $R' \cap B(\tilde R,2d_k)$ with $R \cap B(\tilde R,2d_k)$ which is the
identity outside of $B(\tilde R,3d_k)$. In the third case where $R$ is considered as the
$R_3$ region, we remove all of the rectangles which intersect $R$ within $2d_k$ of $\tilde R$
in forming $R'$. Clearly a simple contraction of these rectangles to the boundary of $R$
gives the homeomorphism in this case as well.
Since the different exceptional rectangles which intersect $R$ are at least
$\epsilon d_{k+1}\gg 3d_k$ apart, we can combine these homeomorphisms
from these exceptional rectangles meeting $R$ to produce a (boundary preserving) homeomorphism
from $R'$ to $R$.

We next show $R''$ is homeomorphic to $R'$ by a boundary preserving homeomorphism.
For each edge $l$ of $R$, say $l$ is horizontal, let $M_l$ be the the set of rectangles
in $\tilde \sR$ which intersects $l$ and are not within $2d_k$ of
an exceptional rectangle $\tilde R$. Let $S_l$ be the rectangle  whose horizontal edges 
have distance $d_k+\epsilon$ from $l$, for some small $\epsilon>0$, and whose
left edge has $x$-coordinate $a-\epsilon$ where the leftmost rectangle $A$
in $M_l$ has $x$-coordinate $a$.
\begin{figure}[h] 
\begin{tikzpicture}[scale=0.05]

\draw[thick] (0,35) to (0,0) to (100,0);
\draw[thick] (-1,35) to (-1,0) to (-30,0);

\draw[fill=yellow, opacity=0.2] (-28,-30) rectangle (32,28);

\draw[fill=blue, opacity=0.5] (30,-4) rectangle (35,5);
\draw[fill=blue, opacity=0.5] (35,-7) rectangle (41,4);
\draw[fill=red, opacity=0.5] (41,-3) rectangle (45,6);
\draw[fill=blue, opacity=0.5] (45,-5) rectangle (50,5);
\draw[fill=red, opacity=0.5] (50,-6) rectangle (54,4);
\draw[fill=blue, opacity=0.5] (54,-3) rectangle (60,3);

\draw[thick] (80,8) to (29,8) to (29,-8) to (80,-8);

\node at (70,30) {$R$};

\node at (0, -20) {$B(\tilde R, 2d_k)$};

\node (A) at (15, 10) {$l$};
\draw[->]  (A) to (15,1);

\node (B) at (60, 20) {$S_l$};
\draw[->] (B) to (46,9);

\node (C) at (70, -20) {$M_l$};
\draw[->] (C) to (55, -4);


\end{tikzpicture}
\caption{The construction of $R''$ from $R'$. The blue rectangles are to be added and the red ones to be removed.} \label{fig:lem1.5f}
\end{figure}
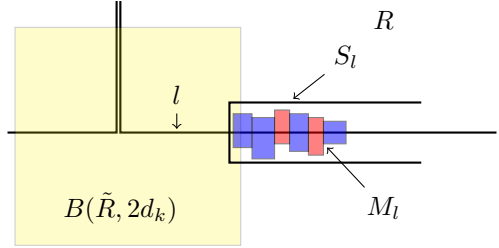
Likewise, the right edge of $S_l$
has $x$-coordinate $b+\epsilon$ where $b$ is the $x$-coordinate of the right edge
of the rightmost rectangle $B$ in $M_l$. If $l$ is vertical the rectangle $S_l$ is similarly defined. Figure~\ref{fig:lem1.5f} illustrates these definitions. Note that none of the
rectangles in $\tilde \sR$ which intersect a different edge $l'$ of $R$ will intersect
$S_l$. In fact, if $S_{l'}$ denotes the corresponding rectangle for the edge $l'$,
then $S_l\cap S_{l'}=\emptyset$ (if $\epsilon <\frac{1}{2} d_k$).
It is easy to see that there is homeomorphism of $S_l$ with itself which is the identity
on the boundary of $S_l$ and which maps $S_l \cap R'$ to $S_l \cap R''$.
Combining these homeomorphisms for the different edges of $R$ gives
a homeomorphism from $R'$ to $R''$.

This finishes a proof of Theorem~\ref{thm:om} for the $n=2$ case.

\subsection{Connected orthogonal marker regions} \label{comr}
In this subsection we give a proof of the general case of
Theorem~\ref{thm:om} that requires more substantial modifications of the
standard construction of orthogonal marker regions.

We still follow the general outline of the standard
construction. Specifically, we will construct an auxiliary clopen
rectangular partition $\tilde \sR$ in the construction of $\sR^m_k$
from $\sR^m_{k+1}$. However, in constructing $\tilde\sR$, we will use
$n+1$ many scale values $\epsilon_0 d'_k<\epsilon_1 d'_k<\cdots < \epsilon_{n} d'_k$
instead of the single value $d_k$ as before. At the end we will take $d_k=\epsilon_0 d'_k$.
By subdividing the rectangles in $\tilde \sR$
we will produce a clopen partition $\tilde S$ with side lengths between
$\frac{1}{2}d_k$ and $d_k$ and an assignment map which assigns to each $S\in \tilde S$
a region of $\sR^m_{k+1}$ so that the resulting regions define the
partition $\sR^m_k$ which satisfy Theorem~\ref{thm:om} for some $\epsilon >0$.

We first let $\sR$ be a clopen partition of
$\fzn$ into rectangles of side lengths between $\frac{1}{2} d'_k$ and
$d'_k$.  We may assume that for
some $\eta>0$ (which depends only on $n$) that (\ref{op_b}) and
(\ref{op_c}) of the orthogonality properties hold using the constant
$\eta$. Furthermore, we may assume that every face of a rectangle in
$\sR$ is also at least $\eta d'_k$ away from any parallel face of a
nearby rectangle in $\sR^m_{k+1}$.  Fix $\epsilon_0< \epsilon_1 <\cdots
<\epsilon_n$ with $\epsilon_n < \frac{1}{3}\eta$ and $\epsilon_j <
\frac{1}{12}\epsilon_{j+1}$ for all $j<n$.

Say a rectangle $R$ in $\sR$ 
is of \textit{type} $j$ (where $0\leq j \leq n$) if there are exactly $j$ distinct basic
unit vectors $e$ such that $R$ intersects a face of some $\sR^m_{k+1}$
region which is perpendicular to $e$. In particular, $R$ is of type 0 if it does not
intersect any face of an $\sR^m_{k+1}$ region, and it is of type $n$ if intersects
a corner of an $\sR^m_{k+1}$ region.


To construct the partition $\tilde\sR$, we further subdivide the
rectangles in $\sR$. We begin with the rectangles $R$ in $\sR$ of type
$n$. Note that $R$ is a rectangle of side lengths between $\frac{1}{2}d_k'$ and $d_k'$, whereas any region in $\sR^m_{k+1}$ is a disjoint union of finitely many rectangles each of which has side lengths at least $\frac{1}{2}d_{k+1}$, with $d_{k+1}\gg d_k'$. Also by our assumption, $R$ intersects $n$ many mutually perpendicular faces of a $\sR^m_{k+1}$ region. Since $\epsilon_n <\frac{1}{3}\eta\ll \frac{1}{2}$, we may subdivide $R$ into rectangles $S$ with side lengths between
$\frac{2}{3}\epsilon_n d'_k$ and $\epsilon_n d'_k$ such that the
perpendicular distance between any face of $S$ and a parallel face of an
$\sR^m_{k+1}$ region intersecting $R$ is at least $\frac{1}{4} \epsilon_nd'_k$. Let $\tilde
\sR_n$ be the resulting partition. Let $\sS_n$ be the collection of
all rectangles of type $n$ in $\tilde{\sR}_n$.  At the next step we
consider the rectangles $R$ in $\tilde{\sR}_n$ of type $n-1$. These
include both the rectangles in $\sR$ of type $n-1$ and the rectangles
$S$ in a subdivision from the first step which happen to be of type
$n-1$. We subdivide all of these rectangles into rectangles with side
lenghs between $\frac{2}{3} \epsilon_{n-1} d_k'$ and
$\epsilon_{n-1}d_k'$ in such a way that the distance from any of their
faces to a face of a region in $\sR^m_{k+1}$ is at least
$\frac{1}{4}\epsilon_{n-1}d_k'$. Let the resulting partition be $\tilde
\sR_{n-1}$.  Let $\sS_{n-1}$ be the collection of all rectangles in
these subdivisions which are still of type $n-1$.  Continuing in this
manner we define the partitions $\tilde \sR_n, \dots,
\tilde\sR_1=\tilde\sR$ each of which is a refinement of the previous
one in the sequence. We also obtain collections $\sS_n,\dots, \sS_1$,
and $\sS_j$ is the collection of all rectangles of type $j$ in the
partition $\tilde \sR$.  In summary, each rectangle in $\sS_j$ has
side lengths between $\frac{2}{3}\epsilon_j d_k'$ and $\epsilon_jd_k'$,
and any face of a rectangle in $\sS_j$ is at least $\frac{1}{4}
\epsilon_jd_k'$ from a parallel face in $\sR^m_{k+1}$. Also, each
rectangle in $\sS_j$ is of type $j$. 
We let $\tilde{\sR}$ be the
collection $\bigcup_j \sS_j$.  The standard arguments show that we may
assume that there is an $\epsilon'>0$ such that the orthogonality properties
\ref{op_b} and \ref{op_c} holds for the faces of regions in $\tilde \sR$ using
constant $\epsilon'$.

We define $\sR^m_{k}$ using $\tilde
\sR$ and an arbitrary assignment of the rectangles in $\tilde \sR$ to
regions of $\sR^m_{k+1}$ which intersect them.
Figure~\ref{fig:secondcon} illustrates this construction in dimension
2.

\begin{figure}[h] 
\begin{tikzpicture}[scale=0.05]

\draw[thick] (0,50) to (0,0) to (100,0);
\draw[thick] (-1,50) to (-1,0) to (-100,0);
\draw[thick] (-100,-1) to (100,-1);

\draw[fill=red, opacity=0.5] (-20,-30) rectangle (30,15);

\draw[fill=blue, opacity=0.5] (30,-4) rectangle (35,5);
\draw[fill=blue, opacity=0.5] (35,-7) rectangle (41,4);
\draw[fill=blue, opacity=0.5] (41,-3) rectangle (45,6);

\draw[fill=blue, opacity=0.5] (-4,15) rectangle (4,19);
\draw[fill=blue, opacity=0.5] (-7,19) rectangle (3,24);
\draw[fill=blue, opacity=0.5] (-3,24) rectangle (5,30);



\node at (-50,30) {$R_1$};
\node at (50,30) {$R_2$};
\node at (-50,-20) {$R_3$};

\node at (45, -20) {$S_2$};
\draw[->] (41,-20) to (31,-20);

\node at (60, 10) {$S_1$};
\draw[->] (55,9) to (46,5);


\end{tikzpicture}
\caption{The second construction in dimension 2.} \label{fig:secondcon}
\end{figure}
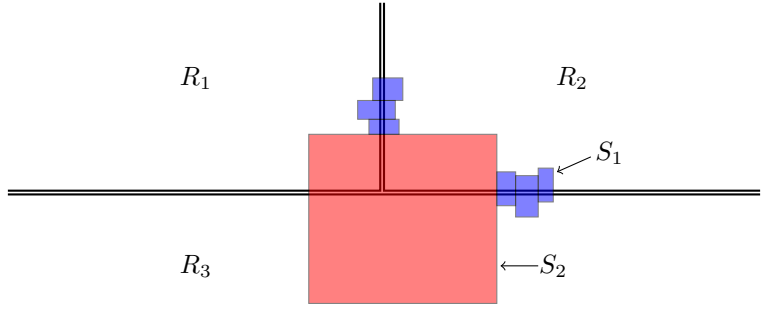
  
It is clear from our construction that the orthogonality properties hold for $\sR^m_k$.
We verify that each region of $\sR^m_k$ is homeomorphic to the closed unit ball
via a homemorphism which preserves the boundaries. This suffices, as we may
then take $d_k=\epsilon_0 d'_k$ and subdivide the rectangles in $\tilde \sR$
to have dimensions between $\frac{1}{2} d_k$ and $d_k$ and maintain the
orthogonality properties \ref{op_b}, \ref{op_c} for some constant $\epsilon>0$.

Fix a region $R$ of $\sR^m_{k+1}$. Let $A$ be the set of all rectangles in
$\tilde \sR$ which intersect $\partial R$. Let $A=B\cup C$ where
$B$ are those rectangles in $A$ which are assigned to $R$ and $C=A-B$.
Let $R'$ be the region of $\sR^m_k$ corresponding to $R$, that is,
$R'$ is the union of $R$ and all of the rectangles in $B$, minus all
of the rectangles in $C$. For $0 \leq j \leq n+1$, let $B_j$, $C_j$
be those rectangles in $B$, $C$ respectively of type $\geq j$. Let $A_j=B_j\cup C_j$.
Let $R_j=( R \cup (\cup B_j))- (\cup C_j)$. Thus, $R'=R_0$ and
$R=R_{n+1}$. We show for each $0 \leq j \leq n$ that $R_j$
is homeomorphic to $R_{j+1}$ by a boundary preserving homeomorphism.
Since $R_{n+1}=R$ is homeomorphic to the ball by induction, this suffices. 
Let $T_{j+1}=\cup A_{j+1}$. 

Consider a $j$-edge $E$ of $\partial R$. By a {\em $j$-edge} we mean a connected component of 
the intersection  of $j$ mutually perpendicular faces of $R$. 
Say the faces determining $E$ have normal vectors
$e_{k_1},\dots, e_{k_j}$. For $i\in \{k_1,\dots, k_j\}$, let $c_i$
be the common value of $\pi_i(x)$ for $x \in E$.
So, $E$ is of the form $E=I_1\times I_2 \times \cdots \times I_n$ where
$I_i$ is a proper interval for $i \notin \{ k_1,\dots,k_j\}$,
and for $i \in \{ k_1,\dots, k_j\}$ we have $I_i=\{ c_i\}$.
Let $\pi_E$ denote the projection map to $E$, so
$\pi_E(t_1,\dots,t_n)=(s_1,\dots,s_n)$ where $s_i=t_i$ if
$i \notin \{ k_1,\dots,k_j\}$, and $s_i=c_i$ otherwise. 

Fix $\epsilon >0$ with $\epsilon_j< \epsilon < \frac{1}{4} \epsilon_{j+1}$. 
Let $E(\epsilon)= J_1 \times J_2\times \cdots \times J_n$ where
$J_i=I_i$ for $i\notin \{ k_1,\dots, k_j\}$, and $J_i=[c_i-\epsilon d_k',
c_i+\epsilon d_k']$ for $i\in \{ k_1,\dots,k_j\}$. Roughly speaking,
$E(\epsilon)$ is the $\epsilon d_k'$-expansion of $E$ in the directions
$e_{k_1},\dots,e_{k_j}$.

Let $E'(\epsilon)= E(\epsilon)\sm T_{j+1}$. Let $F=E\sm T_{j+1}$.
Let 
$F(\epsilon)$ be the set of all $x=(t_1,\dots,t_n)$ such that $\pi_E(x)\in F$ and 
$t_i \in [c_i-\epsilon d_k', c_i+\epsilon d_k']$  for $i \in \{ k_1,\dots,k_j\}$.
So, $F(\epsilon)$ is the $\epsilon d_k'$-expansion of $F$ in
the directions $e_{k_1},\dots,e_{k_j}$.
For each $x\in E$, let $x(\epsilon)$ denote the set of $(t_1,\dots, t_n)$ such that $\pi_E(t_1,\dots, t_n)=x$ and $t_i\in [c_i-\epsilon d_k', c_i+\epsilon d_k']$ for $i\in \{k_1,\dots, k_j\}$.

\begin{claim} \label{cpa}
$E'({\epsilon})=F(\epsilon)$.
\end{claim}

\begin{proof}
It is clear that $F(\epsilon)\subseteq E(\epsilon)$. We show that for every point $x \in E$, if $x \in F$ then
$x(\epsilon) \subseteq E'(\epsilon)$, and if
$x \notin F$ then $x(\epsilon)
\subseteq T_{j+1}$. The claim then follows.

To see this, consider $x\in E$. First suppose $x \in T_{j+1}$.
Then $x$ lies both in $E$ and a rectangle $Q$ in $A_{j+1}$ which intersects $E$.
Since all faces determining $E$ are at least $\frac{1}{4}\epsilon_{j+1} d'_k>\epsilon d_k'$ from any parallel
faces of $Q$, it follows that $x(\epsilon) \subseteq Q$. Thus if $x\notin F$ we have $x(\epsilon)\subseteq T_{j+1}$.
Suppose now that $x \in F$. Suppose toward a contradiction that
$y \in x(\epsilon)\cap T_{j+1}$ (that is, this set is 
non-empty). Let $L_{k_1}$, $L_{k_2},\dots, L_{k_j}$ be the natural sequence of line segments
of length $\leq \epsilon d_k'$ which together form a path $p=(u_0=x, u_{k_1},\dots, u_{k_j}=y)$
from $x$ to $y$. Here
$L_{k_r}$ only involves changing the $k_r$ coordinate. More precisely,
$L_{k_r}$ is the set of points $z$ with $\pi_i(z)=\pi_i(u_{k_{r-1}})$
for $i \neq k_r$, and $\pi_{k_r}(z)$ is between $\pi_{k_r}(u_{k_{r-1}})=\pi_{k_r}(x)=c_{k_r}$ and $\pi_{k_r}(u_{k_r})=\pi_{k_r}(y)$.
Let $r$ be minimal such that $u_{k_r}\in T_{j+1}$. Say $u_{k_r}\in Q$,
where $Q$ is a rectangle in $A_{j+1}$. Note that $u_{k_{r-1}}$ is on the boundary of
$R$, in fact it lies on a face of $\partial R$ with normal vector $e_{k_r}$.
However, $u_{k_r}$ lies on a face of $Q$ which is perpendicular to $e_{k_r}$.
This contradicts the fact that any face of $Q$ is at least
$\frac{1}{4}\epsilon_{j+1} d_{k}'>\epsilon d_k'$ from a parallel face of $R$. 
\end{proof}

Figure~\ref{fig:Ep} illustrates the set $E'(\epsilon)$ in $3$-dimensions
where $E$ is the $j=1$ ``edge'' corresponding to the $x,y$ plane. 
The rectangles in green form
the set $T_2$ (those in $T_1\sm T_2$ are not shown). The region in blue is
$E'(\epsilon)$ (only the top half of it is shown).

\begin{figure} 
\includegraphics[width=3in, height=2in]{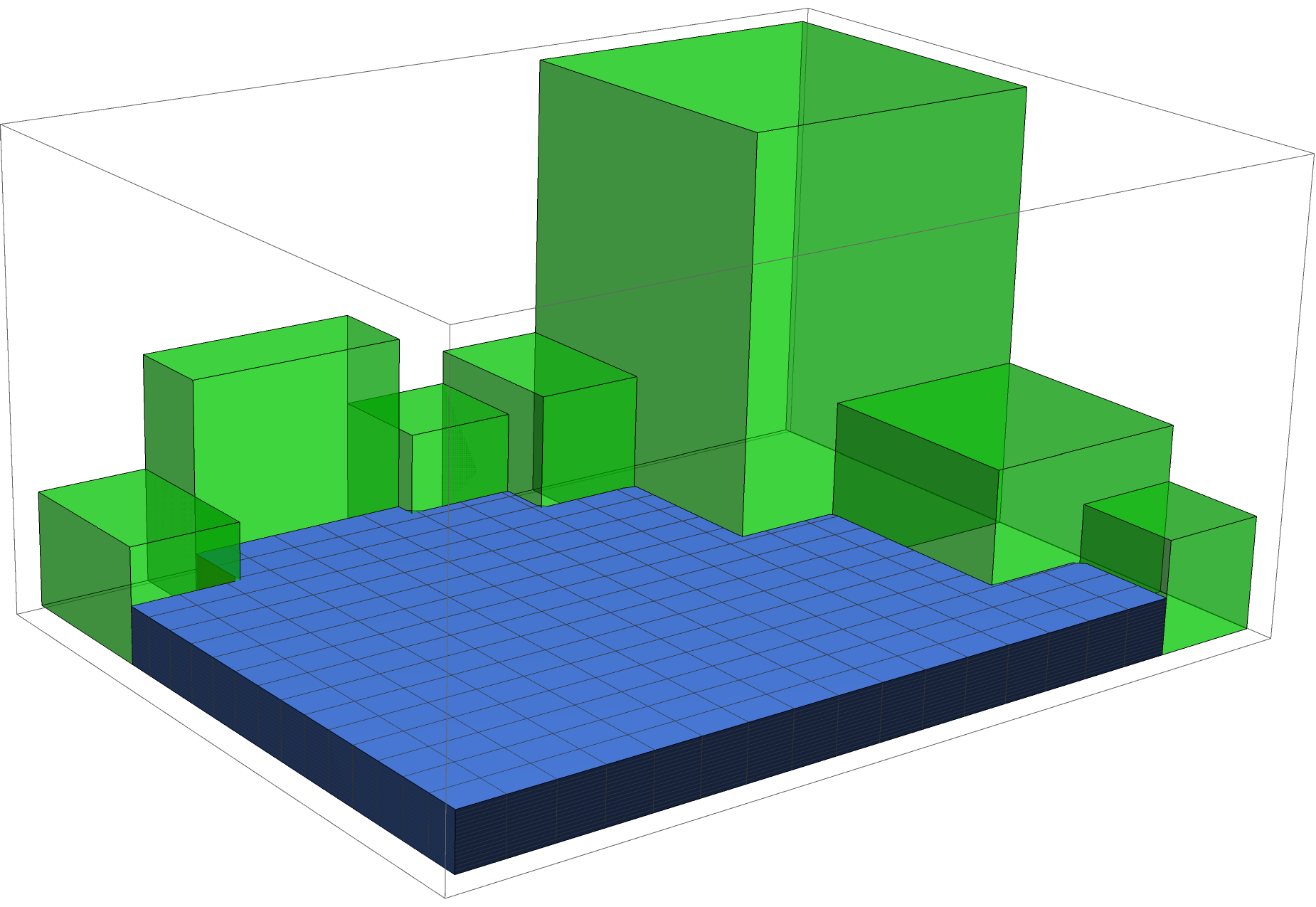}
\caption{The set $E'(\epsilon)$. } \label{fig:Ep}
\end{figure}

Note that by Claim~\ref{cpa}, $E'(\epsilon) \cap R_{j+1}$ is the set of points $x$
such that $\pi_E(x)\in F$ and for $i \in\{ k_1,\dots,k_j\}$,
$\pi_i(x) \in I'_i$ where $I'_i$ is an interval of the form
$[c_i,c_i+\epsilon d_k']$ or of the form $[c_i-\epsilon d_k',c_i]$.
Now, for each such $j$-edge $E$ of $R$ and corresponding 
$E'(\epsilon)$ there is a boundary preserving homeomorphism between
$E'(\epsilon) \cap R_{j}$ and $E'(\epsilon) \cap R_{j+1}$.
This is because, by Claim~\ref{cpa},
$E'(\epsilon) \cap R_j$ is obtained from $E'(\epsilon) \cap R_{j+1}$
by adding the rectangles of type $j$ of $B_j$ which intersect $F$,
and deleting those of $C_j$ which intersect $F$. Each of these rectangles $T$ is of the form
$(T\cap F) \times C$ where $C$ is a $j$-dimensional rectangle with side lengths $<\epsilon d_k'$.
The homeomorphism is induced by continuously contracting the edge lengths
of $C$ to $0$.

For each $j$-edge $E$ of $R$, let $\varphi_E$ be a boundary preserving
homeomorphism of $E'(\epsilon) \cap R_{j}$ with $E'(\epsilon) \cap R_{j+1}$.
we can take  $\varphi_E$ to be the identity on the boundary of $E'(\epsilon)$. 
To finish the proof of Theorem~\ref{thm:om} it suffices to show the following claim, which
shows that we can take the union of the $\varphi_E$ (and the identity on
$R_j\sm \bigcup_E E'(\epsilon)$, where the union ranges over all the $j$-edges $E$
of $R$).

\begin{claim}\label{omclaim}
Let $E_1$, $E_2$ be distinct $j$-edges of $R$. Then the
corresponding $E'_1(\epsilon)$, $E'_2(\epsilon)$
are disjoint.
\end{claim}

\begin{proof}
Let $z \in E'_1(\epsilon)\cap E'_2(\epsilon)$.
Say $E_1$ is defined by faces with normal vectors
$\{ e_{k_1},\dots, e_{k_j}\}$ and values $c_i$ for $i \in \{ k_1,\dots,k_j\}$
as above. Let $D_1=\{ k_1,\dots,k_j\}$ (so for $x \in E_1$,
$\pi_i(x)=c_i$ for $i \in D_1$). Similarly, let $E_2$ be defined
by $\{ e_{\ell_1},\dots,e_{\ell_j}\}$ and values $c'_i$.
Let $D_2=\{ \ell_1,\dots,\ell_j\}$. Since $E_1\neq E_2$, and
$E_1$, $E_2$ are both within distance $\epsilon d_k'$ of $z$, and hence they are of distance $2\epsilon d_k'$ of each other, we must have that
$D_1\neq D_2$.

Let $z_1=\pi_{E_1}(z) \in E_1$, so $z_1$ is  obtained from $z$ by changing all of the values $\pi_i(z)$, for $i \in D_1$, to the value $c_i$. Similarly, let $z_2=\pi_{E_2}(z)$. 
Let $w$ be such that $\pi_i(w)=\pi_i(z)$ for $i\notin D_1\cup D_2$, $\pi_i(w)=c_i$ for $i\in D_1$ and $\pi_i(w)=c'_i$ for $i\in D_2-D_1$. That is, $w$ is the result of changing the values $\pi_i(z)$ for $i \in D_1 \cup D_2$.
Since $D_1\neq D_2$, $w$ lies on a $j'$-edge of $R$ for $j'\geq j+1$.
Also $w$ and $z$ are within $\epsilon d_k'$ of each other. This is a contradiction as $w$ lies in
a rectangle $Q\in T_{j+1}$, and all faces of $Q$ which are perpendicular to
a direction in $D_1 \cup D_2$ have distance at least $\frac{1}{4}\epsilon_{j+1}d_k'>\epsilon d_k'$ from $w$. 
\end{proof}

\section{Orthogonal Decompositions in $F(2^{\Z^n})$} \label{sec:odzn}

In \S\ref{sec:omzn} we recalled the construction of the orthogonal marker regions $\sR_k=\sR^k_1$  for $F(2^{\Z^n})$
corresponding to a set of marker distances $d_1 < d_2 <\cdots$ (which satisfy $d_{k+1}/d_k \geq K$
for some constant $K=K(n)$ depending only on the dimension $n$).
A slightly different point of view results in what
we call an {\em orthogonal decomposition} for $F(2^{\Z^n})$. In \S\ref{sec:genod} we will present a
more general form of the definition in the context of a Borel graphing of a countable Borel equivalence relation.
The basic idea is to consider the sets $X_k= \partial \sR_k= \bigcup\{\partial R: R\in \sR_k\}$, the set of boundary
points for the regions in the Borel partition $\sR_k$. These sets are not pairwise disjoint, but they intersect in
a controlled fashion which we make precise below. We record here the strongest forms of these intersection properties,
while in \S\ref{sec:genod} we abstract a weaker version on these properties which nonetheless suffices
for some applications (e.g., the existence of toast structures and the Borel proper $3$-coloring of the Schreier graph on $F(2^{\Z^n})$). 

For each $k$ we consider all the partitions $\sR_m$ for $m\geq k$ and let $\sA_k$ be the coarsest common refinement of all of them. We call the elements of $\sA_k$ the {\em $k$-atoms} of the orthogonal decomposition. In contrast, we write $X_k^\infty$ for $\bigcup_{m=k}^\infty X_m=\bigcup_{m=k}^\infty \partial \sR_m$ and let $\sR(k)$ denote the collection of connected components of $X-X_k^\infty$ (where $X=F(2^{\Z^n})$).
We call these connected components of  $X-X_k^\infty$ the {\em $k$-nuclei} of the orthogonal
decomposition. In \S\ref{sec:topreg} we will show that we may maintain certain topological regularity for the
$k$-atoms and the $k$-nuclei, as shown possible in Theorem~\ref{thm:om} for the marker regions $R\in \sR_k$ themselves. This will guarantee that
the $k$-atoms and $k$-nuclei are in one-one correspondence, and that each $k$-nucleus is exactly the interior of a $k$-atom.
In the rest of this section we describe the basic orthogonality properties of the $k$-nuclei.


The following theorem summarizes the properties of the orthogonal decomposition.

\begin{thm} \label{thm:so}
Let $\sR_m$ be a sequence of marker regions for $X=F(2^{\Z^n})$ satisfying the orthogonal properties
for distances $d_1<d_2<\cdots$ as contructed in the proof of Theorem~\ref{thm:om}. Let $X_k=\partial \sR_k$,
and $X_k^\infty=\bigcup_{m=k}^\infty X_m$. Then the following hold.

\begin{enumerate}
\item \label{so1}
$\bigcap_k X_k^\infty =\emptyset$, and for each $k$ the connected components of $X-X_k^\infty$
are finite.
\item (strong orthogonality) \label{so2}
There is a constant $\epsilon=\epsilon(n)>0$ and a function $\delta\colon X\times \N\times \N \to n$ such that $\delta(x,k,\ell)=0$ unless
$x \in X_k^\infty$ and $\ell \leq k$, and satisfying:
if $k< k'$, $x\in X_k$, $x' \in X_{k'}$, then for any $\ell \leq k$ if we have $\delta(x,k,\ell)=\delta(x',k',\ell)$,
then $\rho(x,x') > \epsilon d_\ell$. 
\end{enumerate}


\end{thm}

\begin{proof}
Each connected component of $X-X_k^\infty$ is contained in a connected
component of $X-X_k$, which is contained in a set of the form $R-\partial R$
for some $R\in \sR_k=\sR^k_1$. Since each $R\in \sR^k_1$ is finite, (\ref{so1}) holds.

For (\ref{so2}), let $\delta$ be the function defined as follows. If $x \notin
X_k$, set $\delta(x,k,\ell)=0$ for all $\ell$. Suppose now $x \in X_k$ and $\ell \leq k$.
Let $B^k_\ell =\bigcup \{ \partial R': R'\in \sR^k_\ell\}$. 
So, $B^k_\ell$ is the union of the $\ell$th level boundaries of the regions used in the
construction of $\sR_k$. Let $y \in B^k_\ell$ be the point of $B^k_\ell$ closest
to $x$ in the $\rho$-metric on $[x]$ (with ties broken in an arbitrary Borel manner).
Let $\tilde{\sR}^k_\ell$ be the set of auxiliary rectangles used in defining $\sR^k_\ell$ from
$\sR^k_{\ell+1}$.
Then for some $\tilde R \in \tilde{\sR}^k_\ell$ we have $y \in \partial \tilde R$.
We set $\delta(x,k,\ell)=i$ if $y$ is on a face of $\tilde R$ which is perpendicular
to $e_i$ (if there is more than one such $i$, take the least one). 

We note that there is a constant $N=N(n)$ such that for the points $x$ and $y$ in the above definition, $\rho(x, y)\leq N(d_1+\cdots+d_{\ell-1})$. This is because, in the construction of $\sR^k_j$ from $\sR^k_{j+1}$, 
we started with auxiliary rectangles of side lengths as large as $\epsilon_nd_j'$, 
while $d_j=\epsilon_0d_j'$. Thus, any point of $B^k_j$ has distance at most $Nd_j$ to $B^k_{j+1}$,
where $N=\frac{\epsilon_n}{\epsilon_0}$.

Suppose now $k<k'$, $x \in X_k$, $x' \in X_{k}$, and for some $\ell<k$ we have that
$\delta(x,k,\ell)=\delta(x',k',\ell)$. Let $y,y'$ be the points defined as above corresponding to $x$ and $x'$, respectively.
Now we have $\rho(x,y), \rho(x',y') \leq N(d_1+\cdots + d_{\ell-1})$. Because
$\delta(x,k,\ell)=\delta(x',k',\ell)$, $y$ and $y'$ are on parallel faces of rectangles
$R$ and $R'$, where $R\in \tilde{\sR}^k_\ell$ and $R'\in \tilde{\sR}^{k'}_\ell$.
By construction, every face of a rectangle in $\tilde{\sR}^{k'}_\ell$ is at least
$\epsilon d_\ell$ from a parallel face of any rectangle in $\tilde{\sR}^k_\ell$,
where $\epsilon>0$ is a constant depending only on $n$. So,
$$\rho(x,x') > \epsilon d_\ell - 2N(d_1+\cdots +d_{\ell-1}) > \frac{\epsilon}{2} d_\ell$$
provided $\frac{2N}{K-1} <\frac{\epsilon}{2}$ (recall $K$ is a constant so that $d_{j+1}>Kd_j$
for all $j$). This establishes (\ref{so2}), with constant $\frac{\epsilon}{2}$. 
\end{proof}


We refer to sets $\{ X_k\}_{k \in \N}$ satisfying (\ref{so1}) and (\ref{so2})
of Theorem~\ref{thm:so} as a \emph{strong orthogonal decomposition}, and the constant $\epsilon>0$ in (\ref{so2})  as the {\em orthogonality constant} for the decomposition. 

The following corollary is a consequence of Theorem~\ref{thm:so}
which we refer to as \emph{bounded geometry weak orthogonality}
(we will generalize and formalize this definition in the next section).

\begin{cor} \label{cor:wo}
Let $\{ X_k \}_{k \in \N}$ be a strong orthogonal decomposition for $F(2^{\Z^n})$ with orthogonality constant $\epsilon>0$. Then the following hold.

\begin{enumerate}

\item \label{wo1}
For every $k \geq 1$ and every
connected component $R$ of $X - X_k^\infty$ we have for all $\ell \geq k$ that 
\[
\rho(R, X_\ell) \leq \frac{\epsilon}{2} d_1 \Longrightarrow \rho(R, X_\ell) \leq 1.
\]

\item \label{wo2}
For every $k \geq 1$ and every connected component $R$ of $X - X_k^\infty$ we have
\[
|\{\ell \geq k \colon \rho(R, X_\ell) = 1\}| \leq n+1.
\]
\end{enumerate}

\end{cor}

\begin{proof}
Let $R$ be a connected component of $X - X_k^\infty$.
We first show (\ref{wo2}). Let $\ell_1<\ell_2<\cdots <\ell_m$
enumerate the $\ell > k$ such that $\rho(X_\ell, R)=1$. 
Suppose $m>n$. For each $1 \leq j \leq m$ let $x_j \in X_{\ell_j}$
be such that $\rho(x_j, R)=1$. There must be $j_1\neq j_2$ such that
$\delta(x_{j_1}, \ell_{j_1}, k+1)= \delta(x_{j_2},\ell_{j_2},k+1)$. However,
we then have that $\rho(x_{j_1}, x_{j_2}) >\epsilon d_{k+1}
> 2 d_k$ (we assume $K$ is large enough so that $\epsilon d_{i+1} >2d_i$,
that is, $K>\frac{2}{\epsilon}$). However, since
$\rho(x_{j_1},R)=1$, $\rho(x_{j_2},R)=1$, we have that $\rho(x_{j_1},x_{j_2})
\leq d_k+2$, a contradiction.

Next we show (\ref{wo1}). Let $R$ be a connected component of $X-X_k^\infty$,
let $x \in R$, and let $y \in X_\ell$ with $\ell \geq k$ and
$\rho(x,y)\leq \frac{\epsilon}{2} d_1$. Consider the path $p$ from
$x$ to $y$ such that $\rho(z,y) \leq \rho(x,y)$ for all $z \in p$.
Let $x'\in p$ be the first point on the path not in $R$. 
Say $x' \in X_{\ell'}$. If $\ell=\ell'$, then
let $x''$ be the point of $p$ immediately before $x'$, and we have that $x''\in R$ and
$\rho(x'',X_\ell)=1$.

So, assume $\ell'\neq \ell$.
We thus have points $x' \in X_{\ell'}$, $y \in X_\ell$ with $\rho(x', R)=1$ and $\rho(x',y)
\leq \frac{\epsilon}{2}d_1$.  
Recall $e_1,\dots, e_n$ are the coordinate vectors in $\Z^n$. 
Let $x_0=x'$, $y_0=y$, and assume $(\alpha_1e_1+\cdots+\alpha_ne_n)\cdot x_0=y_0$ for $\alpha_1, \dots, \alpha_n\in\Z$. We claim that either $\rho(R, X_\ell)=1$ (and thus the corollary is proved) or there are $\beta_1, \gamma_1\in\Z$ with $\alpha_1\beta_1\geq 0$, $\beta_1\gamma_1\leq  0$, and $\alpha_1=\beta_1-\gamma_1$ such that, letting $x_1=(\beta_1e_1)\cdot x_0$ and $y_1=(\gamma_1e_1)\cdot y_0$, we have $x_1\in X_{\ell'}$, $y_1\in X_{\ell}$, $\rho(x_1, R)=1$, and $\rho(x_1,y_1)\leq \frac{\epsilon}{2}d_1$.

If $\alpha_1=0$ then let $\beta_1=\gamma_1=0$ and we have $x_1=x_0$ and $y_1=y_0$. Suppose $\alpha_1\neq 0$. Let $\beta\geq 0$ be the largest such that $\beta\leq |\alpha_1|$ and $\mbox{sgn}(\alpha_1)\beta e_1\cdot x_0$ is an element of $X_{\ell'}$ and is of distance $1$ to $R$. Also let $\gamma\geq 0$ be the largest such that $\gamma\leq |\alpha_1|$ and $-\mbox{sgn}(\alpha_1)\gamma e_1\cdot y_0$ is an element of $X_{\ell}$. If $\beta+\gamma\geq |\alpha_1|$, then the claim holds by letting $\beta_1=\mbox{sgn}(\alpha_1)\beta$ and $\gamma_1=\beta_1-\alpha_1$. Suppose $\beta+\gamma< |\alpha_1|$. Then $-\mbox{sgn}(\alpha_1)\gamma e_1\cdot y_0$ is on a face $S$ of $X_{\ell}$ that is perpendicular to $e_1$. On the other hand, we have two possible cases for the reason $\beta$ is the largest with $\mbox{sgn}(\alpha_1)\beta e_1\cdot x_0$ being an element of $X_{\ell'}$ and of distance 1 to $R$. Case 1 is  $\mbox{sgn}(\alpha_1)\beta e_1\cdot x_0$ is on a face $T$ of $X_{\ell'}$ that is perpendicular to $e_1$. In this case, $S$ and $T$ are parallel, and $\rho(S, T)\leq \frac{\epsilon}{2}d_1$, contradicting the strong orthogonality condition of Theorem~\ref{thm:so}. Case 2 is $\mbox{sgn}(\alpha_1)\beta e_1\cdot x_0$ is of distance 1 to an element $x''\in R$, and $x''$ is of distance $1$ to a face $T$ of some $X_{\ell''}$ which is perpendicular to $e_1$. In this case, if $\ell''=\ell$, we have $\rho(x'', X_\ell)=1$ and the corollary is proved; otherwise $\ell''\neq \ell$ and again $S$ and $T$ are parallel faces of distance $\leq \frac{\epsilon}{2}d_1$, contradicting the strong orthogonality condition of Theorem~\ref{thm:so}.

Applying the claim, we either have verified the corollary, or we will get $x_1\in X_{\ell'}$, $y_1\in X_{\ell}$ with $\rho(x_1, R)=1$ and $\rho(x_1, y_1)\leq \frac{\epsilon}{2}{d_1}$. Now $(\alpha_2e_2+\dots+\alpha_ne_n)\cdot x_1=y_1$. By similar arguments, we obtain $x_2, \dots, x_n$ and $y_2, \dots, y_n$ (unless we finish the proof of the corollary early) where $x_n=y_n\in X_{\ell'}\cap X_{\ell}$. Since $\rho(x_n, R)=1$, we have $\rho(R, X_\ell)=1$ as promised.
\end{proof}



\section{Weakly Orthogonal Decompositions} \label{sec:genod}
In this section we present some applications of the weakly orthogonal decompositions. One of the applications is to the computation of Borel chromatic numbers. Our results will imply that for all $n\geq 2$, the Schreier graph of $F(2^{\Z^n})$ has Borel chromatic number $3$. This is in contrast with the result in \cite{GJKSinf} that their continuous chromatic numbers are $4$. These are the first known results where the Borel characteristics and the continuous characteristics of countabe group actions behave differently. Another application is to the existence of the so-called Borel unlayered toast structures on $F(2^{\mathbb{Z}^n})$. The significance of this notion is that it implies hyperfiniteness for the orbit equivalence relation. In contrast with other known results, we have shown in \cite{GJKSinf} that continuous toast structures do not exist, and in \cite{GJKS2022} that Borel layered toast structures do not exist.

Fix a standard Borel space $X$ and a Borel graph $\Gf \subseteq X \times X$ on $X$. Let $\pathd_\Gf$ denote the shortest-path-length metric on the connected components of $\Gf$.

\begin{defn} \label{defn:bndgeo}
A \emph{weakly orthogonal decomposition of $X$} is a collection $(X_k)_{k \geq 1}$ of Borel subsets of $X$ satisfying the following two conditions. Below we write $X_k^\infty$ for $\bigcup_{n = k}^\infty X_n$.
\begin{enumerate}
\item[\rm (i)] (decomposition) Every connected component of $\Gf \res (X \setminus X_k^\infty)$ is finite, and $\bigcap_{k \geq 1} X_k^\infty = \varnothing$.
\item[\rm (ii)] (weakly orthogonal) There exists a constant $Q$ such that for every $k \geq 1$ and every connected component $R$ of $\Gf \res (X \setminus X_k^\infty)$ we have for all $n \geq k$
$$\pathd_\Gf(R, X_n^\infty) \leq Q \Longrightarrow \pathd_\Gf(R, X_n^\infty) = 1.$$
\end{enumerate}
We furthermore say that $(X_k)_{k \geq 1}$ has \emph{bounded geometry} if in addition to the above properties it satisfies the following.
\begin{enumerate}
\item[\rm (iii)] (bounded geometry) There exists a constant $P$ such that for every $k \geq 1$ and every connected component $R$ of $\Gf \res (X \setminus X_k^\infty)$ we have
$$|\{n \geq k : \pathd_\Gf(R, X_n) = 1\}| \leq P.$$
\end{enumerate}
We call the smallest $P$ satisfying (iii) the \emph{polygonal bound} for $(X_k)_{k \geq 1}$, and we call the smallest $Q$ satisfying (ii) the \emph{orthogonality constant} for $(X_k)_{k \geq 1}$.
\end{defn}

As we have shown in Corollary~\ref{cor:wo}, in $F(2^{\mathbb{Z}^n})$ bounded geomtry weakly orthogonal decompositions exist with polygonal bound $P=n+1$ and orthogonality constant $Q=\frac{\epsilon}{2}d_1$. Here the constant $\epsilon$ depends on $n$ but $d_1$ does not. In particular, we can pick $d_1$ to be arbitrarily large.

We observe that the existence of a weakly orthogonal decomposition of $X$ necessitates hyperfiniteness.

\begin{lem}\label{lem:hyp}
Let $X$ be a standard Borel space and let $\Gf \subseteq X \times X$ be a Borel graph on $X$. Let $E_\Gf$ be the equivalence relation on $X$ given by the connected components of $\Gf$. If there exists a weakly orthogonal decomposition of $X$ then $E_\Gf$ is hyperfinite.
\end{lem}

\begin{proof}
Fix a weakly orthogonal decomposition $(X_k)_{k \geq 1}$ of $X$. For $k \geq 1$ let $E_k$ be the equivalence relation which is trivial (consists of singletons) when restricted to $X_k^\infty$ and which when restricted to $X \setminus X_k^\infty$ is given by the connected components of $\Gf \res (X \setminus X_k^\infty)$. The sequence $(E_k)_{k \geq 1}$ witnesses the hyperfiniteness of $E_\Gf$.
\end{proof}

We also observe that if a Borel graph $\Gamma$ admits a weakly orthogonal decomposition with bounded
geometry, then necessarily $\Gamma$ is locally finite. For if $x \in X$, then by bounded geometry, there are only finitely many 
$n$ such that $\rho(x,X_n)=1$. Let $m$ be large enough so that if $\rho(x,X_n)=1$, then $n <m$, and $x \notin X^\infty_m$. But then 
$X\sm X_m^\infty$ contains $x$ and all the neighbors of $x$, and so the connected component of $X\sm X_m^\infty$ 
which contains $x$ contains all of its neighbors. Since this component is finite, $x$ has only finitely many neighbors in $\Gamma$.

For the remainder of this section we fix a bounded geometry weakly orthogonal decomposition $(X_k)_{k \geq 1}$ of $X$ with polygonal bound $P$ and orthogonality constant $Q$. We set $X_k^\infty = \bigcup_{n = k}^\infty X_n$.

\begin{defn}
A set $R \subseteq X$ is a \emph{$k$-nucleus} if $R$ is a connected component of $\Gf \res (X \setminus X_k^\infty)$. We say that $R$ is an \emph{nucleus} if $R$ is a $k$-nucleus for some $k \geq 1$.
\end{defn}

\begin{defn}
Let $R$ be a nucleus.
\begin{itemize}
\item We define the \emph{stage} of $R$, denoted $s(R)$, to be the maximum $k$ such that $R$ is a $k$-nucleus.
\item For $m \geq 1$, we define the $m$-boundary index of $R$ to be
$$B_R(m) = \{n \geq m : \pathd_\Gf(R, X_n) = 1\}.$$
\item We define the \emph{amplitude} of $R$ to be $a(R) = \max B_R(s(R))$. 
\item For $d \geq 1$, the \emph{$d$-fundamental interior} of $R$ is
$$f_d(R) = \{r \in R : \forall n \geq s(R) \ \ \pathd_\Gf(r, X_n^\infty) \geq P \cdot d - |B_R(n + 1)| \cdot d\}.$$
\end{itemize}
\end{defn}

\begin{lem} \label{lem:atoms}
Let $R$ and $R'$ be nuclei.
\begin{enumerate}
\item[\rm (i)] If $R \cap R' \neq \varnothing$ and $s(R) \leq s(R')$ then $R \subseteq R'$.
\item[\rm (ii)] If $R \subseteq R'$ and $R \neq R'$ then $s(R) < s(R')$.
\item[\rm (iii)] If $R \subseteq R'$ then $B_R(n) \subseteq B_{R'}(n)$ for all $n \geq s(R')$.
\item[\rm (iv)] $|B_R(n + 1)| \leq P - 1$ for all $n \geq s(R)$.
\item[\rm (v)] If $R \subseteq R'$ then $a(R) \leq a(R')$.
\end{enumerate}
\end{lem}

\begin{proof}
(i). By definition $R$ is connected in $\Gf \res (X \setminus X_{s(R)}^\infty)$ and therefore connected in the larger graph $\Gf \res (X \setminus X_{s(R')}^\infty)$. The set $R'$ is a connected component of the latter graph, so as $R \cap R' \neq \varnothing$ we must have $R \subseteq R'$.

(ii). For every $k \geq s(R')$, $R$ is strictly contained in $R'$ which is connected in the graph $\Gf \res (X \setminus X_k^\infty)$. So $R$ is not a connected component of $\Gf \res (X \setminus X_k^\infty)$. Therefore $s(R) < s(R')$.

(iii). Fix $n \geq s(R')$. Suppose that $m \in B_R(n)$. Then $\pathd_\Gf(R, X_m) = 1$. Since $R'$ is a connected component of $\Gf \res (X \setminus X_{s(R')}^\infty)$ and $m \geq n \geq s(R')$, we must have $R' \cap X_m = \varnothing$. Using $R \subseteq R'$ we obtain $0 < \pathd_\Gf(R', X_m) \leq \pathd_\Gf(R, X_m) = 1$. Therefore $m \in B_{R'}(n)$.

(iv). This follows immediately from the bounded geometry property in Definition \ref{defn:bndgeo} and the fact that by definition $s(R) \not\in B_R(n + 1)$ for $n \geq s(R)$.

(v). It is immediate from the definitions that $s(R) \leq a(R)$ and $s(R') \leq a(R')$. If $a(R) \leq s(R')$ then we are done. Otherwise $a(R) \in B_R(s(R')) \subseteq B_{R'}(s(R'))$ by (iii) and hence
\begin{equation*}
a(R') = \max B_{R'}(s(R')) \geq \max B_R(s(R')) = \max B_R(s(R)) = a(R). \qedhere
\end{equation*}
\end{proof}

\begin{lem} \label{lem:fint}
Let $R$ and $R'$ be nuclei and fix $d \geq 1$. Assume that $P \cdot d \leq Q$.
\begin{enumerate}
\item[\rm (i)] $\pathd_\Gf(f_d(R), X \setminus R) \geq d$.
\item[\rm (ii)] If $R \cap R' = \varnothing$ then $\pathd_\Gf(f_d(R), f_d(R')) \geq 2 d$.
\item[\rm (iii)] If $R \subseteq R'$ then $f_d(R) \subseteq f_d(R')$.
\item[\rm (iv)] If $f_d(R) \subseteq f_d(R')$ and $a(R) < a(R')$ then $\pathd_\Gf(f_d(R), X \setminus f_d(R')) \geq d$.
\end{enumerate}
\end{lem}

\begin{proof}
(i). $R$ is a connected component of $\Gf \res (X \setminus X_{s(R)}^\infty)$, so from Lemma \ref{lem:atoms}.(iv) we obtain
$$\pathd_\Gf(f_d(R), X \setminus R) = \pathd_\Gf(f_d(R), X_{s(R)}^\infty) \geq P \cdot d - |B_R(s(R) + 1)| \cdot d \geq d.$$

(ii). By swapping $R$ and $R'$ if necessary, we may assume that $s(R) \leq s(R')$. Since $R'$ is a finite connected component of $\Gf \res (X \setminus X_{s(R')}^\infty)$ and $R \cap X_{s(R')}^\infty = \varnothing$, we have that $\pathd_\Gf(R, R') \geq 2$. So from (i) we obtain
$$\pathd_\Gf(f_d(R), f_d(R')) \geq \pathd_\Gf(f_d(R), X \setminus R) + \pathd_\Gf(R, R') - 1 + \pathd_\Gf(X \setminus R', f_d(R')) - 1 \geq 2 d.$$ 

(iii). By Lemma \ref{lem:atoms}.(iii) we have for all $n \geq s(R')$
$$\pathd_\Gf(f_d(R), X_n^\infty) \geq P \cdot d - |B_R(n + 1)| \cdot d \geq P \cdot d - |B_{R'}(n + 1)| \cdot d.$$
It follows from the definitions that $f_d(R) \subseteq f_d(R')$.

(iv). Fix $y$ with $\pathd_\Gf(y, f_d(R)) < d$. We will argue that $y \in f_d(R')$. By (i) we have that $y \in R \subseteq R'$. Consider $s(R') \leq n < a(R')$. By Lemma \ref{lem:atoms}.(iii) we have $B_R(n + 1) \subseteq B_{R'}(n + 1)$. The assumption $a(R) < a(R')$ further gives $a(R') \in B_{R'}(n + 1) \setminus B_R(n + 1)$ so that $|B_R(n + 1)| + 1 \leq |B_{R'}(n + 1)|$. Therefore for $s(R') \leq n < a(R')$ we have
$$\pathd_\Gf(y, X_n^\infty) > \pathd_\Gf(f_d(R), X_n^\infty) - d \geq P \cdot d - |B_R(n + 1)| \cdot d - d \geq P \cdot d - |B_{R'}(n + 1)| \cdot d.$$
For $n \geq a(R')$ we have $n \not\in B_R(s(R))$ since $a(R) < a(R')$. So $\pathd_\Gf(R, X_n^\infty) > 1$ and thus by the weakly orthogonal property and (i) we have
\begin{align*}
\pathd_\Gf(y, X_n^\infty) & > \pathd_\Gf(f_d(R), X_n^\infty) - d \geq \pathd_\Gf(f_d(R), X \setminus R) + \pathd_\Gf(R, X_n^\infty) - 1 - d \\
 & \geq Q + 1 - 1 = Q \geq P \cdot d = P \cdot d - |B_{R'}(n + 1)| \cdot d.
\end{align*}
We conclude that $y \in f_d(R')$.
\end{proof}

\begin{rem}
The weakly orthogonal property has only been used thus far in proving clause (iv) of Lemma \ref{lem:fint}.
\end{rem}

For a graph $\Gf$ on $X$ and $n \geq 1$ define
$$\Gf^{(n)} = \{(x, y) : x \neq y \text{ and } \pathd_\Gf(x, y) \leq n\}.$$

\begin{thm} \label{thm:toast}
Let $X$ be a standard Borel space and let $\Gf \subseteq X \times X$ be a Borel graph on $X$. Let $(X_k)_{k \geq 1}$ be a bounded geometry weakly orthogonal decomposition of $X$ with polygonal bound $P$ and orthogonality constant $Q$. If $d \geq 1$ and $(2 d + 1) \cdot P \leq Q$, then there is a Borel set $Y \subseteq X$ such that both of the graphs $\Gf^{(d)} \res Y$ and $\Gf^{(d)} \res (X \setminus Y)$ have finite connected components.
\end{thm}

\begin{proof}
Call a nucleus $R$ {\em maximal} if $a(R) < a(R')$ for every nucleus $R'$ with $R' \supsetneq R$. For a set $A \subseteq X$ set
$$\partial_d A = \{a \in A : \pathd_\Gf(a, X \setminus A) \leq d\}.$$
Define
$$Y = \bigcup \Big\{ \partial_d f_{2d+1}(R) : R \text{ is a maximal nucleus} \Big\}.$$

Let $R$ and $R'$ be maximal nuclei with
$$\pathd_\Gf(\partial_d f_{2d+1}(R), \partial_d f_{2d+1}(R')) \leq d.$$
By Lemma \ref{lem:fint}.(i) we must have that $R \cap R' \neq \varnothing$. By swapping $R$ and $R'$, we may suppose that $s(R) \leq s(R')$ so that $R \subseteq R'$. It follows from Lemma \ref{lem:fint}.(iii) that $f_{2d+1}(R) \subseteq f_{2d+1}(R')$. We have
$$\pathd_\Gf(f_{2d+1}(R), X \setminus f_{2d+1}(R')) \leq d + \pathd_\Gf(\partial_d f_{2d+1}(R), \partial_d f_{2d+1}(R')) \leq 2 d < 2 d + 1,$$
so from Lemma \ref{lem:fint}.(iv) we conclude that $a(R) = a(R')$. By maximality it follows that $R = R'$. We conclude that every connected component of $\Gf^{(d)} \res Y$ is contained in $\partial_d f_{2d+1}(R)$ for some maximal nucleus $R$ and hence is finite.

Now fix $x \in X \setminus Y$. By Definition \ref{defn:bndgeo}.(i) there is $k \geq 1$ with $\pathd_\Gf(x, X_k^\infty) > P \cdot (2 d + 1)$. Let $R$ be the unique $k$-nucleus containing $x$. Notice that $x \in f_{2d+1}(R)$. Any nucleus $R' \supseteq R$ with $a(R') = a(R)$ must be contained in the same connected component of $\Gf \res (X \setminus X_{a(R)}^\infty)$ as $R$. Since the connected components of $\Gf \res (X \setminus X_{a(R)}^\infty)$ are finite, there must exist a maximal $R' \supseteq R$ with $a(R') = a(R)$. By Lemma \ref{lem:fint}.(iii) we have $x \in f_{2d+1}(R) \subseteq f_{2d+1}(R')$. Since $\partial_d f_{2d+1}(R') \subseteq Y$ and $x \in f_{2d+1}(R') \setminus Y$, the connected component of $x$ in $\Gf^{(d)} \res (X \setminus Y)$ is contained in $f_{2d+1}(R')$ and is therefore finite.
\end{proof}

The following corollary is based off an argument due to Conley and Miller \cite{CM2016}.

\begin{cor}\label{cor:genchrom}
Let $X$ be a standard Borel space and let $\Gf \subseteq X \times X$ be a Borel graph on $X$ with $\chi(\Gf) < \aleph_0$. Suppose that there is a bounded geometry weakly orthogonal decomposition of $X$ with polygonal bound $P$ and orthogonality constant $Q$ satisfying $5 \cdot P \leq Q$. Then
$$\chi_B(\Gf) \leq 2 \cdot \chi(\Gf) - 1.$$
\end{cor}

\begin{proof}
By Theorem \ref{thm:toast}, there is a Borel set $Y \subseteq X$ such that both of the graphs $\Gf^{(2)} \res Y$ and $\Gf^{(2)} \res (X \setminus Y)$ have finite connected components. In particular, the connected components of $\Gf \res Y$ are finite and therefore we can find a Borel chromatic coloring
$$c_Y : Y \rightarrow \{1, 2, \ldots, \chi(\Gf)\}$$
of $\Gf \res Y$. Set $Y' = Y \setminus c_Y^{-1}(\chi(\Gf))$ and set $c = c_Y \res Y'$. Since we removed a single color from $Y$, we did not remove any pair of adjacent points in $Y$. It follows that for every connected component $C'$ of $\Gf \res (X \setminus Y')$,
there is a connected component $C$ of $\Gf^{(2)} \res (X \setminus Y)$ such that $C'$ is contained in the union of $C$ and the set of 
neighbors of $C$. 
Since $C$ is finite and $\Gamma$ is locally finite, it follows that $C'$ is finite.
Therefore there is a Borel chromatic coloring
$$d \colon  X \setminus Y' \rightarrow \{\chi(\Gf), \chi(\Gf) + 1, \ldots, 2 \cdot \chi(\Gf) - 1\}$$
of $\Gf \res (X \setminus Y')$. The functions $c$ and $d$ use disjoint colors, so $c \cup d$ is a Borel chromatic $(2 \cdot \chi(\Gf) - 1)$-coloring of $\Gf$.
\end{proof}

\begin{cor} For every $n\geq 1$, $\chi_B(F(2^{\mathbb{Z}^n})=3$.
\end{cor}

\begin{proof} We observe that $\chi(F(2^{\mathbb{Z}^n}))=2$ and that Corollary~\ref{cor:wo} holds with $P=n+1$ and $Q=\frac{\epsilon}{2}d_1$, where we can take $d_1$ to be arbitrarily large. Then the conclusion holds by Corollary~\ref{cor:genchrom}. 
\end{proof}

We recall the definition of toast structures (c.f. \cite{GJKS2022} and \cite{GJKSinf}).

\begin{defn}\label{defn:toas} Let $X$ be a standard Borel space and $\Gamma\subseteq X\times X$ be a Borel graph on $X$. Let $\rho_\Gamma$ be the shortest-path-length metric on the connected components of $\Gamma$. Let $E_\Gamma$ be the equivalence relation on $X$ given by the connected components of $\Gamma$. Let $\{T_n\}$ be a sequence of subequivalence relations of $E_\Gamma$, i.e., each $T_n\subseteq E_\Gamma$ is an equivalence relation on a subset of $X$, called the {\em domain} of $T_n$, denoted $\mbox{dom}(T_n)$. We call $\{T_n\}$ a {\em (unlayered) toast} if
\begin{enumerate}
\item[(1)] Each $T_n$-equivalence class is finite.
\item[(2)] $\bigcup_n \mbox{dom}(T_n)=X$.
\item[(3)] For each $T_n$-equivalence class $C$, and each $T_m$-equivalence class $C'$ where $m>n$, if $C\cap C'\neq\varnothing$ then $C\subseteq C'$.
\item[(4)] For each $T_n$-equivalence class $C$ there is $m>n$ and a $T_m$-equivalence class $C'$ such that $C\subseteq C'\setminus\partial C'$, where $\partial C'=\{x\in C'\colon \rho_\Gamma(x, X\setminus C')=1\}$.
\end{enumerate}
We call $\{T_n\}$ a {\em layered toast} if instead of (4) above, we have
\begin{enumerate}
\item[(4')] For each $T_n$-equivalence class $C$ there is a $T_{n+1}$-equivalence class $C'$ such that $C\subseteq C'\setminus \partial C'$.
\end{enumerate}
If $\{T_n\}$ is an unlayered or layered toast, we call $\{T_n\}$ {\em Borel} if for each $n$, $\mbox{dom}(T_n)\subseteq X$ is Borel and $T_n$ is a Borel equivalence relation on $\mbox{dom}(T_n)$.
If $X$ is a zero-dimensional Polish space and $\{T_n\}$ an unlayered or layered toast on $X$, we call $\{T_n\}$ {\em continuous} if for every $n$, $\mbox{dom}(T_n)\subseteq X$ is clopen in $X$ and $T_n$ is a clopen equivalence relation on $\dom(T_n)$.
\end{defn}

With an argument similar to Lemma~\ref{lem:hyp}, we can see that the existence of a Borel unlayered toast implies hyperfiniteness of $E_\Gamma$.

\begin{thm}\label{thm:gentoast} Let $X$ be a standard Borel space and let $\Gamma$ be a Borel graph on $X$. Let $d\geq 2$. Suppose that there is a bounded geometry weakly orthogonal decomposition of $X$ with polygonal bound $P$ and orthogonality constant $Q$ satisfying $d\cdot P\leq Q$. Then there is a Borel unlayered toast on $X$.
\end{thm}

\begin{proof} As in the proof of Theorem~\ref{thm:toast} we consider maximal nuclei, i.e., nuclei $R$ for which $a(R)<a(R')$ 
for every nucleus $R'$ with $R'\supsetneq R$. As in the proof of
Theorem~\ref{thm:toast}, every nucleus is contained in some maximal nucleus.

For each $n$, define
$$ \mbox{dom}(T_n)=\bigcup\{ f_d(R)\colon s(R)=n \text{ and } R \text{ is maximal}\} $$
and $T_n$ by the partition of $\mbox{dom}(T_n)$ given by $\{f_d(R)\colon s(R)=n\}$.
Definition~\ref{defn:toas} (1) and (2) follow immediately from Definition~\ref{defn:bndgeo} (i) (decomposition). For Definition~\ref{defn:toas} (3), let $m>n$. Consider a $T_n$-equivalence class $C=f_d(R)$ and a $T_m$-equivalence class $C'=f_d(R')$. Assume $C\cap C'\neq\varnothing$. It follows that $R\cap R'\neq\varnothing$. Hence by Lemma~\ref{lem:atoms} (i), $R\subseteq R'$. By Lemma~\ref{lem:fint} (iii), $C=f_d(R)\subseteq f_d(R')=C'$.

Finally, for Definition~\ref{defn:toas} (4), let $C=f_d(R)$ be a $T_n$-equivalence class. 
Let $R'$ be any maximal nucleus with $R'\supsetneq R$. Then $R\subseteq R'$ and $a(R')>a(R)$. 
By Lemma~\ref{lem:fint} (iii) and (iv), we have $\rho_\Gamma(C, X\setminus C')\geq d\geq 2$. 
Hence $C\subseteq C'\setminus\partial C'$.
\end{proof}

\begin{cor} For each $n\geq 1$, there exists a Borel unlayered toast on $F(2^{\mathbb{Z}^n})$.
\end{cor}

\begin{proof} This follows from Corollary~\ref{cor:wo} and Theorem~\ref{thm:gentoast}, 
again noting that the condition $d\cdot P\leq Q$ can be satisfied in Corollary~\ref{cor:wo} by choosing a large enough $d_1$.
\end{proof}

In contrast, in \cite{GJKS2022} we showed that there is no Borel layered toast on $F(2^{\mathbb{Z}^n})$ for any $n\geq 1$, and in \cite{GJKSinf} we showed that there is no continuous unlayered toast on $F(2^{\mathbb{Z}^n})$ for any $n\geq 1$.

\section{Topological regularity} \label{sec:topreg}

The orthogonality results of Theorem~\ref{thm:so} and Corollary~\ref{cor:wo}
establish a certain ``geometric regularity'' for the connected components
$R$ of $X-X_k^\infty$. We prove here a result which shows that we may strengthen the
conclusion to get a ``topological regularity'' for these components as well.
In fact, we can control (in a sense to be made precise) the possible
shapes and configurations for these components. Since we only need 
these arguments for dimension $2$, we will restrict to this case. The main point is that 
we can produce an orthogonal decomposition with components homeomorphic to the 
unit ball of $\R^2$.




We continue to let $e_1, e_2$ denote the standard generators of $\Z^2$, 
or equivalently the standard unit vectors of $\R^2$. 

We first give an informal description of our construction in this section, which will be a variation of the basic orthogonal marker construction that is a refinement of the construction used in \S\ref{comr}  to prove Theorem~\ref{thm:om}.
As in that proof we will use marker distances $d_1 \ll d_2 \ll \cdots \ll d_n \ll d_{n+1}\cdots$ along with different scale values on each level $n$. The scale values will be of the form $\delta d_k$, where $d_k$ is a marker distance the use of which depends on the stage of the construction, and $\delta$ is a multiplier which is a fixed parameter not depending on $k$. We will use several such fixed parameters $0<\alpha<\epsilon_2<\eta<\epsilon_1<1$, and they do not depend on the stage of the construction in which they are used. 
In fact we will have $$\alpha<\frac{1}{10}\epsilon_2<\epsilon_2<\eta^2<\eta<2\eta<\epsilon_1<\frac{1}{2}$$ to serve various purposes. Their exact values will be determined later in the construction. 

As before, we will construct orthogonal marker decompositions $\sR^m_k$ for all $m\geq 1$ and $1\leq k\leq m$. If $(m_1,k_1)$, $(m_2,k_2)$ are pairs of positive integers with $k_1\leq m_1$ and $k_2 \leq m_2$, 
we let $(m_1,k_1)\prec (m_2,k_2)$ if $m_1 <m_2$ or $(m_1=m_2 \mbox{ and } k_1>k_2)$. This ordering has order-type $\omega$
and corresponds to the order in which we construct the $\sR^m_k$ regions. Thus, we 
construct $\sR^1_1$, $\sR^2_2$, $\sR^2_1$, $\sR^3_3$, $\sR^3_2$, $\sR^3_1, \cdots$. In particular, for each $m$, we will start with a clopen 
rectangular decomposition $\sR^m_m$ and then successively fractilize 
each region in $\sR^m_{k+1}$ to obtain the clopen decomposition $\sR^m_k$. 
We thus eventually define the clopen decomposition $\sR^m_1$. 

The rectangles in $\sR^m_m$ (the ``diagonal'' steps of the construction) 
will have side lengths roughly $\alpha d_m$, whereas the rectangles used in modifying the 
regions in $\sR^m_{k+1}$ to obtain the regions in $\sR^m_k$ 
will have side lengths between $\eta^2 d_k$ and $d_k$. Since $\alpha d_k <\frac{1}{10}\eta^2 d_k$,
each rectangle $R$ in $\sR^k_k$ will be at least $10$ times smaller than any rectangle
used to get an $\sR^m_k$ region from an $\sR^m_{k+1}$ region. 
Similar to previous constructions, we will first define an auxiliary clopen rectangular partial decomposition $\tsQ^m_k$
which will cover the boundaries of the $\sR^m_{k+1}$ regions. We will then define 
a rectangular partial decomposition $\tsR^m_k$ which refines the $\tsQ^m_k$ in that
every $R\in \tsR^m_k$ is a subset of a $Q\in \tsQ^m_k$. The $\tsR^m_k$ regions will also cover the 
boundaries of the $\sR^m_{k+1}$ regions.
Finally, we use a continuous (but otherwise arbitrary) assignment of the $\tsR^m_k$ 
regions to one of the $\sR^m_{k+1}$ regions it intersects. This will define the $\sR^m_k$ 
regions. 

We next describe the construction in more, formal details. Asssume inductively that the $\sR^{m_1}_{k_1}$ have been constructed for all 
$(m_1,k_1)\prec (m,k)$ and that the following inductive hypotheses are satisfied. 
\medskip

{\bf Inductive Hypotheses} 


\begin{enumerate}[label={(H\arabic*)}]
\item \label{i1}
Each $R\in \sR^{m_1}_{m_1}$ is a rectangle with side lengths between 
$\frac{1}{2}\alpha d_{m_1}$ and $\alpha d_{m_1}$. 
\item \label{i2}
$\tsQ^{m_1}_{k_1}$ is defined if $k_1<m_1$, and is a clopen rectangular partial decomposition 
which covers the boundaries of the regions in $\sR^{m_1}_{k_1+1}$. Each rectangle 
$Q\in \tsQ^{m_1}_{k_1}$ has side lengths between $\frac{1}{2} d_{k_1}$ and $d_{k_1}$. 
\item (orthogonality of the $\tsQ$) \label{i3}
If $k_1<m_1$ and $F$ is a side of a rectangle in $\tsQ^{m_1}_{k_1}$, then $F$  
is at least $\epsilon_1 d_{k_1}$ from any parallel side $F'$ of a rectangle in one of the following decompositions:
\begin{itemize}
\item $\sR^{k_1+1}_{k_1+1}$,
\item $\tsR^{m_2}_{k_1+1}$ for $m_2\in (k_1+1, m_1]$,
\item $\tsQ^{\ell}_{k_1}$ for $\ell\in (k_1, m_1)$,
\item $\tsQ^{m_1}_{k_1}$ if $\rho(F, F')>1$;
\end{itemize}
Also, $F$ is at least $\alpha \epsilon_1 d_{k_1}$ from a parallel side of a rectangle in $\sR^{k_1}_{k_1}$. 
\item \label{i4}
$\tsR^{m_1}_{k_1}$ is defined if $k_1<m_1$ and $\tsR^{m_1}_{k_1}$ is a clopen 
rectangular refinement of $\tsQ^{m_1}_{k_1}$ which still covers the boundaries of the regions in $\sR^{m_1}_{k_1+1}$. In particular, each $R \in \tsR^{m_1}_{k_1}$
is a subrectangle of a $Q\in \tsQ^{m_1}_{k_1}$. 
Each rectangle $R\in \tsR^{m_1}_{k_1}$ has side lengths between $\eta^2 d_{k_1}$ and $d_{k_1}$. 
\item (short directions for the $\tsR$) \label{i5}
If $R\in \tsR^{m_1}_{k_1}$ and $R$ intersects $2$ non-parallel sides of regions in $\tsR^{m_1}_{k_1+1}$,
then $R$ has side lengths $\eta d_{k_1}$. In this case, we say that $R$ 
has {\em type $2$}. Otherwise, $R$ intersects exactly one side $F$ of a region in $\tsR^{m_1}_{k_1+1}$. 
Then the length of $R$ in the direction perpendicular to $F$ is $\eta^2 d_{k_1}$. 
In this case we say that $R$ has {\em type $1$}, and call the direction perpendicular to $F$ 
the {\em short direction} for $R$.
\item (coherence of $\tsQ$ and $\tsR$) \label{i6}
If $F$ is a side of a rectangle $R\in \tsR^{m_1}_{k_1}$, then either $F$ is part of a side
of a rectangle $Q\in \tsQ^{m_1}_{k_1}$ or $F$ is within $\eta d_{k_1}$ of a parallel boundary side 
of a region in $\sR^{m_1}_{k_1+1}$. In fact, if $R$ is of type-$1$ and not adjacent 
to a rectangle of type-$2$, then $F$ is within $\eta^2 d_{k_1}$ of a parallel boundary side 
of a region in $\sR^{m_1}_{k_1+1}$. 
\item (orthogonality of the $\tsR$) \label{i7a}
Any side $F$ of a rectangle $R\in \tsR^{m_1}_{k_1}$ is at least $\epsilon_2 d_{k_1}$ from any 
parallel side $F'$ of a rectangle $R'\in \tsR^{m_1}_{k_1}$ if $\rho(F, F')>1$.

\item (orthogonality of the $\tsR$) \label{i7b}
Any side $F$ of a rectangle $R\in \tsR^{m_1}_{k_1}$ is at least $\epsilon_2 d_{k_1}$ from any 
parallel side of a rectangle $R' \in \tsR^{m_2}_{k_1}$ for $m_2\in (k_1, m_1)$. 
Also, $F$ is at least $\alpha \epsilon_2 d_{k_1}$ from any parallel side of a rectangle in $\sR^{k_1}_{k_1}$.

\end{enumerate}

\begin{rem} \label{rem:oc}
Although the constants $\epsilon_1, \eta, \epsilon_2$ and $\alpha$ are pre-fixed, absolute constants not depending on the stages of the construction, we will determine their exact values through describing the construction in detail at the inductive step. 
The constant $\epsilon_1$ will be a preliminary orthogonality constant used in constructing the 
$\tsQ$ (to maintain the orthogonality in the inductive hypothesis \ref{i3}). We will then pick constant $\eta$ with $\eta^2<\eta< 2\eta<\epsilon_1$
which is used in construction the $\tsR$ rectangles as a control of the side lengths. The $\tsR$ rectangles 
will be adjusted to satisfy orthogonality (items \ref{i7a} and \ref{i7b} 
of the induction hypothesis), which determines the orthogonality constant $\epsilon_2<\eta^2$. Throughout this construction the diagonal decompositions $\sR^m_m$ will be specially defined to have a multiplier $\alpha$ for their side lengths compared to the non-diagonal decompositions. A key observation is that the orthogonality estimates involving the diagonal decompositions hold with the same multiplier $\alpha$ no matter what the value of $\alpha$ is. In other words, none of the constants $\epsilon_1, \eta$ and $\epsilon_2$ depend on the value of $\alpha$.  We finally pick $\alpha < \frac{1}{10}\epsilon_2$. In the end, the constant $\epsilon=\alpha \epsilon_2$
will be the final resulting orthogonality constant for the $\tsR$. 
\end{rem}

\begin{rem}
If $R\in \tsR^{m_1}_{k_1}$ is of type $2$, then $R$ has edge lengths $\eta d_{k_1}$ from 
\ref{i6}. If $R$ is of type $1$ and not adjacent to a type $2$ rectangle 
$R'\in \tsR^{m_1}_{k_1}$ then $R$ has short dimension $\eta^2 d_{k_1}$ and 
has other side length between $\frac{1}{2} d_{k_1}$ and $d_{k_1}$ from 
\ref{i3} and \ref{i6}. If $R$ is of type $1$ and is adjacent to 
an $R$ of type-$2$ then $R$ still has short dimension $\eta^2 d_{k_1}$
and has other side length between $(\epsilon_1-\eta) d_{k_1}$ 
and $d_{k_1}$. This follows from \ref{i4} and \ref{i5}. 
\end{rem}

We now turn to the construction of the $\sR^m_k$ regions. For $k=m$, 
let $\sR^m_m$ be a rectangular clopen decomposition of $F(2^{\Z^2})$ with each 
rectangle $R\in \sR^m_m$ having side lengths between 
$\frac{1}{2}\alpha d_m$ and $\alpha d_m$. Then the inductive hypothesis \ref{i1} is maintained. 
Assume now the clopen partition $\sR^m_{k+1}$ 
has been defined, and we define $\sR^m_k$. 

Let $\tsP^m_k$ be a clopen rectangular partition of $F(2^{\Z^2})$ satisfying the following:

\begin{enumerate}
\item \label{v1}
Each rectangle $Q\in \tsP^m_k$ has side lengths between $\frac{1}{2} d_k$ and $d_k$.
\item \label{v4}
Each side of a rectangle $Q\in \tsP^m_k$ is at least $\epsilon_1 d_k$ from a parallel side
of a rectangle $R\in \sR^{k+1}_{k+1}$.
\item \label{v4.5}
Each side of a rectangle $Q\in \tsP^m_k$ is at least $\epsilon_1 d_k$ from a parallel side
of a rectangle $R\in \tsR^{m_2}_{k+1}$ for any $m_2\in (k+1, m]$.
\item \label{v2}
Each side of a rectangle $Q\in \tsP^m_k$ is at least $\epsilon_1 d_k$ from a parallel side 
of a rectangle $Q'\in \tsQ^\ell_{k}$ for any $\ell\in(k,m)$.
\item \label{v2.5}
Each side $F$ of a rectangle $Q\in \tsP^m_k$ is at least $\epsilon_1 d_k$ from a parallel side $F'$
of a rectangle $Q' \in \tsP^m_k$ 
if $\rho(F,F')>1$. 
\item \label{v3}
Each side $F$ of a rectangle $Q\in \tsP^m_k$ is at least $\alpha \epsilon_1 d_k$
from a parallel side $F'$ of a rectangle $Q'\in \sR^k_{k}$.
\end{enumerate}

The construction of the $\tsP^m_k$ is by the standard orthogonal marker region construction using the inductive hypotheses. 
The orthogonality condition in (\ref{v4}) is guaranteed by \ref{i1} and the fact that $\alpha d_{k+1}>d_k$ since $d_{k+1}\gg d_k$. The condition in (\ref{v4.5}) is guaranteed by \ref{i7a} and \ref{i7b}, together with the fact that $\alpha\epsilon_2 d_{k+1}>10d_k$ because $d_{k+1}\gg d_k$. For the condition in (\ref{v2}), we note that by \ref{i2} each rectangle in $\tsQ^\ell_k$ for $\ell\in (k,m)$ is within distance $d_k$ to the boundaries of a rectangle in $\tsR^\ell_{k+1}$. Again by \ref{i7a} and \ref{i7b}, together with the fact that $\alpha\epsilon_2 d_{k+1}>10d_k$, we conclude that there is an absolute upper bound (independent from $m, k$) for the number of such $\tsQ^\ell_k$ rectangles $Q'$ to consider. The condition in (\ref{v2.5}) can be guaranteed by the standard construction. Finally, 
 the factor of $\alpha$ in (\ref{v3}) comes from the fact that 
the rectangles in $\sR^k_k$ have smaller side lengths of size roughly $\alpha d_k$, as opposed to $d_k$.
Again, this orthogonality condition can be guranteed for any value of $\alpha$ as long as the $\sR^k_k$
rectangles have side lengths between $\frac{1}{2} \alpha d_k$ and $\alpha d_k$, which is given by \ref{i1}. 

The consideration of (\ref{v1})--(\ref{v2.5}) gives an absolute upper bound for the number of rectangle sides to avoid. This upper bound, which does not depend on $m$ and $k$, determines the constant $\epsilon_1$. Then the consideration of (\ref{v3}) follows for any constant value $\alpha$ which is independent from $\epsilon_1$. The exact value of $\alpha$ will be determined later in the construction.
 
Now we define $\tsQ^m_k$ to consist of those rectangles in $\tsP^m_k$ which intersect the boundary of a $\sR^m_{k+1}$
region. Then it is straightforward to check that the inductive hypotheses \ref{i2} and \ref{i3} are maintained.

We now define a clopen rectangular partial decomposition $\tsR^m_k$ which still covers the boundaries of $\sR^m_{k+1}$. Each rectangle $R\in \tsR^m_k$ is a subrectangle of some $Q \in \tsQ^m_k$. The construction of $R$ from $Q$ will be done in a continuous manner. We will then let $\sR^m_k$ be obtained from $\sR^m_{k+1}$
by using the auxiliary partial decomposition $\tsR^m_k$ in the usual way: we assign each region $R\in \tsR^m_k$
 to one of the regions in $\sR^m_{k+1}$ which $R$ intersects. For each region $D\in \sR^m_{k+1}$, let $U$ be the union of all the rectangles $R \in \stR^m_k$ which are assigned to $D$, and let $V$ be the union of all the rectangles $R'\in\stR^m_k$ which intersect the boundary of $D$ but are not assigned to $D$ (they are assigned to some neighboring regions of $D$). Let $D'=(D\cup U)-V$.  Then $\sR^m_k$ is the collection of all such $D'$ for $D\in \sR^m_{k+1}$. 

Observe that from this construction we will have maintained inductively that any boundary side of a region in $\sR^m_k$ is part of a side of some rectangle in $\tsR^m_k$. We will use this observation tacitly in our proof below.

We will be using a constant $\eta$ that satisfies $\eta^2 <\eta <2\eta<\epsilon_1$. As we proceed to describe the rectangles in $\tsR^m_k$, we use Figure~\ref{fig:topregcon} to illustrate the construction.

Let $Q\in \tsQ^m_k$. As in \ref{i5}, 
we say $Q$ has type $1$ if $Q$ intersects exactly one side of a rectangle in 
$\sR^m_{k+1}$, and say $Q$ is of type $2$ if it intersects $2$ such sides.

First suppose that $Q$ has type $2$. Let $S_0 \subseteq Q$ be a rectangle centered at a corner point 
of an $\sR^m_{k+1}$ region, and such that $S_0$ has side lengths $\eta d_k$. 
This is possible by (\ref{v4.5}) above and the fact that $2\eta<\epsilon_1$. 
The basic orthogonality arguments give that we may translate 
$S_0$ in each direction by a small amount (say by an amount $<\frac{1}{10} \eta^2 d_k$) 
such that any side of the translated $S_0$ is at least $\epsilon_2 d_k$ from a 
parallel side of a rectangle in $\tsR^{\ell}_{k}$ for any $\ell\in (k,m)$, and also at least $\alpha\epsilon_2 d_k$
from a parallel side of a rectangle in $\sR^k_k$. We may also assume the translation is small enough so that 
any side of the translated $S_0$ is at least $\frac{\eta}{2} d_k$ from a parallel side of a region in $\sR^m_{k+1}$. In doing this construction we use the inductive hypotheses \ref{i2}, \ref{i4} and \ref{i7b}. 
This orthogonality argument is the first of several that will determine the value of $\epsilon_2$. One may think of a current value of $\epsilon_2$, noting that we may have to shrink the value later. In any case, $\epsilon_2$ will depend on $\eta$ (but not on $m, k$), and we assume that $\epsilon_2<\eta^2$. Similar to the situation of (\ref{v3}) above, this construction works with any value $\alpha$.

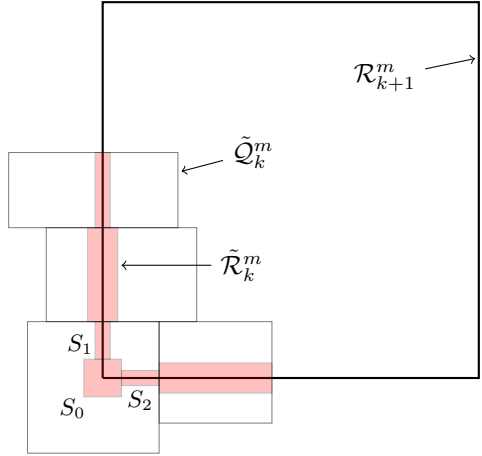
\begin{figure}[h] 
\begin{tikzpicture}[scale=0.05]

\pgfmathsetmacro{\a}{0.5}
\pgfmathsetmacro{\b}{0.25}

\draw[thick] (0,0) to (100,0) to (100,100) to (0,100) to (0,0);

\draw[opacity=\a] (-20,-20) rectangle (15,15); 
\draw[opacity=\a] (-15,15) rectangle (25,40);
\draw[opacity=\a] (-25,40) rectangle (20,60);

\draw[opacity=\a] (15,-12) rectangle (45,15);

\draw[fill=red, opacity=\b] (-5,-5) rectangle (5,5); \node at (-8,-8){\small $S_0$};
\draw[fill=red, opacity=\b] (-2,5) rectangle (2, 15); \node at (-6,9){\small $S_1$};
\draw[fill=red, opacity=\b] (-4,15) rectangle (4, 40);
\draw[fill=red, opacity=\b] (-2,40) rectangle (2, 60);
\draw[fill=red, opacity=\b] (5,-2) rectangle (15, 2); \node at (10,-6){\small $S_2$};
\draw[fill=red, opacity=\b] (15,-4) rectangle (45, 4);

\node (A) at (75,80) {$\sR^m_{k+1}$};
\draw[->] (A) to (99, 85);

\node (B) at (40,60) {$\tilde \sQ^m_k$};
\draw[->] (B) to (21,55);

\node (C) at (37, 30) {$\tilde \sR^m_k$};
\draw[->] (C) to (5, 30);

\end{tikzpicture}
\caption{Construction of the $\tilde \sR^m_k$} \label{fig:topregcon}
\end{figure}

We continue to construct other subrectangles of $Q$ which will be in $\tsR^m_k$. Let $S_0$ denote the translated rectangle constructed above. For each of the $2$, $3$, or $4$ sides $s$ of an $\sR^m_{k+1}$ region which are incident to the corner point,
we let $S_1$, $S_2$ (and possibly $S_3$, $S_4$) be rectangles contained in $Q\setminus S_0$ 
which contain the the portion of the boundary sides of an $\sR^m_{k+1}$ region intersecting $Q$ which
lie outside of $S_0$, and each of these rectangles has width $\eta^2 d_k$. (In Figure~\ref{fig:topregcon} only $S_1$ and $S_2$ are shown.) Note that we only need to specify two sides of the $S_i$ since the other two sides are parts of the sides of $Q$ and $S_0$, respectively. 
We then translate the $S_i$ rectangles by no more than $\frac{1}{10}\eta^2 d_k$  
so that the two sides of $S_i$ which are 
perpendicular to their adjacent side of $Q$ are at least $\epsilon_2 d_k$ from a parallel side
of a rectangle in $\tsR^{\ell}_{k}$ for $\ell\in (k,m)$, and at least $\alpha \epsilon_2 d_k$ 
from a parallel side
of a rectangle in $\sR^{k}_{k}$.  In doing this construction we use the inductive hypotheses \ref{i2}, \ref{i4}, \ref{i7a} and \ref{i7b}.
Also, we may have to shrink the value of $\epsilon_2$ obtained from the first step where we constructed $S_0$. But the construction works with any value of $\alpha$.

Next we consider a rectangle $Q$ of type $1$. In this case, there is only one boundary side $F$ of an $\sR^m_{k+1}$
region which intersects $Q$. We let $S_5 \subseteq Q$ be a rectangle containing $F \cap Q$, with the two sides parallel to $F$ having distance $\eta^2 d_k$ and the other two sides being parts of sides of $Q$. 
We then translate $S_5$  perpendicular to $F$ so that the long sides of $S_5$ (the sides parallel to $F$)
are:
\begin{itemize}
\item at least $\epsilon_2 d_k$ away from parallel sides of a rectangle in $\tsR^{\ell}_{k}$ for $\ell\in (k, m)$, 
\item at least $\epsilon_2 d_k$ away from parallel sides of rectangles of the form $S_1, S_2, S_3$ and $S_4$, and
\item at least $ \alpha \epsilon_2d_k$ away from parallel sides of a rectangle in $\sR^{k}_{k}$.
\end{itemize}
As above, this uses the inductive 
hypotheses \ref{i2}, \ref{i4}, \ref{i7a} and \ref{i7b} and may require shrinking the constant $\epsilon_2$ again, but the construction works with any value of $\alpha$.  

This finishes the construction of $\tsR^m_k$, and the value of $\epsilon_2$ is determined from the above orthogonality arguments. At this point, we fix the consatnt $\alpha<\frac{1}{10}\epsilon_2$.


We now verify the inductive hypotheses for the $\tsR^m_k$ rectangles. 
Assume $k<m$. 
The inductive hypotheses \ref{i4}, \ref{i5} and \ref{i6}) are immediate from the construction
of the rectangles in $\tsR^m_k$. The inductive hypothesis \ref{i7a} is implicit from the above construction; just note that the sides of rectangles of the form $S_0$--$S_5$ are designed to be sufficiently far apart.

The next lemma verifes induction hypothesis \ref{i7b}.

\begin{lem} (orthogonality of the $\tsR$)
The $\tsQ$, $\tsR$ rectangles satisfy inductive hypothesis \ref{i7b}.
\end{lem}


\begin{proof}
Let $F$ be a side of a rectangle $R\in\tsR^{m}_{k}$. Then by our construction, $F$ is either part of a side of a rectangle $Q\in\tsQ^m_k$ or is within $\eta d_k$ of a parallel side on the boundary of a region in $\sR^m_{k+1}$. In the latter case, $F$ is within $\eta d_k$ of a parallel side of a rectangle in $K\in \tsR^m_{k+1}$.

Suppose $F'$ be a side of a rectangle $R'\in \tsR^{\ell}_{k}$ where $\ell\in (k,m)$.
Since $(\ell, k)\prec(m,k)$, by \ref{i4}, \ref{i5} and \ref{i6} for $\tsR^{\ell}_k$, $F'$ is either part of a side of a rectangle $Q'\in \tsQ^{\ell}_{k}$
or is within $\eta d_{k}$ of a parallel side of a rectangle $K'\in \tsR^{m_2}_{k+1}$ (or $K'\in \sR^{m_2}_{k+1}$ if $m_2=k+1$). 

If $F$ is in the $K$ case, then the result follows from the above constructions of the $S_0$--$S_5$ regions. If both $F$, $F'$ are in the $Q$, $Q'$ case respectively, then the result follows from (\ref{v2}) above.  If $F$ is in the $Q$ case and $F'$ is in the $K'$ case, then by (\ref{v4}) and (\ref{v4.5}) above, $F$ is $\epsilon_1 d_k$ away from any parallel side of $K'$. Since $F'$ is within $\eta d_k$ of such a parallel side, we conclude that $F$ and $F'$ are at least $(\epsilon_1-\eta)d_{k}> \eta d_{k}>\epsilon_2 d_k$ apart.

Now suppose $F'$ is a side of a rectangle $R'\in \sR^k_k$. If $F$ is in the $Q$ case the result follows from (\ref{v3}). If $F$ is in the $K$ case the result follows from the above constructions of the $S_0$--$S_5$.
\end{proof}

For convenience, we summarize the sizes of the $\tsR^m_k$ rectangles 
in the following. 

\begin{fact}
The rectangles in $\tsR^m_k$ form a clopen assignment and the rectangles in $\tsR^m_k$ 
cover the boundaries of the regions in $\sR^m_{k+1}$.  We furthermore have the following.
\begin{enumerate}
\item
Each rectangle in $\tsR^m_k$ has side lengths between $\eta^2 d_k$ and $d_k$, and is of 
type-$1$ or type-$2$ with respect to the regions in $\sR^m_{k+1}$. 
\item
A type-$2$ rectangle in $\tsR^m_k$ has side lengths $\eta d_k$. 
\item
A type-$1$ rectangle in $\tsR^m_k$ has smaller side length $\eta^2 d_k$ in the short direction.
\item
A type-$1$ rectangle $R$ in $\tsR^m_k$ has smaller sides which are part of the boundaries of 
a region in $\tsQ^m_k$ if $R$ is not adjacent to a type-$2$ rectangle in $\tsR^m_k$. In particular,
such an $R$ has longer side lengths between $\frac{1}{2} d_k$ and $d_k$. 
\end{enumerate}
\end{fact}

We let $\epsilon=\alpha \epsilon_2$, which is the smallest of the constants appearing in the 
induction hypotheses. In most of the arguments we can use the single constant $\epsilon$ 
in our estimates. An exception occurs in the following Claim~\ref{claim:sii}.

\begin{thm} \label{hmt}
For each $\ell$, the connected components of $F(2^{\Z^2})\sm \bigcup_{m=\ell}^\infty \partial \sR^m_1$ are 
homeomorphic to disks. 
\end{thm}

\begin{proof}
We fix $\ell$ and show by reverse induction on $k$, starting from $k=\ell$, that 
the connected components of $\sN^\ell_k$ are homeomorphic to disks,
where $\sN^\ell_k= F(2^{\Z^2})\sm \bigcup_{m=\ell}^\infty \partial \sR^m_k$. 
The result for $k=1$ then finishes the proof.

Consider first the case $k=\ell$. The connected components of $\sN^\ell_\ell$
are contained within a single $\sR^\ell_\ell$ rectangle $R$.

\begin{claim} \label{claim:si}
For each of the two directions $e_1$, $e_2$, there is 
at most one value of $m>\ell$ such that  a boundary side of an $\sR^m_\ell$ region that is parallel to $e_i$ intersects $R$. 
\end{claim}

\begin{proof}
Fix $e \in \{ e_1,e_2\}$, and suppose there is a boundary side $F_1$ of a region in $\sR^{m_1}_{\ell}$ and a parallel
coundary side $F_2$ of a region in $\sR^{m_2}_{\ell}$ such that $F_1$ and $F_2$ both intersect the rectangle 
$R\in \sR^\ell_\ell$. Since $\ell$ is not equal to $m_1$ or $m_2$, from \ref{i7a} and \ref{i7b} of the induction hypotheses
we have that $\rho(F_1,F_2)>\epsilon_2 d_\ell>10\alpha d_\ell$. Since $R$ has side lengths $\alpha d_\ell$, 
both $F_1$ and $F_2$ cannot intersect $R$. 
\end{proof}

Claim~\ref{claim:si} says that each rectangle $R$ in $\sR^\ell_\ell$ is divided into connected 
components of $\sN^\ell_\ell$ in one of a small number of possible ways; these are illustrated in 
Figure~\ref{fig:possible_shapes}. Thus the theorem holds in the case $k=\ell$.

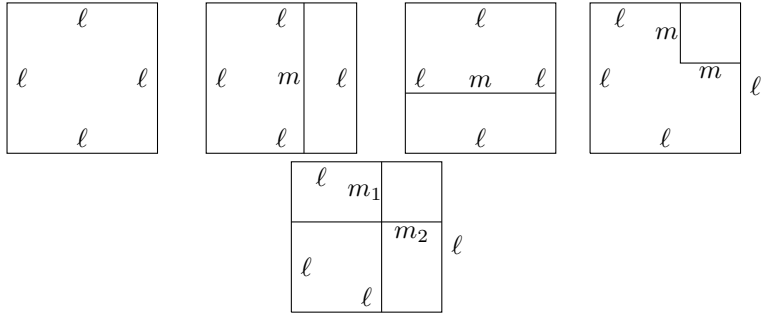
\begin{figure}[h] 

\begin{subfigure} {0.2\textwidth}
\centering

\begin{tikzpicture}[scale=0.02]

\pgfmathsetmacro{\a}{0.5}

\draw (0,0) rectangle (100,100);
\node (A) at (10,50) {$\ell$};
\node (B) at (50,90) {$\ell$};
\node (A) at (90,50) {$\ell$};
\node (A) at (50,10) {$\ell$};

\end{tikzpicture}

\end{subfigure}
\begin{subfigure} {0.2\textwidth}
\centering

\begin{tikzpicture}[scale=0.02]

\pgfmathsetmacro{\a}{0.5}

\draw (0,0) rectangle (100,100);
\node (A) at (10,50) {$\ell$};
\node (B) at (50,90) {$\ell$};
\node (A) at (90,50) {$\ell$};
\node (A) at (50,10) {$\ell$};
\draw (65,100) to (65,0);
\node(E) at (55,50) {$m$};
\end{tikzpicture}

\end{subfigure}
\begin{subfigure} {0.2\textwidth}
\centering
\begin{tikzpicture}[scale=0.02]

\pgfmathsetmacro{\a}{0.5}

\draw (0,0) rectangle (100,100);
\node (A) at (10,50) {$\ell$};
\node (B) at (50,90) {$\ell$};
\node (A) at (90,50) {$\ell$};
\node (A) at (50,10) {$\ell$};
\draw (0,40) to (100,40);
\node(E) at (50,48) {$m$};

\end{tikzpicture}
\end{subfigure}
\begin{subfigure} {0.2\textwidth}
\centering
\begin{tikzpicture}[scale=0.02]

\pgfmathsetmacro{\a}{0.5}

\draw (0,0) to (100,0) to (100,100) to (0,100) to (0,0);
\draw (60,100) to (60,60) to (100,60);
\node (A) at (10,50) {$\ell$};
\node (B) at (50,10) {$\ell$};
\node (C) at (20,90) {$\ell$};
\node (D) at (51,80) {$m$};
\node (E) at (110,45) {$\ell$};
\node (F) at (80,54) {$m$};
\end{tikzpicture}
\end{subfigure}

\begin{subfigure} {0.2\textwidth}
\centering
\begin{tikzpicture}[scale=0.02]

\pgfmathsetmacro{\a}{0.5}

\draw (0,0) to (100,0) to (100,100) to (0,100) to (0,0);
\draw (0,60) to (100,60);
\draw (60,0) to (60,100);
\node (A) at (10,30) {$\ell$};
\node (B) at (50,10) {$\ell$};
\node (C) at (20,90) {$\ell$};
\node (D) at (49,80) {$m_1$};
\node (E) at (110,45) {$\ell$};
\node (F) at (80,52) {$m_2$};
\end{tikzpicture}
\end{subfigure}

\caption{Possible ways in which an $\sR^\ell_\ell$ rectangle is divided 
by higher level $\sR^m_\ell$ regions.}
\label{fig:possible_shapes}
\end{figure}

We now assume the connected components of $\sN^\ell_{k+1}$ are homeomorphic to disks, and 
show the result for $\sN^\ell_k$.

\begin{claim}\label{claim:sii}
Let $R_0\in R^\ell_\ell$, $R_1 \in \sR^{m_1}_{\ell}$, and possibly $R_2\in \sR^{m_2}_\ell$
be such that $\partial R_i\cap \partial R_0\neq \emptyset$ for $i=1,2$. 
Let $R_0(k)\in\sR^\ell_k$, $R_1(k)\in\sR^{m_1}_k$, $R_2(k)\in \sR^{m_2}_k$ be the regions corresponding to $R_0$, $R_1$ and $R_2$, respectively. 
Then any member of $\sN^\ell_k$ contained in $R_0(k)$ has boundary contained in $\partial R_0(k) \cup \partial R_1(k)\cup\partial R_2(k)$. 
\end{claim}

\begin{proof}

Figure~\ref{fig:possible_shapes} illustrates the possible ways in which this situation can happen. To be specific, let us
consider the case where $R_1 \in \sR^{m_1}_\ell$ exists but $R_2$ does not and we are in the 
second case of Figure~\ref{fig:possible_shapes} (the other cases are similar).  
From \ref{i7b} we have that the boundary of $R_1$ is at least $\epsilon_2 \alpha d_\ell$ 
from a parallel boundary of $R_0$. Since $d_k\ll \cdots \ll d_{\ell-1}\ll d_\ell$, we have $\sum_{i=k}^{\ell-1} d_i< \frac{1}{2} \epsilon_2\alpha d_\ell$, and it follows that 
$\partial R_1(k) \cap R_0(k)\neq \emptyset$. Similarly, if $R' \neq R_1(k)$ is another regon in 
$\bigcup_{m=\ell}^\infty \sR^m_k$, then $\partial R'\cap R_0(k)=\emptyset$ and so $\partial R_1(k)$
is the only $k$-level boundary which intersects $R_0(k)$. Thus, any element of $\sN^\ell_k$ contained in $R_0(k)$ is formed
from $\partial R_1(k)$, and its boundary is contained in $\partial R_0(k)\cup\partial R_1(k)$. 

\end{proof}

In view of Claim~\ref{claim:sii} we only need to consider at most three regions $R_0\in\sR^{\ell}_{\ell}$, $R_1\in\sR^{m_1}_{\ell}$, $R_2\in \sR^{m_2}_{\ell}$ in the situation illustrated in Figure~\ref{fig:possible_shapes} and show that the members of $\sN^{\ell}_k$ contained in $R_0(k)$ are homeomorphic to disks. In the following argument we continue to use the notation established in Claim~\ref{claim:sii}. 

Clearly $R_1=R_1(\ell)$ divides $R_0=R_0(\ell)$ into two connected components, each of which (a rectangle)
is homeomorphic to a disk. Suppose inductively that $R_1(k+1)$ divides $R_0(k+1)$
into two components, each of which is homeomorphic to a disk. Assume inductively that the following hold. 

\begin{enumerate}[label={(I\arabic*)}]
\item \label{as1}
$\partial R_1(k+1)$ intersects $\partial R_0(k+1)$ in exactly two points $x=x(k+1)$, $y=y(k+1)$. 
\item \label{as2}
At the points $x$, $y$, a vertical (horizontal) edge of $\partial R_1(k+1)$ intersects a horizontal (vertical) edge of 
$\partial R_0(k+1)$. 
\end{enumerate}

Note that by \ref{i7b}, any boundary side of $\partial R_1(k+1)$ is at least $\epsilon d_{k+1}$ from a parallel coundary side 
of $\partial R_0(k+1)$. Let us call $L_1$ the boundary side of  $\partial R_1(k+1)$ containing $x$, and 
$L_0$ the boundary side of $\partial R_0(k+1)$ containing $x$ (the argument at the point $y$ is similar). We use Figure~\ref{fig:hdfig} to illustrate the situation as we proceed with the argument.

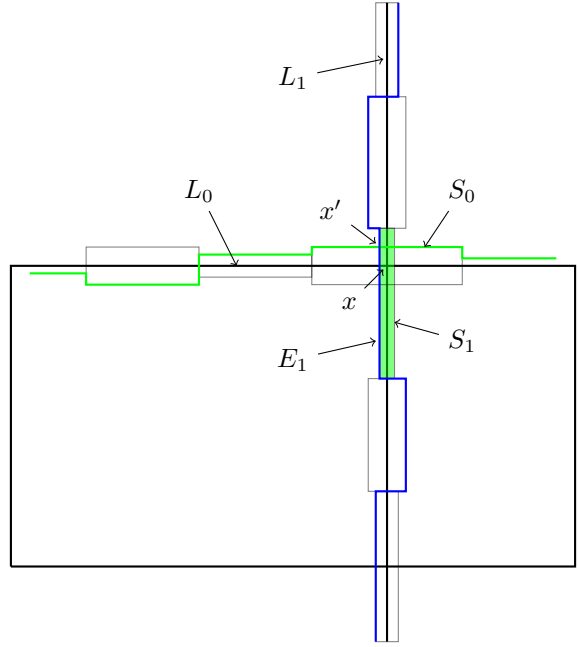
\begin{figure}[h] 
\begin{tikzpicture}[scale=0.05]

\pgfmathsetmacro{\a}{0.5}
  
\draw[thick] (0,0) to (150,0) to (150,80) to (0,80) to (0,0);
\draw[thick] (100,-20)to (100, 150);


\draw[opacity=\a] (97,-20) rectangle (103,20);
\draw[opacity=\a] (95,20) rectangle (105, 50);

\draw[opacity=\a, fill=green] (98,50) rectangle (102, 90);

\draw[opacity=\a] (95,90) rectangle (105, 125);
\draw[opacity=\a] (97,125) rectangle (103, 150);

\draw[opacity=\a] (20,75) rectangle (50,85);
\draw[opacity=\a] (50,77) rectangle (80,83);
\draw[opacity=\a] (80,75) rectangle (120,85);


\draw[color=green, thick] (5,78) to (20,78) to (20,75) to (50,75) to (50,83)
to (80,83) to (80,85) to (120,85) to (120,82) to (145,82);

\draw[color=blue, thick] (103,150) to (103,125) to (95,125) to (95,90) to (98,90)
to (98,50) to (105,50) to (105,20) to (97,20) to (97,-20);





\node (E) at (75,55) {$E_1$};
\draw[->] (E) to (97, 60);

\node (F) at (75,130) {$L_1$};
\draw[->] (F) to (99,135);

\node (G) at (50,100) {$L_0$};
\draw [->] (G) to (60,80);

\node (J) at (90,70) {$x$};
\draw [->] (J) to (99, 79);

\node (K) at (85, 95) {$x'$};
\draw [->] (K) to (97,86);

\node (H) at (120,100) {$S_0$};
\draw [->] (H) to (110,85);

\node (I) at (120,60) {$S_1$};
\draw [->] (I) to (102,65);

\end{tikzpicture}
\caption{Proof of Theorem~\ref{hmt}. The green shaded rectangle is $S_1$. The boundary of $R_1(k)$
is shown in blue, and the boundary of $R_0(k)$ in green.} \label{fig:hdfig}
\end{figure}

The segment $L_1$ must extend vertically above and below $L_0$ at least $\epsilon d_{k+1} \gg d_k$, which follows from 
\ref{i7b}. Similarly, $L_0$ extends horizontally past $x$ at least $\epsilon d_{k+1} \gg d_k$. Also, outside of the balls of 
radius $\epsilon d_{k+1}$ around $x$ and $y$, every point of $\partial R_1(k+1)$ between $x$ and $y$ is at least 
$\epsilon d_{k+1}$ from any point of  $\partial R_0(k+1)$. This follows from \ref{i7b} and assumption \ref{as1} above. 
It is therefore immediate that outside of these balls we also have that $\partial R_1(k)$ and $\partial R_0(k)$
are disjoint. Let us consider the ball $B$ of radius $\epsilon d_{k+1}$ aound $x$. 
Let $S_1 \in \tsR^m_k$ containing $x$. This is one of the rectangles used in forming $R_1(k)$ from $R_1(k+1)$ (by 
either adding or removing it from $R_1(k+1)$). 
Let $E_1$ be the vertical edge of $S_1$ which forms part of the boundary of $R_1(k)$. 
Note that $E_1$ crosses $L_0$.
From \ref{i6} we have that the short sides of $S_1$ 
(those parallel to $L_0$) are part of the sides of a rectangle $Q_1 \in \tsQ^m_k$. 
From \ref{i3} we have that these short sides of $S_1$ 
are at least $\epsilon_1 d_k$ away from $L_0$. Thus, $E_1$ extends vertically at least $\epsilon_1 d_k$ above and below $L_0$.

On the other hand, let $S_0 \in \tsR^\ell_k$ be the rectangle used in forming $R_0(k)$ from $R_0(k+1)$ such that 
$E_1$ intersects $S_0$ (or equivalently $S_0$ contains $x$). From \ref{i5} we have that the sides of $S_0$ parallel to $L_0$ 
are within $\eta^2 d_k < \epsilon_1 d_k$ of $L_0$. Thus, the part of $\partial R_1(k)$ formed by the 
short sides of $S_1$ and $E_1$ only intersects $\partial R_0(k)$ in a single point $x'$ on a long side of $S_0$. 
It likewise follows from \ref{i5} that inside the ball $B$, $x'$ is the only point of intersection of 
$\partial R_1(k)$ and $\partial R_0(k)$. This establishes the above assumptions \ref{as1} and \ref{as2} for 
$R_1(k)$ and $R_0(k)$.

The above argument applies verbatim to pairs $R_0, R_1$ and $R_1, R_2$, giving that any element of $\sN^\ell_k$ is homeomorphic to a disk.

Finally, the completed induction for $k=1$ finishes the proof of Theorem~\ref{hmt}. 
\end{proof}

From Theorem~\ref{hmt} it is clear that in a orthogonal decomposition with topological regularity, a $k$-nucleus, namely a connected component of $F(2^{\mathbb{Z}^2})\setminus \bigcup_{m\geq k}\partial\sR^m_1$, is exactly $A-\partial A$ for an $k$-atom. It thus follows that all $k$-atoms are also homeomorphic to disks.

\begin{cor}\label{cor:atomdisk} There is an orthogonal decomposition of $F(2^{\mathbb{Z}^2})$ in which, for any $k$, all $k$-atoms are homeomorphic to disks.
\end{cor}

We do not yet know how to generalize Corollary~\ref{cor:atomdisk} to higher dimensions. One possible approach is to prove a version of Lemma~\ref{lem:no4corners} in higher dimensions so that, from a local point of view, an arbitrary assignment of the $\tR^m_k$ regions to one of the $\sR^m_{k+1}$ regions will result in the marker regions obtained homeomorphic to disks, as in the 2-dimensional construction in this section. This boils down to a finitary problem of whether such preliminary decompositions always exist in dimension $3$ and above (e.g. whether one can decompose an arbitrary finite $n$-dimensional rectangular polygon so that no point is close to $n+2$ different regions). It seems that if this preliminary decomposition is possible, then the rest of the method used in this section can be applied to obtain a generalization for higher dimensions. Of course, Corollary~\ref{cor:atomdisk} suffices for our results in the rest of this paper.

\section{Borel Perfect Matchings in $F(2^{\Z^n})$} \label{sec:matching}
In this section we use the method of orthogonal decompositions to
show that there is a Borel perfect matching for $F(2^{\Z^n})$ for $n \geq 2$. By a {\em Borel perfect
matching} we mean a Borel complete section $B\subseteq F(2^{\Z^n})$ and a Borel function
$f \colon B\to \{ 0,\dots, n-1\}$ such that the pairs $(x,e_{f(x)}\cdot x)$ for $x \in B$
partition the space $F(2^{\Z^n})$. 

More generally we consider an free Borel action of $\Z^n$ on a Polish space $X$ and its (undirected) Schreier graph associated
with the standard generating set $\{e_0, \dots, e_{n-1}\}$ of $\Z^n$. A Borel perfect matching is similarly defined. We show the following theorem.

\begin{thm} \label{thm:mt}
For any $n \geq 2$ and any free Borel action of $\Z^n$ on a Polish space $X$,
there is a Borel perfect matching for the Schreier graph on $X$.
\end{thm}

In order to prove Theorem~\ref{thm:mt} it is clearly enough to consider the case $n=2$.
The proof for a general free action of $\Z^2$ is identical to the proof for the shift action
of $\Z^2$ on $X=F(2^{\Z^2})$, so we consider this case.

\subsection{Charges, corner functions, and finite Stokes theorems}

We first prove a combinatorial result for $\Z^2$ itself. In fact, this result can be extended to
$\Z^n$, which we give, though we only need the case $n=2$. The result is a sort of ``finite
version'' of Stokes theorem. We first present the $n=2$ case.

For $x=(i,j) \in \Z^2$, let $q(x)=(-1)^{i+j}$ and call it the {\em charge} of $x$. 
For $x=(i,j), y=(i',j') \in \Z^2$, let $\rho(x,y)=\max\{ |i-i'|,|j-j'|\}$
denote the $\ell^\infty$ distance between $x$ and $y$. If $F\subseteq \Z^2$,
let $\partial F= \{ x \in F\colon \rho(x, \Z^2-F)=1\}$.

For $F\subseteq \Z^2$ we define a {\em corner function} $c \colon F\to \Z$
as follows (the name derives from the fact that for a rectangular polygon in $\Z^2$
the function $c(x)$ is $0$ except when $x$ is a corner point of the boundary). 
For $x \in F$ we set
\[
c(x)=4- 2e(x) +s(x),
\]
where $e(x)$ is the number of horizontal and vertical edges that connect $x$ to a point of
$F$, that is, $e(x)=| \{ p \in \{ (\pm 1,0), (0,\pm 1)\} \colon x+p \in F\}|$.
Also $s(x)$ is the number of points $y$ of the form $x+ (\pm 1, \pm 1)$ which are in $F$
and are such that the two points of $\Z^2$ which are adjacent to both $x$ and $y$
are also in $F$.

The following is our ``finite Stokes theorem'' for finite $F\subseteq \Z^2$.

\begin{thm} [Finite Stokes] \label{thm:stokes}
Let $F \subseteq \Z^2$ be finite. Then we have
\[
\sum_{x\in F} q(x)= \frac{1}{4} \sum_{x \in \partial F} q(x)c(x).
\]
\end{thm}

\begin{proof}
For $x \in F-\partial F$, we have that $c(x)=0$, and so
\[
\sum_{x \in \partial F} q(x)c(x)=\sum_{x \in F} q(x)c(x).
\]
So, it suffices to show that $\sum_{x \in F} (4q(x)- q(x)c(x))=0$.
From the definition of $c(x)$ we have that
$\sum_{x \in F} (4q(x)- q(x)c(x))= \sum_{x \in F} (2q(x)e(x)-q(x)s(x))$.
It suffices to show that each of the sums $\sum_{x \in F} q(x)e(x)$
and $\sum_{x \in F} q(x)s(x)$ is equal to $0$.

Let $E$ be the set of pairs of points (``edges")  $(x,y)$ with
$x,y \in F$ and $x,y$ differing by $(\pm 1,0)$ or $(0, \pm 1)$.
Then $\sum_{x \in F} q(x)e(x)= \sum_{(x,y)\in E} (q(x)+q(y))=0$
as $q(x)+q(y)=0$ for all $(x,y)\in E$.

Let $S$ be the set of ``squares'' in $F$, that is, $S$ is the set of all
$4$-element subsets of $F$ of the form $R=\{ (i,j), (i+1,j), (i,j+1), (i+1,j+1)\}$
for some $i, j$. Then we have
$\sum_{x \in F} q(x)s(x)= \sum_{R\in S} \sum_{(a,b)\in R} (-1)^{a+b}=0$
as for each square $R$ in $F$ we have that $\sum_{(a,b)\in R} (-1)^{a+b}=0$.

This completes the proof of Theorem~\ref{thm:stokes}.
\end{proof}

Although Theorem~\ref{thm:stokes} is all we need for the results of this paper, 
we state the natural generalization to $\Z^n$ for any $n\geq 2$. If $k\leq n$, by a {\em $k$-dimensional $1$-cube} in $\Z^n$ we mean
a set of the form
$$ C=\{x=(x_1,\dots, x_n)\in \Z^n\,:\, x_i=\pm 1 \mbox{ for $i\in I$ and } x_i=0 \mbox{ for $i\not\in I$}\} $$
for some $I\subseteq\{1,\dots, n\}$ with $|I|=k$, i.e., 
$C$ is a set with $2^k$ many elements $x=(x_1,\dots, x_n) \in \Z^n$ such that any non-zero coordinate of $x$ 
must occur in one of the coordinates $i_1,\dots, i_k$ where $1\leq i_1 <\cdots <i_k \leq n$, 
and any non-zero coordinate of $x$ must be  $\pm 1$. There are $2^k \binom{n}{k}$ many distinct
$k$-dimensional $1$-cubes in $\Z^n$.

\begin{thm} \label{thm:genstokes}
Let $F\subseteq \Z^n$ be finite. Then 
\[
\sum_{x \in F} q(x)= \frac{1}{2^n} \sum_{x \in \partial F} q(x)c(x)
\]
where $q(x)=\sum_{i=1}^n (-1)^{x_i}$ and 
$c(x)=\sum_{i=0}^n (-1)^i 2^{n-i} c_i(x)$ where $c_0(x)=1$, and 
for $i>0$, $c_i(x)$ is the number of $i$-dimensional $1$-cubes $C$ 
such that $C+x\subseteq F$. 
\end{thm}

The proof of Theorem~\ref{thm:genstokes} is an immediate generalization of 
that of Theorem~\ref{thm:stokes}.

Returning to $\Z^2$, 
we will need the following simple lemma about the $\sum q(x)c(x)$ expression
of Theorem~\ref{thm:stokes}.

\begin{figure}[!h] 
\begin{tikzpicture}[scale=0.04]

\pgfmathsetmacro{\a}{5}
  
\draw[thick] (0,0) to (80,0) to (80,60) to (120,60) to (120,100);
\draw[thick] (0,\a)to (80-\a, \a) to (80-\a, 60+\a) to (120-\a, 60+\a) to (120-\a, 100)   ;

\node(p) at (40,-5) {$p$};
\node(p) at (40,\a+5) {$p'$};

\end{tikzpicture}
\caption{Adjacent paths $p, p'$ as in Lemma~\ref{lem:pp}} \label{fig:pp}
\end{figure}
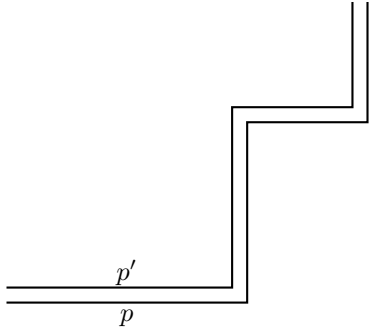

\begin{lem} \label{lem:pp}
Let $p$ be a simple (i.e. non-self-intersecting) path consisting of a concatenation of horizontal and vertical
segments of points in $\Z^2$. Assume each segment has length at least $5$.
Let $p'$ be a path disjoint from $p$ which is contained in $B_{\rho_\infty}(p,1)$
(the set of points of distance $1$ from $p$ in the $\rho_\infty$ metric), and which is
adjacent to $p$ as shown in Figure~\ref{fig:pp}. Fix an orientation for $p$ so that
$p'$ lies to the left of $p$. 
Let $s=\sum_{x \in p} q(x)c(x)$,
where the values of $c(x)$ are computed for the region $R=\Z^2\sm p'$ and let 
$s'=\sum_{x \in p'} q(x)c'(x)$, where $c'(x)$ is computed for the region $R'=\Z^2\sm p$.
Then $s=-s'$. 

\end{lem}

\begin{proof}
For each of the pairs $(x,x')$ of corresponding corner points for $p$, $p'$
we have that $q(x)=q(x')$, and it is easy to check that $c(x)=-c'(x')$,
and the result follows.
\end{proof}

\subsection{Currents and extensions of partial matchings}

We now discuss some terminolgy and simple results concerning matchings of finite regions in $\mathbb{Z}^2$. 

The regions we consider will be rectangular polygons that are in between two ``parallel" paths. To begin, we consider the simplest kind of such region, which is a region strictly between two parallel (horizontal or vertical) lines in $\Z^2$. For definiteness, let us consider two parallel horizontal lines $y=\alpha$ and $y=\beta$ (where $\alpha<\beta$ are in $\Z$). The region will be defined as
$$ R=\{(x,y)\in \Z^2\,:\, a\leq x\leq b, \ \alpha<y<\beta\} $$
for some $a<b$ in $\Z$. In general $b-a$ is much larger than $\beta-\alpha$, and we thus refer to $\beta-\alpha-1$ (the number of integers strickly between $\alpha$ and $\beta$) as the {\em thickness} of $R$. For reasons that will become clear, we assume the thickness of $R$ is even.

We will be considering matchings of the induced subgraph of the Cayley graph on $R$. To emphasize that these are not necessarily perfect mathchings, we refer to them as {\em partial matchings}. If $M$ is a partial matching of $R$, then $M$ is a subgraph of the Cayley graph on $R$ such that each vertex has degree 0 or 1, i.e., $\mbox{deg}_M(x)=0 \mbox{ or } 1$ for $x\in R$. We call the set $\{x\in R\,:\, \mbox{deg}_M(x)=1\}$ the {\em domain} of $M$, and denote it by $\dom(M)$.

A partial matching of $R$ is said to be of the
{\em canonical form} if it is as shown in Figure~\ref{fig:cm}, i.e., all edges of the matching are horizontal, and the domain of the matching is either the entire $R$ (if $b-a+1$ is even) or all of $R$ except those points with $x$-coordinate $b$ (as shown, if $b-a+1$ is odd).

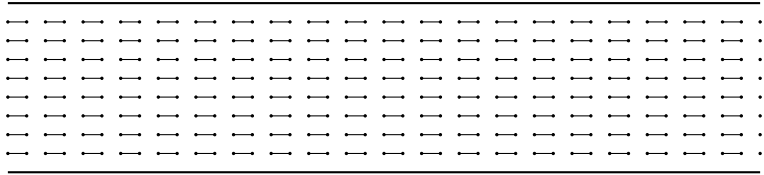
\begin{figure}[h] 
\begin{tikzpicture}[scale=0.05]

\pgfmathsetmacro{\a}{5}
  
\draw[thick] (-100,0) to (100,0);
\draw[thick] (-100,45)to (100, 45);

\foreach \i in { 0,..., 40}
\foreach \j in {1,..., 8}
         {\draw[fill=black] (-100+ 5*\i, 5*\j) circle (0.3);
           }

\foreach \i in { 0,..., 19}
\foreach \j in {1,..., 8}
         {\draw (-100+10 *\i, 5*\j) to (-100+ 10*\i+5,5*\j);
           }

\end{tikzpicture}
\caption{The canonical matching.} \label{fig:cm}
\end{figure}

Consider a more general partial matching $M$ of $R$. Let us fix an orientation for this region, say left to right. Consider
$$ S=\{(x, y)\in R\,:\, a<x<b, \ \alpha<y<\beta\}. $$
Suppose the domain of $M$ includes all points of $S$. For all $a<x<b$, we define the {\em current} of $M$ at $x$,
which we denote $C_M(x)$ or just $C(x)$ if $M$ is understood. 

For $a<x<b$, let $S_x=\{ (x, y)\,:\, \alpha<y<\beta\}$. For each $z\in S_x$,
if $M$ matches it with a point to the right of $z$, or with a point below or above $z$, then $z$ contributes $0$ to the current. If $M$ matches $z$ with a point to the
left of $z$, then $z$ contributes $q(z)$ to the current. Summing these values for the points in
$S_x$ gives the current $C_M(x)$. Note that since $R$ has even thickness, the canonical matching has current $0$ at all $x$.

The current is well-defined by the following fact.

\begin{lem} \label{curcon}
Suppose $M$ is a partial 
matching of $R$ whose domain includes all points of $S$. Then for any $a < x_0 \leq x_1 <b$ we have that
$C_M(x_0)=C_M(x_1)$.
\end{lem}

\begin{proof} It suffices to prove the lemma for $x_0=a+1$ and $x_1=b-1$.
Since $R$ has even thickness, we have $\sum_{z\in R} q(z)=0$. Also
$\sum_{z\in \dom(M)} q(z)=0$, and so, $\sum_{z\in R-\dom(M)}q(z)=0$. Since $S\subseteq \dom(M)$, we get 
$$\sum_{(a,y)\in R-\dom(M)} q(a,y)
+ \sum_{(b,y)\in R-\dom(M)} q(b,y)=0.$$ 
We claim that the first sum equals the current $C_M(x_0)$. This is because, for each $(a, y)\in R-\dom(M)$, $M$ matches $(a+1,y)=(x_0,y)$ with a point either to the right, or above and below, and thus $(x_0,y)$ contributes $0$ to $C_M(x_0)$; on the other hand, for each $(a, y)\in\dom(M)$, $M$ matches $(x_0,y)$ to $(a, y)$, which is to the left of $(x_0,y)$, and thus $(x_0,y)$ contributes $q(x_0,y)=-q(a,y)$ to $C_M(x_0)$. Therefore 
$$ C_M(x_0)=\sum_{(a,y)\in \dom(M)} -q(a,y). $$
Note again $R$ has even thickness, and thus $\sum_{(a,y)\in R}q(a,y)=0$. It follows that
$$ C_M(x_0)=\sum_{(a,y)\in R-\dom(M)} q(a, y). $$
Now we claim that the second sum equals $-C_M(x_1)$. To see this, again note $\sum_{(x_1,y)\in R}q(x_1,y)=0$ and write the sum as
$$\sum_{(x_1,y)\in A} q(x_1,y)+ \sum_{(x_1,y) \in B} q(x_1,y) + \sum_{(x_1,y)\in C} q(x_1,y)=0$$ 
where $A$ is the set of points which are matched by $M$ vertically to a point (again in $A$),
$B$ is the set of points which are matched by $M$ to a point to the left, and $C$ is the set of points which are matched by $M$ to a point to the right. The first sum equals 0, and the second sum equals $C_M(x_1)$.  For each $(x_1,y)\in C$, $(x_1+1,y)=(b,y)$ is matched by $M$ to $(x_1, y)$, which is to the left of $(b,y)$, and $q(x_1,y)=-q(b,y)$. Let $D$ be the set of all $(b, y)\in R$ which are matched by $M$ to the point on the left. Thus it suffices to show that
$$ \sum_{(b,y)\in D} q(b,y)+\sum_{(b,y)\in R-\dom(M)}q(b,y)=0. $$
Finally, let $E$ be the set of all points $(b,y)\in R$ which are matched by $M$ vertically to a point (again in $E$). Since $\sum_{(b,y)\in R} q(b,y)=0$, we get that
$$ \sum_{(b,y)\in R-\dom(M)}q(b,y)+\sum_{(b,y)\in D}q(b,y)+\sum_{(b,y)\in E}q(b,y)=0. $$
Note that the third term is $0$, thus we get the desired equality. In conclusion, we get
$C_M(x_0)=C_M(x_1)$. 
\end{proof}

In Figure~\ref{fig:curr} we give an example of a partial matching of $R=\{(x, y)\in \Z^2\,:\, 0\leq x\leq 40,\ 0<y<9\}$ (with the left-to-right orientation) whose current is $1$.

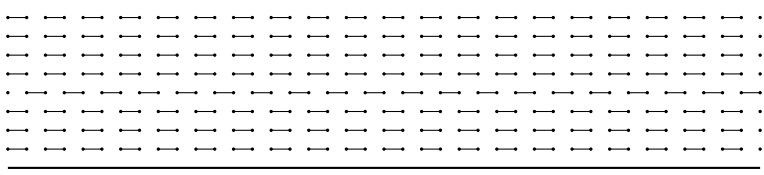
\begin{figure}[h] 
\begin{tikzpicture}[scale=0.05]

\pgfmathsetmacro{\a}{5}
  
\draw[thick] (-100,0) to (100,0);
\draw[thick] (-100,45)to (100, 45);

\foreach \i in { 0,..., 40}
\foreach \j in {1,..., 8}
         {\draw[fill=black] (-100+ 5*\i, 5*\j) circle (0.3);
           }

\foreach \i in { 0,..., 19}
\foreach \j in {1,2,3,5,6,7,8}
         {\draw (-100+10 *\i, 5*\j) to (-100+ 10*\i+5,5*\j);
           }

         \foreach \i in { 1,..., 20}
{\draw (-100+10 *\i-5, 5*4) to (-100+ 10*\i,5*4);
           }

\end{tikzpicture}
\caption{A partial matching of current $1$. } \label{fig:curr}
\end{figure}

 In general, we call a partial matching $M$ of a region 
$$ R=\{(x,y)\in \Z^2\,:\, a\leq x\leq b,\ \alpha<y<\beta\} $$
{\em full} if the domain of $M$ includes all points of the set $$\{(x, y)\in\Z^2\,:\, a<x<b, \ \alpha<y<\beta\}.$$ Lemma~\ref{curcon} guarantees that the current is well-defined for full partial matchings.

We next present some lemmas about partial matchings with a given current. We first prove a lemma that guarantees a ``standard form" for a partial matching with a given current. In the lemma we tacitly assume that the orientations of the regions are all left-to-right.

\begin{lem} \label{lem:standard}
Suppose $M$ is a full partial matching of the region 
$$R=\{(x,y)\in\Z^2\,:\, a\leq x\leq b,\ \alpha<y<\beta\}, $$
where the thickness $w=\beta-\alpha-1$ is even. Suppose $M$ has current $k$. There is an absolute constant $C_0$
such that if $b'>b+C_0w^2$, then we can extend $M$ to a full partial matching $M'$ of 
$$R'=\{(x,y)\in\Z^2\,:\, a\leq x\leq b',\ \alpha<y<\beta\} $$ 
such that $(b',y)\in \dom(M')$ iff $y\in L\cup A$, where $L\in\{L_0, L_1\}$ and $A\in\{A_0,A_1\}$, and
$$\begin{array}{rcl}
L_0 &=&\varnothing, \\
L_1 &=&\{ y\,:\, \alpha+2|k|+1\leq y<\beta\}, \\
A_0  &=&\{ \alpha+2i-1\,:\, 1\leq i\leq |k|\}, \\
A_1 &=&\{\alpha+2i\,:\, 1\leq i\leq |k|\}. 
\end{array}
$$

\end{lem}

\begin{proof}
We describe an algorithm to extend $M$ to a full partial matching $M'$ of $R’$ by following the left-to-right orientation. First consider the right-edge of $R$, that is, the points with $x$-cooordinate $x=b$.
By extending $M$ by adding horizontal edges, we 
may assume without loss of generality that
all of the points on the right-edge of $M$ are eventually matched to points to their
left or right (that is, $M$ contains no vertical edges between points on the right-edge of $R$, and for those points on the right-edge of $R$ for which $M$ is undefined, $M'$ will match them to the points on the right). There is also no loss of generality (by extending $M$ further by adding horizontal edges)
in assuming that the top point of the right-edge of $R$ is included in the domain of $M$, that is, it is matched by $M$ to the point to its left. 

In the first step of the algorithm, define $M'$ for all points on the right-edge of $R$ in an obvious way, that is, if $z$ is a point of the right-edge of $R$ so that $z\not\in\dom(M)$, then $M'$ matches $z$ to the point on its right. Along the right-edge of $R$, from the top to the bottom, we use the letter $l$ to represent the point if $M'$ matches it to the point on the left, and use the letter $r$ to represent the point if $M'$ matches it to the point on the right. Thus we obtain a sequence of $l$'s and $r$'s to represent the pattern of matching $M'$ on the right-edge. From our assumption above, the sequence starts with an $l$. 

Consider first the case where the left/right matching pattern from top to bottom
along the right-edge of $R$ is of the form 
$$\underbracket{ll \cdots l}_{\mbox{\small $i$}}\ \underbracket{lrl\cdots\cdots rlr}_{\mbox{\small $t$ alternations}}\ \underbracket{r\cdots rr}_{\mbox{\small $j$}}$$ or
$$\underbracket{ll \cdots l}_{\mbox{\small $i$}}\ \underbracket{lrl\cdots\cdots lrl}_{\mbox{\small $t$ alternations}}\ \underbracket{l\cdots ll}_{\mbox{\small $j$}}.$$ That is, the pattern begins and ends with a block
of $l$'s or $r$'s, and in-between is a single block of alternations of $l$'s and
$r$'s with a total of, say, $t$ alternations. In Figure~\ref{fig:oa}, if we take 
$b=1$ (that is, we read the pattern
from the second column of points) then we see the pattern $llrlrlll$ with $t=4$ alternations. 
Note that the current in this case is equal to $2$. Figure~\ref{fig:oa} illustrates a procedure to extend $M$ so that in the extended partial matching, the left/right matching pattern of the right edge becomes $llllrlrl$.

\begin{figure}[h] 
\begin{tikzpicture}[scale=0.05]

\pgfmathsetmacro{\a}{5}
  
\draw[thick] (-100,0) to (-45,0);
\draw[thick] (-100,45)to (-45, 45);

\foreach \i in { 0,..., 11}
\foreach \j in {1,..., 8}
         {\draw[fill=black] (-100+ 5*\i, 5*\j) circle (0.3);
           }

\foreach \i in { 0}
\foreach \j in {1,2,3,5,7,8}
         {\draw (-100+5 *\i, 5*\j) to (-100+ 5*\i+5,5*\j);
           }

         \foreach \i in { 1}
         \foreach \j in {4,6}
{\draw (-100+5 *\i, 5*\j) to (-100+ 5*\i+5,5*\j);
           }

\foreach \i in { 2}
\foreach \j in {1,2,3,5,7,8}
         {\draw (-100+5 *\i, 5*\j) to (-100+ 5*\i+5,5*\j);
           }

\draw (-100+3*\a, 4*\a) to (-100+3*\a+\a, 4*\a);
\draw (-100+3*\a, 6*\a) to (-100+3*\a+\a, 6*\a);
\draw (-100+4*\a, 1*\a) to (-100+4*\a+\a, 1*\a);
\draw (-100+4*\a, 5*\a) to (-100+4*\a+\a, 5*\a);
\draw (-100+4*\a, 7*\a) to (-100+4*\a+\a, 7*\a);
\draw (-100+4*\a, 8*\a) to (-100+4*\a+\a, 8*\a);  
         
\draw (-100+5*\a, 2*\a) to (-100+5*\a+\a, 2*\a);  
\draw (-100+5*\a, 6*\a) to (-100+5*\a+\a, 6*\a);

\draw (-100+6*\a, 1*\a) to (-100+6*\a+\a, 1*\a);
\draw (-100+6*\a, 3*\a) to (-100+6*\a+\a, 3*\a);
\draw (-100+6*\a, 7*\a) to (-100+6*\a+\a, 7*\a);
\draw (-100+6*\a, 8*\a) to (-100+6*\a+\a, 8*\a);

\draw (-100+7*\a, 2*\a) to (-100+7*\a+\a, 2*\a);
\draw (-100+7*\a, 4*\a) to (-100+7*\a+\a, 4*\a);

\draw (-100+8*\a, 1*\a) to (-100+8*\a+\a, 1*\a);
\draw (-100+8*\a, 3*\a) to (-100+8*\a+\a, 3*\a);
\draw (-100+8*\a, 5*\a) to (-100+8*\a+\a, 5*\a);
\draw (-100+8*\a, 6*\a) to (-100+8*\a+\a, 6*\a);
\draw (-100+8*\a, 7*\a) to (-100+8*\a+\a, 7*\a);
\draw (-100+8*\a, 8*\a) to (-100+8*\a+\a, 8*\a);

\foreach \i in {0,...,3}
         { \draw[color=green] (-100+4*\a +\i*\a, 2*\a+\i*\a) to (-100+4*\a +\i*\a, 2*\a+\i*\a+\a);
           }

\draw (-100+ 9*\a,2*\a) to (-100+ 9*\a+\a,2*\a);
\draw (-100+ 9*\a,4*\a) to (-100+ 9*\a+\a,4*\a);

\end{tikzpicture}
\caption{Moving the alternations down. } \label{fig:oa}
\end{figure}
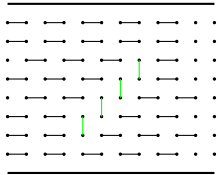

Consider the set $A$ of $t+1$
vertically consecutive points on the right-edge of $R$ which correspond to the block of $t$ alternations.
If $A$ is followed by $j\geq 2$ terms in the sequence (i.e., there are at least two
points vertically below the last point), then an extension of $M$ as shown in
Figure~\ref{fig:oa} will result in a similar pattern along its right edge except the block $A$
has been shifted two to the right in the pattern (that is, two steps down along the right edge).
In this figure, the pattern $llrlrlll$ (taking $b=1$)  has been shifted to $llllrlrl$ 
(in this case the alternating block
$A$ has been shifted to the bottom of the pattern).  Repeating this process, we may move the alternating block
down until it either terminates the pattern (e.g., $llllrlrl$) or ends at the
next to last position (e.g., $lllrlrll$). The above argument shows that in the second case
we may also move the alternating blocks to the right by two, so for example
$lllrlrll$ becomes $lllllrlr$). So, a single aternating block may always be moved to
a tail of the pattern (in the last case, the size of the final alternating block has decreased by one). The horizontal distance needed to implement a single step of the procedure illustrated in Figure~\ref{fig:oa} is linear in $t$. The number of iterations needed to move the block of alternations to the tail is approximately $\frac{j}{2}$. Thus the horizontal distance needed to perform the entire procedure is linear in $w^2$.

In the general case, the pattern will consist of an intial segment of $l$'s (without
loss of generality) and then a number of alternating blocks $A_1,\dots, A_p$
which are separated by constant blocks of length at least one. We then proceed as above,
successively moving the blocks down to either form a larger alternating tail segment,
or else resulting in the concatenation of two alternating blocks.
For example, we could change $lrllllrl$ into $lllrllrl$ which ends with
two alternating blocks of the form $lrl$. However, a variation of the previous argument
shows that a concatenation of alternating blocks may be changed to end in a single
alternating block (perhaps of length $0$). In this example,  it is easy to see that $lllrllrl$ may
be changed to $llllllll$. Continuing this process, we change the initial pattern
to one which is a constant block of $l$'s followed by a single alternating block.
Note the length of the final terminating alternating block is uniquely given by the
conservation of current, Lemma~\ref{curcon}. Again, the horizontal distance needed to 
implement this procedure is linear in $w^2$.

So, we may extend $M$ to the right to some partial matching $N$ so that the pattern of the right edge
of $N$ is $ll\cdots l A$, where $A$ is an alternating block starting with $r$.
Extending $N$ to the entire region $R‘$ by only adding horizontal edges, we obtain a full partial matching $M'$ of $R'$. The pattern $ll\cdots lA$ shows up either on  the right-edge of $R'$ or on the column immediately to the left of the right-edge. In the first case, we have $(b',y)\in \dom(M')$ exactly when $y\in L_0\cup A_0=A_0$ or exactly when $y\in L_0\cup A_1=A_1$ (which is uniquely determined by the conservation of current). In the second case, we have $(b',y)\in \dom(M')$ exactly when $y\in L_1\cup A_0$ or exactly when $y\in L_1\cup A_1$. 
\end{proof}

The next lemma says that for a given current $k$, we may extend a partial matching
with current $k$ to ``join" any other partial matching of current $k$.

\begin{lem} \label{lem:ua}
Suppose $M$ is a full partial matching of the region 
$$R=\{(x,y)\in\Z^2\,:\, a\leq x\leq b,\ \alpha<y<\beta\}, $$
where the thickness $w=\beta-\alpha-1$ is even. Suppose $M$ has current $k$. Let $C_0$ be the absolute constant given by Lemma~\ref{lem:standard}. If $b'> a'>b+2C_0 w^2$, and $M'$ is a full partial matching of 
$$R'=\{(x,y)\in\Z^2\,:\, a'\leq x\leq b',\ \alpha<y<\beta\} $$ 
with current $k$, then we can extend $M \cup M'$
to a full partial matching of the region 
$$ \hat{R}=\{(x,y)\in \Z^2\,:\, a\leq x\leq b',\ \alpha<y<\beta\}. $$
\end{lem}

\begin{proof} 
By Lemma~\ref{lem:standard} we may extend $M$ to the right to some partial matching $N$ so that the pattern on the right edge of $N$ is of the form $ll\cdots lA$, where $A$ is a single alternating block of $l$'s and $r$'s. Symmetrically, we may apply Lemma~\ref{lem:standard} to $M'$ and extend it to the left to some partial matching $N'$ so that the pattern on the left edge of $N'$ is of the form
$ll\cdots l A'$, where again $A'$ is an alternating block.
From Lemma~\ref{curcon} it follows that $A$ and $A'$ have the same length and that
if the right edge of $N$ has $x$-coordinate $e$ and the left edge of $N'$
has $x$-coordinate $f$, then $f-e$ is even. Note here that if the right edge pattern
is $l^{2i}  A$ where $A=rl\cdots rl$ is of even length $w-2i$,
the current at this vertical edge is $\frac{1}{2} (w-2i)$ (asssume 
the overall parity is such that $q(z)=-1$ for the top point of the right edge), and 
if the right edge has pattern $l^{2i+1} A$ where $A=rl\cdots r$ has odd length $w-2i-1$,
then the right edge has current $-\frac{1}{2}(w-2i)$.

We may fill in the remaining
region of points with $x$-coordinates between $e$ and $f$ to a full partial matching
extending both $M$ and $M'$. 
\end{proof}

So far in this subsection we have carefully considered the region 
$$ R=\{(x, y)\in\Z^2\,:\, a\leq x\leq b, \ \alpha<y<\beta\} $$ 
where the thickness $w=\beta-\alpha-1$ is even and the orientation is left-to-right. All of our results can be generalized to the case where the orientation is right-to-left, and to the case where
$$ R=\{(x,y)\in\Z^2\,:\, \alpha<x<\beta,\ a\leq y\leq b\}, $$
the thickness $w=\beta-\alpha-1$ is even and the orientation is either top-to-bottom or bottom-to-top. In the next lemma, we consider a slightly more complex region, where we have to ``turn a corner." We show that there is no difficulty in generalizing our results to this case.
The statement and proof are illustrated in Figure~\ref{fig:ob}.

\begin{figure}[h] 
\begin{tikzpicture}[scale=0.05]

\pgfmathsetmacro{\a}{5}
  
\draw[thick] (-30,0) to (30,0);
\draw[thick] (30,0) to (30,50);

\draw[thick] (-30,35)to (-5, 35);
\draw[thick] (-5,35) to (-5, 50);

\foreach \i in { 0,..., 8}
\foreach \j in {1,..., 6}
         {\draw[fill=black] (-15+ 5*\i, 5*\j) circle (0.3);
           }

\foreach \i in {0,...,5}
\foreach \j in {7,...,9}
{\draw[fill=black] (\a*\i, \a*\j) circle (0.3);
           }

\draw (-2*\a, \a) to ( -\a,\a);
\draw (-2*\a, 2*\a) to ( -\a,2*\a);
\draw (-2*\a-\a, 3*\a) to ( -\a-\a,3*\a);
\draw (-2*\a-\a, 4*\a) to ( -\a-\a,4*\a);
\draw (-2*\a, 5*\a) to ( -\a,5*\a);
\draw (-2*\a-\a, 6*\a) to ( -\a-\a,6*\a);

\draw[color=green] (0*\a,\a) to (1*\a,\a);
\draw[color=green] (2*\a,\a) to (3*\a,\a);
\draw[color=green] (4*\a,\a) to (5*\a,\a);
\draw[color=green] (5*\a,2*\a) to (5*\a,3*\a);
\draw[color=green] (5*\a,4*\a) to (5*\a,5*\a);
\draw (5*\a,6*\a) to (5*\a,7*\a);

\draw[color=green] (0*\a,2*\a) to (1*\a,2*\a);
\draw[color=green] (2*\a,2*\a) to (3*\a,2*\a);
\draw[color=green] (4*\a,2*\a) to (4*\a,3*\a);
\draw[color=green] (4*\a,4*\a) to (4*\a,5*\a);
\draw (4*\a,6*\a) to (4*\a,7*\a);

\draw[color=green] (-1*\a,3*\a) to (0*\a,3*\a);
\draw[color=green] (1*\a,3*\a) to (2*\a,3*\a);
\draw[color=green] (3*\a,3*\a) to (3*\a,4*\a);
\draw[color=green] (3*\a,5*\a) to (3*\a,6*\a);

\draw[color=green] (-1*\a,4*\a) to (0*\a,4*\a);
\draw[color=green] (1*\a,4*\a) to (2*\a,4*\a);
\draw[color=green] (2*\a,5*\a) to (2*\a,6*\a);

\draw[color=green] (0*\a,5*\a) to (1*\a,5*\a);
\draw  (1*\a,6*\a) to (1*\a,7*\a);
\draw[color=green] (-1*\a,6*\a) to (0*\a,6*\a);

\draw  (0*\a,7*\a) to (0*\a,8*\a);
\draw  (2*\a,7*\a) to (2*\a,8*\a);
\draw  (3*\a,7*\a) to (3*\a,8*\a);
\draw  (1*\a,8*\a) to (1*\a,9*\a);
\draw  (4*\a,8*\a) to (4*\a,9*\a);
\draw  (5*\a,8*\a) to (5*\a,9*\a);

\end{tikzpicture}
\caption{Turning a corner. } \label{fig:ob}
\end{figure}
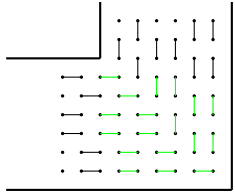

\begin{lem} \label{lem:ub}
Suppose $w>0$ is even, and $a, b>1$. Consider the region 
$$\begin{array}{rl} R'=\{(x,y)\in\Z^2\,:\,& (0 \leq x \leq a+w,\ 0 < y <w+1) \mbox{ or } \\
& (a < x <a+w+1,\ w+1\leq y\leq w+b+1)\} \end{array}$$ as shown in Figure~\ref{fig:ob}.
Let 
$$S=\{(x,y)\in\Z^2\,:\, 0<x <a,\ 0< y<w+1\},$$ 
$$S'=\{(x,y)\in\Z^2\,:\, 0\leq x\leq a, \ 0<y<w+1\}, $$ and 
$$\begin{array}{rl} R=\{(x,y)\in\Z^2\,:\,& (0 < x \leq a+w,\ 0 < y <w+1) \mbox{ or } \\
& (a < x <a+w+1,\ w+1\leq y< w+b+1)\}. \end{array}$$ 
Suppose $M$ is a
partial matching with $S\subseteq \dom(M)\subseteq S'$ which consists of only horizontal edges. Then $M$ can be extended to a partial matching $\hat{M}$
with $R\subseteq \dom(\hat{M})\subseteq R'$. Moreover, assuming the orientation of $R'$ is left-to-right followed by bottom-to-top, the current of $\hat{M}$ at $y=w+b$ is equal to the current of $M$ at $x=1$.
\end{lem}

\begin{proof}
For each point $(a,y)$ on the right edge of $S'$, we extend the partial matching $M$ by following
the horizontal line through $(a,y)$ to the point $(a+w-y+1,y)$ on the diagonal line between
the two corner points, and then vertically
to the point $(a+w-y+1, w+b+1)$. Since the point $(a,y)$ on the right edge of $S'$ have the same parity
as the corresponding point $(a+w-y+1,w+1)$, this results in the desired matching. In particular, the current of $\hat{M}$ at $x=a$
is equal to the current of $M$ at $x=1$ by Lemma~\ref{curcon}, and by the above observation about the parities of points, it is equal to the current of $\hat{M}$ at $y=w+1$, which is again equal to the current of $\hat{M}$ at $y=w+b$ by Lemma~\ref{curcon}.
\end{proof}

The next lemma says that a partial matching with a certain current can absorb a given charge,
with a resulting change in the current.

\begin{figure}[h] 
\begin{tikzpicture}[scale=0.05]

\pgfmathsetmacro{\a}{5}
  
\draw[thick] (-5,0) to (60,0);
\draw[thick] (-5,35) to (60,35);

\foreach \i in { 0,..., 11}
\foreach \j in {1,..., 6}
         {\draw[fill=black] (5*\i, 5*\j) circle (0.3);
           }

\foreach \j in {2,4}
{ \draw (0*\a, \j*\a) to (0*\a+\a,\j*\a);}

\foreach \j in {1,3,5,6}
{ \draw (1*\a, \j*\a) to (1*\a+\a,\j*\a);}

\foreach \j in {2,4}
{ \draw (2*\a, \j*\a) to (2*\a+\a,\j*\a);}

\foreach \j in {1,3}
{ \draw (3*\a, \j*\a) to (3*\a+\a,\j*\a);}

\foreach \j in {2}
{ \draw (4*\a, \j*\a) to (4*\a+\a,\j*\a);}

\foreach \j in {1}
{ \draw (5*\a, \j*\a) to (5*\a+\a,\j*\a);}

\foreach \i in {3,4,5,6,7}
{ \draw[color=green]  (\i*\a, 8*\a-\i*\a) to (\i*\a,9*\a-\i*\a);}

\draw[fill=black] (8*\a,\a) circle (0.6);
\draw[thick] (8*\a,0) to (8*\a,\a);

\foreach \i in {4,6,8,10}
{\draw (\i*\a,6*\a) to (\i*\a+\a,6*\a);}

\foreach \i in {5,7,9}
{\draw (\i*\a,5*\a) to (\i*\a+\a,5*\a);}

\foreach \i in {6,8,10}
{\draw (\i*\a,4*\a) to (\i*\a+\a,4*\a);}

\foreach \i in {7,9}
{\draw (\i*\a,3*\a) to (\i*\a+\a,3*\a);}

\foreach \i in {8,10}
{\draw (\i*\a,2*\a) to (\i*\a+\a,2*\a);}

\foreach \i in {9}
{\draw (\i*\a,1*\a) to (\i*\a+\a,1*\a);}

\end{tikzpicture}
\caption{Absorbing a charge.} \label{fig:oc}
\end{figure}
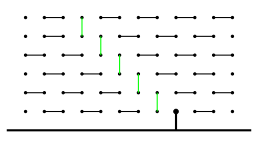

\begin{lem} \label{lem:uc}
Let 
$$ R=\{(x,y)\in\Z^2\,:\, a\leq x\leq b,\ \alpha<y<\beta\} $$
with $w=\beta-\alpha-1$ even. 
Suppose $M$ is a full partial matching of the region $R$ and suppose on the right-edge of $R$ the partial matching $M$ is of the standard form as defined in Lemma~\ref{lem:standard}. Let 
$z=(c,\alpha+1)$ be a point with $c > b+w$. Then for any $b'>c+w$, we can extend $M$ to a full partial matching $M'$
of the region 
$$ R'=\{(x,y)\in\Z^2\,:\, a\leq x\leq b',\ \alpha<y<\beta\}-\{z\} $$
so that $M'$ is of the standard form on the right-edge of $R'$.
\end{lem}

\begin{proof}
Assume the right edge of 
$M$ has the pattern $ll\cdots l A$ where $A$ is a
block of $t$ alternations of $l$'s and $r$'s. We consider two cases. The first case is
$c-b$ is even and $A$ ends with an $l$, or $c-b$ is odd and $A$ ends with an $r$. In this case Figure~\ref{fig:oc}
shows the algorithm for extending $M$ (in Figure~\ref{fig:oc} the point $z$ is shown as being ``grounded", that is, matched to a point on the boundary). The resulting partial matching is already of the standard form, with the absolute value of the current
increasing by $1$ as we pass the absorbed charge. In the example
of Figure~\ref{fig:oc}, which is the non-trivial case, if the overall parity 
is chosen so that $q(z)=1$, then the current changes from $+2$ to $+3$ 
as we move past the new charge $z$. The second case is $c-b$ is even and $A$ ends with an $r$, or $c-b$ is odd and $A$ ends with an $l$. In this case, we only need to extend $M$ by adding horizontal edges up to $x=c$ to absorb the charge, and the absolute value of the current decreases by $1$. Now this partial matching has the pattern $ll\cdots lA'll$ or $ll\cdots lA'rr$, where $A'$ is a block of $t-2$ alternations. We apply the moving-the-alternations-down algorithm in the proof of Lemma~\ref{lem:standard} to obtain the standard form $ll\cdots lA'$. Since this algorithm takes at most $t$ steps, we obtain the conclusion of the lemma with any $b'>c+w$.
\end{proof}

The next lemma says that a partial matching can absorb a number of given charges of the same sign, with no restrictions on the distances between them, provided that the number of charges is small compared to the thickness of the region.

\begin{lem} \label{lem:ucc}
Let 
$$ R=\{(x,y)\in\Z^2\,:\, a\leq x\leq b,\ \alpha<y<\beta\} $$
with $w=\beta-\alpha-1$ even. 
Suppose $M$ is a full partial matching of the region $R$ and suppose on the right-edge of $R$ the partial matching $M$ is of the standard form as defined in Lemma~\ref{lem:standard}. Suppose $M$ has current $k$. Let $l<\frac{1}{2}(w-2k)$ and $b+w<c_1<\dots<c_l$. Suppose for $1\leq i<l$, $c_{i+1}-c_i$ is even. For $i=1,\dots, l$, let $z_i=(c_i,\alpha+1)$. Let $C_0$ be the absolute constant given by Lemma~\ref{lem:standard}. Then for any $b'>c_l+C_0w^2$, we can extend $M$ to a full partial matching $M'$
of the region 
$$ R'=\{(x,y)\in\Z^2\,:\, a\leq x\leq b',\ \alpha<y<\beta\}-\{z_1, \dots, z_l\} $$
so that $M'$ is of the standard form on the right-edge of $R'$.
\end{lem}

\begin{proof} We use the same algorithms as in the proof of Lemma~\ref{lem:uc} to absorb the charges one-by-one. By our assumption, only one of the cases in the proof of Lemma~\ref{lem:uc} occurs for all given charges. Now suppose all the charges are of the form of the first case in that proof. Then note that the algorithm we use for this case can be applied simultaneously for all charges at the same time without creating a conflict. Moreover, the pattern of the resulting partial matching is still in the standard form. 
\begin{figure}[h] 

\begin{subfigure}{0.475\textwidth}
\centering
\begin{tikzpicture}[scale=0.05]

\pgfmathsetmacro{\a}{5}
  
\draw[thick] (-5,0) to (60,0);
\draw[thick] (-5,45) to (60,45);

\foreach \i in { 0,..., 11}
\foreach \j in {1,..., 8}
         {\draw[fill=black] (5*\i, 5*\j) circle (0.3);
           }

\foreach \j in {2,4}
{ \draw (0*\a, \j*\a) to (0*\a+\a,\j*\a);}

\foreach \j in {1,3,5,6,7,8}
{ \draw (1*\a, \j*\a) to (1*\a+\a,\j*\a);}

\foreach \j in {2,4}
{ \draw (2*\a, \j*\a) to (2*\a+\a,\j*\a);}

\foreach \j in {1,3}
{ \draw (3*\a, \j*\a) to (3*\a+\a,\j*\a);}

\foreach \j in {2}
{ \draw (4*\a, \j*\a) to (4*\a+\a,\j*\a);}

\foreach \j in {1}
{ \draw (5*\a, \j*\a) to (5*\a+\a,\j*\a);}

\foreach \i in {3,4,5,6,7}
{ \draw[color=green]  (\i*\a, 8*\a-\i*\a) to (\i*\a,9*\a-\i*\a);}

\foreach \i in {3,4, 5, 6, 7,8,9}
{\draw[color=purple] (\i*\a, 10*\a-\i*\a) to (\i*\a,11*\a-\i*\a);}

\draw[fill=black] (8*\a,\a) circle (0.6);
\draw[thick] (8*\a,0) to (8*\a,\a);

\draw[fill=black] (10*\a,\a) circle (0.6);
\draw[thick] (10*\a, 0) to (10*\a, \a);

\foreach \i in {6,8,10}
{\draw (\i*\a,6*\a) to (\i*\a+\a,6*\a);}

\foreach \i in {7,9}
{\draw (\i*\a,5*\a) to (\i*\a+\a,5*\a);}

\foreach \i in {8,10}
{\draw (\i*\a,4*\a) to (\i*\a+\a,4*\a);}

\foreach \i in {9}
{\draw (\i*\a,3*\a) to (\i*\a+\a,3*\a);}

\foreach \i in {10}
{\draw (\i*\a,2*\a) to (\i*\a+\a,2*\a);}

\foreach \i in {5, 7, 9}
{\draw (\i*\a,7*\a) to (\i*\a+\a,7*\a);}

\foreach \i in {4,6,8,10}
{\draw (\i*\a, 8*\a) to (\i*\a+\a, 8*\a);}

\end{tikzpicture}
\caption{The first case}
\end{subfigure}
\hfill
\begin{subfigure}{0.475\textwidth}
\centering
\begin{tikzpicture}[scale=0.05]

\pgfmathsetmacro{\a}{5}
  
\draw[thick] (-5,0) to (60,0);
\draw[thick] (-5,45) to (60,45);

\foreach \i in { 0,..., 11}
\foreach \j in {1,..., 8}
         {\draw[fill=black] (5*\i, 5*\j) circle (0.3);
           }

\foreach \j in {2,4}
{ \draw (0*\a, \j*\a) to (0*\a+\a,\j*\a);}

\foreach \j in {1,3,5,6,7,8}
{ \draw (1*\a, \j*\a) to (1*\a+\a,\j*\a);}

\foreach \j in {2,4}
{ \draw (2*\a, \j*\a) to (2*\a+\a,\j*\a);}

\foreach \j in {1,3}
{ \draw (3*\a, \j*\a) to (3*\a+\a,\j*\a);}

\foreach \j in {2}
{ \draw (4*\a, \j*\a) to (4*\a+\a,\j*\a);}

\foreach \j in {1}
{ \draw (5*\a, \j*\a) to (5*\a+\a,\j*\a);}

\foreach \i in {8}
{\draw[color=purple] (\i*\a, 9*\a-\i*\a) to (\i*\a,10*\a-\i*\a);}

\draw[fill=black] (7*\a,\a) circle (0.6);
\draw[thick] (7*\a,0) to (7*\a,\a);

\draw[fill=black] (9*\a,\a) circle (0.6);
\draw[thick] (9*\a, 0) to (9*\a, \a);

\foreach \i in {3,5,7, 9}
{\draw (\i*\a,6*\a) to (\i*\a+\a,6*\a);}

\foreach \i in {3,5, 7,9}
{\draw (\i*\a,5*\a) to (\i*\a+\a,5*\a);}

\foreach \i in {4, 6, 8,10}
{\draw (\i*\a,4*\a) to (\i*\a+\a,4*\a);}

\foreach \i in {5, 7, 9}
{\draw (\i*\a,3*\a) to (\i*\a+\a,3*\a);}

\foreach \i in {10}
{\draw (\i*\a, \a) to (\i*\a+\a, \a);}

\foreach \i in {6,9}
{\draw (\i*\a,2*\a) to (\i*\a+\a,2*\a);}

\foreach \i in {3, 5, 7, 9}
{\draw (\i*\a,7*\a) to (\i*\a+\a,7*\a);}

\foreach \i in {3,5,7,9}
{\draw (\i*\a, 8*\a) to (\i*\a+\a, 8*\a);}

\end{tikzpicture}
\caption{The second case}
\end{subfigure}

\caption{Absorbing multiple charges of the same sign.} \label{fig:occ}
\end{figure}
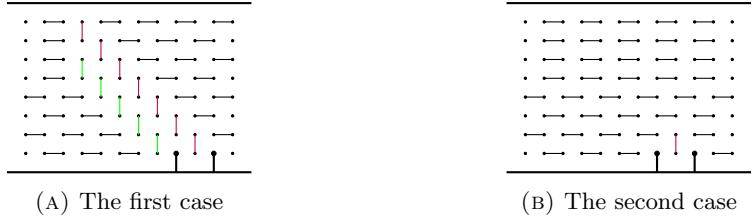

Similarly, if all the charges are of the form of the second case, then the greedy algorithm to use horizontal edges in the matching works also for all the charges. Figure~\ref{fig:occ} illustrates these situations. 
After absorbing all the charges, it takes at most $C_0w^2$ many steps to put the partial matching back to the standard form. 
\end{proof}

We will also use the following lemma on partial matchings of a rectangle.

\begin{lem}\label{lem:rectangle} Let $R$ be a rectangle in $\Z^2$ and $k=\sum_{x\in R} q(x)$ be its total charge. Then $k\in\{0,1,-1\}$. Suppose $S\subseteq \partial R$ with $\sum_{x\in S}q(x)=k$. Suppose also that any side length of $R$ is greater than the size of $S$, $|S|$. Then there is a perfect matching $M$ of $R-S$.
\end{lem}

\begin{proof} That $k\in\{0,1,-1\}$ is obvious. We prove the second part of the lemma by induction on $|S|$. When $|S|=0$ we have $k=0$ and $R$ must be a rectangle with dimensions $l\times w$ where at least one of $l$ and $w$ is even. In this case it is obvious that $R$ has a perfect matching. When $|S|=1$ we have $k=\pm 1$ and the only element $x_0\in S$ satisfies $q(x_0)=k$. Now $R$ is a rectangle with dimensions $l\times w$ where both $l$ and $w$ are odd, and all four corner points of $R$ must have charge $k$. Let $E$ be a side of $R$ containing $x_0$.  Then $E-\{x_0\}$ consists of at most two segments of even size, and therefore each admits a perfect matching. $R-E$ is a rectangle with total charge $0$ and therefore admits a perfect matching by the above observation. In summary, we have that $R-\{x_0\}$ has a perfect matching.

Now consider $|S|>1$. Then at least two elements of $S$ have opposite charge. Fixing an orientation of $\partial R$, we also conclude that at least two elements of $S$, say $x_0$ and $x_1$, have opposite charge and are adjacent (that is, there are no other
points of $S$ in between $x_0$ and $x_1$ along $\partial R$ from $x_0$ to $x_1$). Let $E$ be the path segment of $\partial S$ with $x_0$ and $x_1$ as endpoints. $E$ is either a part of a side of $R$ or a part of two consecutive sides of $R$. In either case let $F\subseteq \partial R$ be either a side or two consecutive sides containing $E$. Let $\tilde{F}$ be the set of all points of $R$ with at most distance $1$ to $F$. Consider $R'=R-\tilde{F}$. Then $R'$ is still a rectangle. The total charge of $R'$ is still $k$. Let $S_0'=S-\tilde{F}$ and $S_1'\subseteq \partial R'$ be the set of points $x'\in\partial R'$ that have distance 2 to some $x\in (F-E)\cap S$ in the Cayley graph of $\Z^2$. Then $S'=S_0'\cup S_1'$ has size at most $|S|-2$, which is less than any side length of $R'$. By the inductive hypothesis, there is a perfect matching $M'$ of $R'-S'$. We obtain a perfect matching $M$ of $R-S$ by extending $M'$ and making the following definitions:
\begin{enumerate}
\item[(i)] $M$ is a perfect matching of elements of $E-\{x_0,x_1\}$ (this is possible since there are even number of points from $x_0$ to $x_1$ along $E$);
\item[(ii)] Let $\tilde{E}$ be the set of points $x\not\in \partial R$ but are adjacent to some points of $E$ in the Cayley graph of $R$. Then $M$ is a perfect matching of $\tilde{E}$ (this is because either $\tilde{E}$ is a path of the same length as $E$ or their length differ by $4$);
\item[(iii)] For each $x\in F-E-S$, $M$ matches $x$ to the unique $x^*\not\in \partial R$ adjacent to $x$ in the Cayley graph of $R$;
\item[(iv)] For each $x\in (F-E)\cap S$, $M$ leaves $x$ unmatched, but matches its corresponding point $x'$ in $S_1'$ to the unique $x^*\not\in\partial R$ adjacent to $x$ in the Cayley graph of $R$.
\end{enumerate}

\begin{figure}[h]
\begin{tikzpicture}[scale=0.05]

\draw[thin, lightgray] (0,30) to (0,0) to (150,0) to (150,30);
\node at (5,-5) {$F$};
\draw[thin, lightgray] (0,6) to (150,6);
\draw[thin, lightgray] (0,12) to (150,12);
\node at (156,3) {$\big\}\tilde{F}$};

\draw[fill] (0,0) circle (0.8);
\draw[fill] (6,0) circle (0.8);
\draw[fill] (6,6) circle (0.8);
\draw[fill] (0,6) circle (0.8);
\draw[thick,black] (0,0) to (0,6);
\draw[thick,black] (6,0) to (6,6);

\node at (14,3) {$\cdots$};

\draw[fill] (26,0) circle (0.8);
\draw[fill] (26,6) circle (0.8);
\draw[fill] (26,12) circle (0.8);
\draw[thick, black] (26,6) to (26,12);
\node at (26, -5) {$S$};
\node at (26, 17) {$S'$};

\draw[fill] (20,0) circle (0.8);
\draw[fill] (20,6) circle (0.8);
\draw[thick, black] (20,0) to (20,6);
\draw[fill] (32,0) circle (0.8);
\draw[fill] (32,6) circle (0.8);
\draw[thick, black] (32,0) to (32,6);
\node at (40, 3) {$\cdots$};

\draw[fill] (55,0) circle (0.8);
\draw[fill] (61,0) circle (0.8);
\draw[fill] (67,0) circle (0.8);
\draw[fill] (81,0) circle (0.8);
\draw[fill] (87,0) circle (0.8);
\draw[fill] (93,0) circle (0.8);
\node at (55, -5) {$x_0$};
\node at (93, -5) {$x_1$};
\node at (72, -5) {$E$};
\draw[thick, black] (61,0) to (67,0);
\draw[thick,black] (81,0) to (87,0);
\node at (75,1) {$\cdots$};

\node at (72,11) {$\tilde{E}$};

\draw[fill] (55,6) circle (0.8);
\draw[fill] (61,6) circle (0.8);
\draw[fill] (67,6) circle (0.8);
\draw[fill] (81,6) circle (0.8);
\draw[fill] (87,6) circle (0.8);
\draw[fill] (93,6) circle (0.8);

\draw[thick, black] (55,6) to (61,6);
\draw[thick, black] (67,6) to (70,6);
\draw[thick, black] (78,6) to (81,6);
\draw[thick, black] (87,6) to (93,6);

\node at (104,3) {$\cdots$};

\draw[fill] (116,0) circle (0.8);
\draw[fill] (116,6) circle (0.8);
\draw[fill] (116,12) circle (0.8);
\draw[thick, black] (116,6) to (116,12);
\node at (116, -5) {$S$};
\node at (116, 17) {$S'$};

\draw[fill] (110,0) circle (0.8);
\draw[fill] (110,6) circle (0.8);
\draw[thick, black] (110,0) to (110,6);
\draw[fill] (122,0) circle (0.8);
\draw[fill] (122,6) circle (0.8);
\draw[thick, black] (122,0) to (122,6);
\node at (130, 3) {$\cdots$};
\end{tikzpicture}
\caption{Constructing a partial matching of a rectangle.}\label{fig:rectlem}
\end{figure}
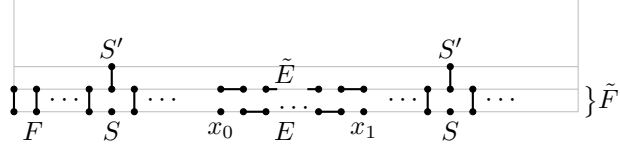

The construction when $F$ is a single edge is illustrated in Figure~\ref{fig:rectlem}. The construction when $F$ is the union of two edges is similar.
\end{proof}

\subsection{An orthogonal decomposition with bounded charge}


An important step in the construction of a Borel perfect matching for $F(2^{\Z^2})$ is
the construction of an orthogonal decomposition such that all of the $k$-atoms  
have a charge bounded by a fixed constant $K$, which is independent of $k$.
The key to doing this is the use of the finite Stokes theorem, Theorem~\ref{thm:stokes},
which says that it is enough to control the function $q(x)c(x)$ along the boundaries
of these regions. This we do by a modification of the orthogonal marker construction given in Section~\ref{sec:topreg}.

Recall that for each $\ell\geq 1$, the $\sR^\ell_\ell$ regions form a rectangular decomposition of $F(2^{\Z^2})$. Then, as we inductively defined $\sR^\ell_k$ for $k<\ell$, each $\sR^\ell_k$ region is a ``fractalized" version of a corresponding $\sR^\ell_{k+1}$ region. Nevertheless, each $\sR^\ell_k$ region is still a rectangular polygon with side lengths between $\epsilon d_k$ and $d_k$. At the end, we obtain $\sR_\ell=\sR^\ell_1$, where each region is a rectangular polygon with side lengths between $\epsilon d_1$ and $d_1$. We will need $d_1$ to be sufficiently large in order for various arguments below to work. Here we do not specify its exact value, but only note that $d_1$ can obviously be made to be larger than any given constant.

Next we analyze the boundaries of the $\sR_\ell$ regions, denoted $X_\ell=\partial \sR_\ell$. By a {\em path segment}, we mean a 
simple (i.e. non-self-intersecting) path consisting of a concatenation of horizontal and vertical segment of points in the Schreier graph of 
$F(2^{\Z^2})$ (cf. Lemma~\ref{lem:pp}). The boundary of any $\sR_\ell$ region itself (picking a starting point) is a path segment. We say that two path segments $p$ and $p^*$ are {\em adjacent}, if for any $x\in p$ there is $x^*\in p^*$ with $\rho(x,x^*)=1$ and vice versa. 

We remark that these objects might look somewhat different at different scales. Figure~\ref{fig:pb} illustrates the construction of $\sR^\ell_\ell$, where each region is a rectangle. From a large scale, for example, the scale $d_\ell\gg d_1$, the boundaries of the regions can be roughly viewed as lines on $\R^2$ that divide the plane into rectangular regions. However, from a detailed scale, that is, the scale $d_1$, the boundaries of $\sR^\ell_\ell$ regions can be seen to actually consist of adjacent path segments that are either horizontal or vertical. For results in this subsection we need the perspective from the $d_1$ scale. 

\begin{figure}[h] 
\centering

\begin{subfigure} {0.475\textwidth}
\centering
\begin{tikzpicture}[scale=0.03]

\pgfmathsetmacro{\a}{5}
  
\draw (0,0) rectangle (50,50);
\draw (0,-60) rectangle (50,0);
\draw (50,-20) rectangle (90,-75);
\draw (50,-20) rectangle (90,50);
\draw (0,50) rectangle (70,75);
\draw (-30,20) rectangle (0,65);
\draw (-50,-40) rectangle (0,20);

\draw[thick, red] (0,50) to (0,20);
\draw[thick, red] (0,20) to (0,0);
\draw[thick, red] (0,0) to (50,0);
\draw[thick, red] (50,0) to (50,50);
\draw[thick, red] (50,50) to (0,50);

\draw[fill] (0,0) circle (1.5);
\draw[fill] (50,50) circle (1.5);
\draw[fill] (50,0) circle (1.5);
\draw[fill] (0,-60) circle (1.5);
\draw[fill] (50,-20) circle (1.5);
\draw[fill] (90,50) circle (1.5);
\draw[fill] (0,50) circle (1.5);
\draw[fill] (0,75) circle (1.5);
\draw[fill] (-30,20) circle (1.5);
\draw[fill] (0,65) circle (1.5);
\draw[fill] (-50,-40) circle (1.5);
\draw[fill] (0,20) circle (1.5);
\draw[fill] (70,50) circle (1.5);
\draw[fill] (0,-40) circle (1.5);
\draw[fill] (50,-60) circle (1.5);

\node(A) at (25,25) {$R$} ;
\end{tikzpicture}
\caption{At the scale $d_\ell$}
\end{subfigure} 
\hfill
\begin{subfigure} {0.475\textwidth}
\centering
\begin{tikzpicture}[scale=0.03]

\draw (0,30) to (0,-30) to (30,-30);
\draw (30,-35) to (0,-35) to (0,-50);
\draw (-5,30) to (-5,5) to (-35,5);
\draw (-35,0) to (-5,0) to (-5,-50);

\end{tikzpicture}
\caption{At the scale $d_1$}
\end{subfigure}

\caption{Perspectives of the boundaries of the $R^\ell_\ell$ regions.} \label{fig:pb}
\end{figure}
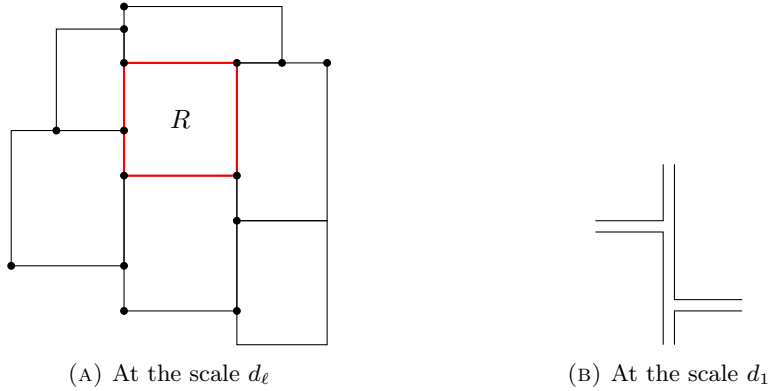 

\begin{lem} \label{lem:adjpath} There is an absolute constant $C_1$ such that for any $\ell\geq 1$, there is a Borel partition $P$ of $X_\ell$ into pairs of adjacent path segments so that for each $R\in \sR_\ell$, $\partial R$ consists of at most $C_1$ many path segments in $P$.
\end{lem}

\begin{proof} We obtain $P$ by following the inductive construction of $\sR^\ell_k$ for $k\leq \ell$. For each $k\leq \ell$, we obtain a Borel partion $P_k$ of $\partial \sR^\ell_k$ into pairs of adjacent path segments so that the boundary of each $R\in \sR^\ell_k$ consists of at most $C_1$ many path segments in $P_k$. Eventually we let $P=P_1$. 

For $k=\ell$ the statement obviously holds with $C_1=12$. This is because, each $R\in \sR^\ell_\ell$ is a rectangle with side lengths between $\frac{1}{2}\alpha_\ell d_\ell$ and $\alpha_\ell d_\ell$ by \ref{i1}, and the usual othogonality of the sides of $\sR^\ell_\ell$ guarantees that each horizontal or vertical segment of points in $P_\ell$ has length at least $\epsilon d_\ell$. 

In general, as we consider $k<\ell$ and construct $\sR^\ell_k$ regions from $\sR^\ell_{k+1}$ regions, each of the path segments in $P_{k+1}$ gives rise to a ``fractalized" path segment in $P_k$. For example, each horizontal or vertical segment from $P_\ell$ becomes a path segment (no longer necessarily horizontal or vertical) in $P_{\ell-1}$. Note that we have $\epsilon d_\ell>\epsilon d_1\gg 5$, and it follows that each pair of adjacent path segments in $P_{k+1}$ corresponds still to a pair of adjacent path segments in $P_k$.

Thus, at the end of the process, we obtain $P_1$ which is a partition of $X_\ell$ into pairs of adjacent path segments, where the boundary of each $\sR^\ell_1$ region still consists of no more than $C_1$ many path segments in $P_1$. By the inductive hypotheses on the orthgonality of $\sR^\ell_k$ ($1\leq k\leq\ell$) regions, we get that, for each path segment in $P_1$ consisting of horizontal and vertial segments of points, each horizontal and vertical segment has lengths at least $\epsilon d_1$.
\end{proof}

The next lemma modifies the construction of $\sR_\ell$ in Section~\ref{sec:topreg} so that the total charge of any $\sR_\ell$ region is bounded.

\begin{lem} \label{lem:bounded_charge}

The orthogonal marker decompositions $\{\sR^\ell_k\}_{1\leq k\leq \ell}$ may be constructed so as to satisfy the following additional property: There is an absolute constant $4\leq K<d_1$ such that for any $\ell\geq 1$ and any path segment $p$ contained in the boundary of some $\sR_\ell$ region, we have that
$$ | \sum_{x\in p} q(x)c(x)| \leq K. $$

\end{lem}

\begin{proof}

For any $\ell\geq 1$, we leave all the constructions of $\sR^\ell_k$ for $1<k\leq\ell$ as they were done in Section~\ref{sec:topreg} and only modify the last step of the construction. To do this, first let $\sR_\ell'$ be the $\sR^\ell_1$ constructed in Section~\ref{sec:topreg}. We indicate below how to modify it to obtain $\sR_\ell$ as we desire.

Let $C_1$ be the absolute constant in Lemma~\ref{lem:adjpath} and let $P$ be a Borel partition of $X_\ell'=\partial \sR'_\ell$ into pairs of adjacent path segments given by Lemma~\ref{lem:adjpath}. Let $C_0$ be the absolute constant in Lemma~\ref{lem:standard}. Let $w\geq 10$ be an even number with $\epsilon d_1\gg 40C_0w^2$.

Consider a pair $\{p,p^*\}$ of adjacent path segements in $P$. Both $p$ and $p^*$
are simple (i.e. non-self-intersecting) paths and consist of a concatenation of vertical and
horizontal segments of lengths at least $\epsilon d_1$. Let $R$, $R^*$ be the corresponding
regions in $R_\ell'$ with $p \subseteq \partial R$, $p^*\subseteq \partial R^*$. Write $p$ as a concatenation $s_0 s_1\dots s_h$ where the $s_j$ are horizontal
or vertical line segments of points of lengths between $\epsilon d_1$ and $d_1$,
except perhaps for the first and last segment. By our construction (specifically, by Claim~\ref{claim:si} and Figure~\ref{fig:possible_shapes}), there are at most two of the $s_j$
which intersect an $X'_m$ for $m>\ell$. For all of the remaining $\geq h-4$
segments $s_j$, we can position a $(40w+1)\times 40w$ rectangle $S_j$ of points near the
center of the segment $s_j$ (we will say more about its specific location later in this proof) as shown in Figure~\ref{fig:ca} so as to
change the value of $\sum_{x \in s_j} q(x)c(x)$ by either $+4$ or $-4$, where here $c(x)$
is computed relative to the region $R$. Since $\epsilon d_1\gg 40w+1$ and these $S_j$ are
near the centers of the segments $s_j$, they are contractible to $s_j$ without
intersecting any of the other $S_{j'}$ or $s_{j'}$. Thus, adding these $S_j$
does not alter the topology of the region $R$ (it is still homeomorphic to a disk),
nor the connectedness of $\partial R$. 

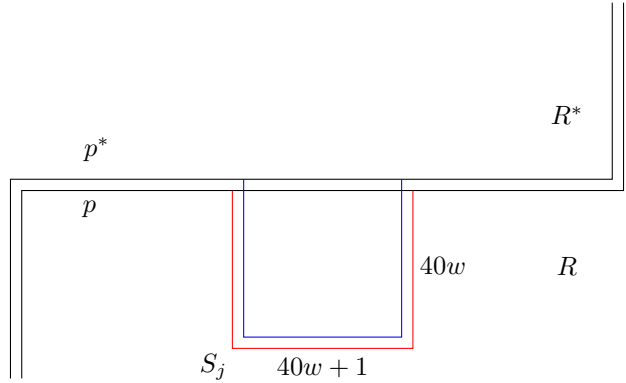
\begin{figure}[h]

\begin{tikzpicture}[scale=0.05]

\pgfmathsetmacro{\a}{3};
\pgfmathsetmacro{\b}{24};
\pgfmathsetmacro{\c}{42}

\draw(-20,0) to (-20,50) to (140,50) to (140,100);
\draw (-\a-20,0) to (-\a-20,50+\a) to (140-\a,50+\a) to (140-\a,100);

\draw[red] (60-\b, 50) to (60-\b, 50-\c) to (60+\b, 50-\c) to (60+\b, 50);
\draw[blue] (60-\b+\a, 50+\a) to (60-\b+\a, 50-\c+\a) to (60+\b-\a, 50-\c+\a) to (60+\b-\a, 50+\a);

\node(a) at (60-\b-5, 50-\c-5) {$S_j$};
\node(b) at (-2,45) {$p$};
\node(c) at (0,61) {$p^*$};
\node(d) at (125, 30) {$R$};
\node(e) at (125, 70) {$R^*$};

\node(f) at (60, 3) {$40w+1$};
\node(g) at (92, 30) {$40w$};

\end{tikzpicture}
\caption{Adjusting the boundaries.} \label{fig:ca}
\end{figure}

Thus it is possible to add the regions $S_j$ so that for any initial segment
$u$ of $p$ we have that $|\sum_{x \in u}q(x)c(x)| \leq 5C_1$. Furthermore,
the $S_j$ are not placed in the first or last segments of $p$, nor in
a segment which intersects an $X_m$ for $m>\ell$. A simple orthognality
argument shows that 
we will have that any $S_j$ is at least $\frac{1}{3}\epsilon d_1$ from $X_m$. 
If $s_j$ is not the initial or final line segment of $p$, nor does it intersect $X'_m$ for any $m>\ell$, then we
will take into account one more consideration as we determine the exact location of $S_j$ on the line segment $s_j$. Consider the set $s_j\cap \bigcup_{k<\ell} X'_k$. 
By the orthogonality from our construction, each point in $s_j\cap \bigcup_{k<\ell} X'_{k}$ is at least $\epsilon d_1$ from the endpoints of $s_j$, and if $x, y$ are two distinct points of $s_j\cap \bigcup_{k<\ell} X'_k$, then either $\rho(x,y)=1$ or $\rho(x,y)>\epsilon d_1$. Thus $s_j\setminus \bigcup_{k<\ell}X'_{k}$ consists of a number of line segments, each of which has length at least $\epsilon d_1$. We will actually put $S_j$ near the center of one of these line segments, so that it is at least $\frac{1}{3}\epsilon d_1\gg C_0w^2$ from $s_j\cap \bigcup_{k<\ell}X'_k$ and from the endpoints of $s_j$. This property will be useful in our construction in the next subsection.

From Lemma~\ref{lem:pp} we also have that $|\sum_{x \in u^*}q(x)c(x)| \leq 5C_1$
for any initial segment $u^*$ of $p^*$. Since the number of path segments
in the boundary of any $\sR_\ell'$ region $R$ is bounded by $C_1$,
this gives a bound $K$ for the sum $\sum_{x \in p} q(x)c(x)$ for any path
$p \in \partial R$ for an $\sR_\ell$ region $R$, as required. 
\end{proof}

Similarly, in view of the analysis of
Figure~\ref{fig:possible_shapes}, we get a similar bound for paths $p$ in $\partial A$
for any $k$-atom $A$.

\begin{lem} \label{lem:bounded_charge_atoms}

There is an absolute constant $4\leq K<d_1$ such that for any $k\geq 1$ and any path segment $p$ contained in the boundary of some
$k$-atom $A$, we have that
$|\sum_{x \in p} q(x)c(x)| \leq K$. 
\end{lem}

\begin{proof}
This follows immediately from Lemma~\ref{lem:bounded_charge} and the fact that each 
$k$-atom $A$ intersects ${X}_\ell$ for at most $2$ distinct $\ell >k$ (see Figure~\ref{fig:possible_shapes}). 
\end{proof}

\subsection{The construction of a Borel perfect matching}

We are now ready to give the construction of a Borel perfect matching $M$ of $F(2^{\Z^2})$.
In the arguments constructing orthogonal decompositions up to this point,
we assumed the set $d_k$ of distance scales satisfied $d_1 < d_2 \cdots $
and $\frac{d_{k+1}}{d_k}>C$ for some fixed constant $C$ independent of $k$.
We now wish to impose an {\em intermediate growth condition} on the $d_k$:
we assume hence forth that $\frac{d_{k+1}}{d_k}\leq C'$ for some fixed constant $C'$.
It follows that there is an absolute bound $C''$ (independent of $k$) on the number of
$(k-1)$-atoms that can be contained within a given $k$-atom.

Let us fix an even number  $w>10 C'' K$, where $K$  is as in Lemma~\ref{lem:bounded_charge_atoms}. Intuitively,
$C''K$ represents the possible ``accumulation of charge'' within a $k$-atom in the construction, $w$ will be the thickness of
a ``buffer" region in which the inductive construction takes place. 

Let $C_0$ be the absolute constant given by Lemma~\ref{lem:standard}. We assume henceforth that in our construction
$\epsilon d_1\gg 40(K+3)C_0 w^2$.


For each $k$, if $A$ is a $k$-atom of the orthogonal decomposition, then we define the corresponding {\em $w$-buffered atom} $A(w)=\{
x \in A \colon \rho(x, \partial A) \geq w\}$. Also let $\partial_w(A)=A- A(w)$ denote the {\em buffer region}.  Note from the construction of Lemmas~\ref{lem:bounded_charge} and \ref{lem:bounded_charge_atoms}, all the line segments forming the boundary of $A$ have
lengths at least $40w$,  thus the buffer region takes the shape of a ``rectangular polygonal pipe" of thickness $w$, as illustrated in Figure~\ref{fig:buffer}. We refer to $\partial A(w)$ as the {\em buffered boundary}, and each line segment of $\partial A(w)$ still has length at least $30w$. Moreover, $\partial A$ consists of path segments each of which is contained entirely on $X_\ell$ for some $\ell\geq k$, and we refer to each of the path segments as being on level $\ell$; there are corresponding path segments of $\partial A(w)$ which are also referred to as being on level $\ell$.

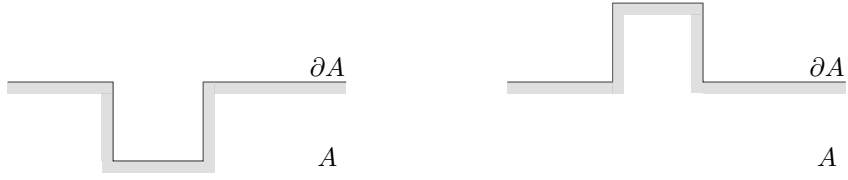
\begin{figure}[h]

\begin{subfigure}{0.475\textwidth}

\begin{tikzpicture}[scale=0.05]

\pgfmathsetmacro{\a}{3};
\pgfmathsetmacro{\b}{12};
\pgfmathsetmacro{\c}{21}

\draw (20,50) to (60-\b, 50) to (60-\b, 50-\c) to (60+\b, 50-\c) to (60+\b, 50) to (110,50);

\draw[draw=none,fill=lightgray, opacity=0.5] (23-\a,50-\a) rectangle ((60-\b,50);
\draw[draw=none,fill=lightgray, opacity=0.5] (57-\b,50-\a) rectangle ((60-\b,50-\c);
\draw[draw=none,fill=lightgray, opacity=0.5] (57-\b,50-\c-\a) rectangle ((60+\b+\a,50-\c);
\draw[draw=none,fill=lightgray, opacity=0.5] (60+\b,50-\c) rectangle ((60+\b+\a,50);
\draw[draw=none,fill=lightgray, opacity=0.5] (60+\b+\a,50-\a) rectangle ((110,50);

\node(d) at (105, 30) {$A$};
\node(e) at (105, 54) {$\partial A$};

\end{tikzpicture}
\end{subfigure}
\hfill 
\begin{subfigure}{0.475\textwidth}

\begin{tikzpicture}[scale=0.05]

\pgfmathsetmacro{\a}{3};
\pgfmathsetmacro{\b}{12};
\pgfmathsetmacro{\c}{21}

\draw (20,50) to (60-\b, 50) to (60-\b, 50+\c) to (60+\b, 50+\c) to (60+\b, 50) to (110,50);

\draw[draw=none,fill=lightgray, opacity=0.5] (23-\a,50-\a) rectangle ((60-\b,50);
\draw[draw=none,fill=lightgray, opacity=0.5] (60-\b,50-\a) rectangle ((60-\b+\a,50+\c);
\draw[draw=none,fill=lightgray, opacity=0.5] (60-\b+\a,50+\c-\a) rectangle ((60+\b,50+\c);
\draw[draw=none,fill=lightgray, opacity=0.5] (60+\b-\a,50+\c-\a) rectangle ((60+\b,50-\a);
\draw[draw=none,fill=lightgray, opacity=0.5] (60+\b,50-\a) rectangle ((110,50);

\node(d) at (105, 30) {$A$};
\node(e) at (105, 54) {$\partial A$};

\end{tikzpicture}
\end{subfigure}

\caption{The buffer region.} \label{fig:buffer}
\end{figure}

In particular, since the $k$-atoms are topologically of the form as in Figure~\ref{fig:possible_shapes}
and are homeomorphic to disks with boundaries being simple closed curves consisting
of paths of horizontal and vertical line segments of lengths $>30w$.
From this it easily follows that the buffered atoms $A(w)$ also have these properties.

We now proceed by induction on $k$ to define a partial matching on $\bigcup \{ A(w)
\colon A\in \sA_k\}$, where $\sA_k$ is the set of all $k$-atoms. We assume the following
inductive hypothesis at level $k-1$.
\medskip

{\bf Induction Hypothesis} ${\mathbf H}_{k-1}$:
We have defined a Borel partial matching $M_{k-1}$ of $F(2^{\Z^2})$
with domain $\dom(M_{k-1})=D_{k-1}\subseteq \bigcup \{ A(w) \colon A\in \sA_{k-1}\}$. For
each $(k-1)$-atom $A\in \sA_{k-1}$, denote $E_{k-1}(A)=A(w)-D_{k-1}$. Then $E_{k-1}(A)$ is a finite set
of size $\leq K$ and $E_{k-1}(A) \subseteq \partial A(w)$. Furthermore,
every point of $E_{k-1}(A)$ is at least $C_0w^2$ from the endpoint of a line segment
containing it which is part of $\partial A(w)$ and also every such point
is within $w$ of $X_{k-1}=\partial \sR_{k-1}$.
\medskip

The set $E_{k-1}(A)$ is the set of unmatched (by $M_{k-1}$) points of $A(w)$. In view of Figure~\ref{fig:possible_shapes}, the boundary of each atom $A\in \sR_{k-1}$ consists of path segments from some $X_\ell$ where $\ell\in\{k-1,m_0,m_1\}$ for some $m_0, m_1>k-1$. The last part of the induction hypothesis ${\mathbf H}_{k-1}$ says that these unmatched points of
$E_{k-1}(A)$ in the buffered $(k-1)$-atom $A(w)$, for $A \in \sA_{k-1}$, are part of the
buffered boundary for the lowest level $X_\ell$ (i.e., $\ell=k-1$) which forms
part of the boundary of $A$. 

\begin{lem} \label{lem:H1} ${\mathbf H}_1$ holds. 
\end{lem}

\begin{proof} Let $A\in \sA_1$ be a $1$-atom. By the construction of the orthogonal decomposition in Section~\ref{sec:topreg}, particularly by Figure~\ref{fig:possible_shapes}, there is an $\sR^1_1$ region $A_0$ such that one of the following holds:
\begin{enumerate}[label={(\arabic*)}]
\item\label{case1} $A=A_0$;
\item\label{case2} $A=A_0\cap R$, where $R$ is an $\sR^k_1$ region for some $k>1$;
\item\label{case3} $A=A_0\cap R_1\cap R_2$, where $R_1$ is an $\sR^k_1$ region and $R_2$ is an $\sR^\ell_1$ region for $\ell>k>1$. 
\end{enumerate}
Let $A_0', R', R_1', R_2'$ denote the corresponding regions before the adjustments performed in the proof of  Lemma~\ref{lem:bounded_charge} and \ref{lem:bounded_charge_atoms}. 
Since the constant $K\geq 4$ in Lemmas~\ref{lem:bounded_charge} and \ref{lem:bounded_charge_atoms}, no adjustments took place for the boundaries of $A_0'$, and and thus $A_0=A_0'$ remains a rectangle of side lengths between $\frac{1}{2}\alpha d_1$ and $\alpha d_1$ by \ref{i1}. Note that 
the side lengths of line segments consistituting $\partial R'$, $\partial R_1'$ or $\partial R_2'$ are at least $\epsilon_1 d_1$ by the orthogonality assumption \ref{i3}. Since $\alpha<\epsilon_1$, we conclude that $A_0\cap \partial R'$ is either a line segment or a path segment with two (perpendicular) line segments in Case~\ref{case2}, and $A_0\cap \partial R_1'$ and $A_0\cap \partial R_2'$ are line segments in Case~\ref{case3}. Therefore, after the adjustments from the proof of Lemmas~\ref{lem:bounded_charge} and \ref{lem:bounded_charge_atoms} are performed,  all possible geometric shapes of $A$ are illustrated in Figure~\ref{fig:H1}, where the small structures along the line segments represent the results after adding or removing various $S_j$'s of dimensions $(40w+1)\times 40w$. The line segments other than from these small structures all have lengths at least $\frac{1}{3}\epsilon d_1\gg 10C_0w^2$.

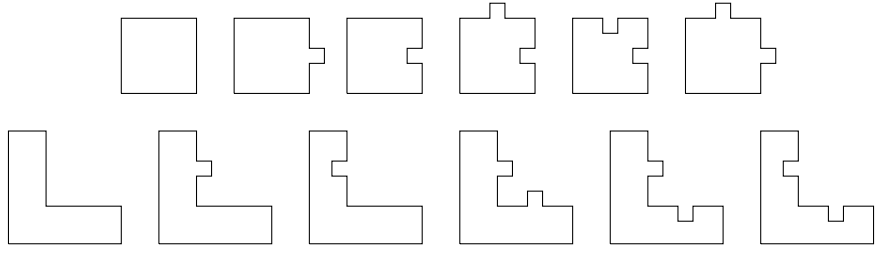
\begin{figure}[h]
\begin{tikzpicture}[scale=0.05]
\pgfmathsetmacro{\a}{5}
\pgfmathsetmacro{\b}{30}
\pgfmathsetmacro{\c}{40}

\draw (0,0) to (20,0) to (20,20) to (0,20) to (0,0);
\draw (\b,0) to (20+\b, 0) to (20+\b,8) to (24+\b,8) to (24+\b,12) to (20+\b,12) to (20+\b,20) to (\b,20) to (\b,0);
\draw (2*\b,0) to (20+2*\b, 0) to (20+2*\b,8) to (16+2*\b,8) to (16+2*\b,12) to (20+2*\b,12) to (20+2*\b,20) to (2*\b,20) to (2*\b,0);
\draw (3*\b,0) to (20+3*\b, 0) to (20+3*\b,8) to (16+3*\b,8) to (16+3*\b,12) to (20+3*\b,12) to (20+3*\b,20) to (12+3*\b, 20) 
to (12+3*\b, 24) to (8+3*\b, 24) to (8+3*\b, 20) to (3*\b,20) to (3*\b,0);
\draw (4*\b,0) to (20+4*\b, 0) to (20+4*\b,8) to (16+4*\b,8) to (16+4*\b,12) to (20+4*\b,12) to (20+4*\b,20) to (12+4*\b, 20) 
to (12+4*\b, 16) to (8+4*\b, 16) to (8+4*\b, 20) to (4*\b,20) to (4*\b,0);
\draw (5*\b,0) to (20+5*\b, 0) to (20+5*\b,8) to (24+5*\b,8) to (24+5*\b,12) to (20+5*\b,12) to (20+5*\b,20) to (12+5*\b, 20) 
to (12+5*\b, 24) to (8+5*\b, 24) to (8+5*\b, 20) to (5*\b,20) to (5*\b,0);

\draw (0-\b,0-\c) to (30-\b,0-\c) to (30-\b,10-\c) to (10-\b,10-\c) to (10-\b,30-\c) to (0-\b,30-\c) to (0-\b,0-\c);
\draw (0-\b+\c,0-\c) to (30-\b+\c,0-\c) to (30-\b+\c,10-\c) to (10-\b+\c,10-\c) to (10-\b+\c, 18-\c) to (10-\b+\c+4, 18-\c) to (10-\b+\c+4, 22-\c) to (10-\b+\c, 22-\c) to (10-\b+\c,30-\c) to (0-\b+\c,30-\c) to (0-\b+\c,0-\c);
\draw (0-\b+2*\c,0-\c) to (30-\b+2*\c,0-\c) to (30-\b+2*\c,10-\c) to (10-\b+2*\c,10-\c) to (10-\b+2*\c, 18-\c) to (10-\b+2*\c-4, 18-\c) to (10-\b+2*\c-4, 22-\c) to (10-\b+2*\c, 22-\c) to (10-\b+2*\c,30-\c) to (0-\b+2*\c,30-\c) to (0-\b+2*\c,0-\c);
\draw (0-\b+3*\c,0-\c) to (30-\b+3*\c,0-\c) to (30-\b+3*\c,10-\c) to (22-\b+3*\c, 10-\c) to (22-\b+3*\c, 14-\c) to (18-\b+3*\c, 14-\c) to (18-\b+3*\c, 10-\c) to (10-\b+3*\c,10-\c) to (10-\b+3*\c, 18-\c) to (10-\b+3*\c+4, 18-\c) to (10-\b+3*\c+4, 22-\c) to (10-\b+3*\c, 22-\c) to (10-\b+3*\c,30-\c) to (0-\b+3*\c,30-\c) to (0-\b+3*\c,0-\c);
\draw (0-\b+4*\c,0-\c) to (30-\b+4*\c,0-\c) to (30-\b+4*\c,10-\c) to (22-\b+4*\c, 10-\c) to (22-\b+4*\c, 6-\c) to (18-\b+4*\c, 6-\c) to (18-\b+4*\c, 10-\c) to (10-\b+4*\c,10-\c) to (10-\b+4*\c, 18-\c) to (10-\b+4*\c+4, 18-\c) to (10-\b+4*\c+4, 22-\c) to (10-\b+4*\c, 22-\c)to (10-\b+4*\c,30-\c) to (0-\b+4*\c,30-\c) to (0-\b+4*\c,0-\c);
\draw (0-\b+5*\c,0-\c) to (30-\b+5*\c,0-\c) to (30-\b+5*\c,10-\c) to (22-\b+5*\c, 10-\c) to (22-\b+5*\c, 6-\c) to (18-\b+5*\c, 6-\c) to (18-\b+5*\c, 10-\c) to (10-\b+5*\c,10-\c) to (10-\b+5*\c, 18-\c) to (10-\b+5*\c-4, 18-\c) to (10-\b+5*\c-4, 22-\c) to (10-\b+5*\c, 22-\c)to (10-\b+5*\c,30-\c) to (0-\b+5*\c,30-\c) to (0-\b+5*\c,0-\c);

\end{tikzpicture}
\caption{The possible geometric shapes of a $1$-atom.}\label{fig:H1}
\end{figure}

Consider the buffered $1$-atom $B=A(w)$. Then $B$ has similar properties as $A$, and in particular Figure~\ref{fig:H1} continues to represent all possible geometric shapes of $B$. Let $E$ be a line segment on the boundary of $B$ with length $\geq \frac{1}{3}\epsilon d_1$. We will define a partial matching $M_1$ of $B$ with $E_1(A)=B-\dom(M_1)\subseteq E$ and $E_1(A)$ at least $C_0w^2$ from the endpoints of $E$.

We give the remaining proof for one of the possible shapes (the last one in Figure~\ref{fig:H1}); the proof for the rest of them is similar. By direct observation we note that $B$ can be decomposed into no more than $6$ rectangles $R_0$--$R_5$. Let $R_0$ be the rectangle such that $E\subseteq \partial R_0$, and let $E_0$ be the line segment of $\partial R_0$ containing $E$. We describe an algorithm to define $E_1(A)\subseteq E$ and a partial matching $M_1$ of $B-E_1(A)$.

\begin{figure}[h]
\begin{tikzpicture}[scale=0.15]
\pgfmathsetmacro{\a}{5}
\pgfmathsetmacro{\b}{30}
\pgfmathsetmacro{\c}{40}


\draw (0,0) to (10,0) to (10,18) to (0, 18) to (0,0);
\draw (10.5,0) to (18,0) to (18,10) to (10.5,10) to (10.5,0);
\draw (18.5,0) to (21.5,0) to (21.5,6) to (18.5,6) to (18.5,0);
\draw (22,0) to (30, 0) to (30,10) to (22,10) to (22, 0);
\draw (0,18.5) to (6, 18.5) to (6, 21.5) to (0, 21.5) to (0,18.5);
\draw (0,22) to (10,22) to (10,30) to (0,30) to (0,22);

\node at (5,9) {$R_0$};
\node at (8, 4) {$E_0$};
\node at (11.5,15) {$E$};
\node at (14, 5) {$R_1$};
\node at (20, 3) {$R_2$};
\node at (26, 5) {$R_3$};
\node at (3,19.8) {$R_4$};
\node at (5, 26) {$R_5$};

\end{tikzpicture}
\caption{Defining a partial matching on a buffered $1$-atom.}\label{fig:H1c}
\end{figure}
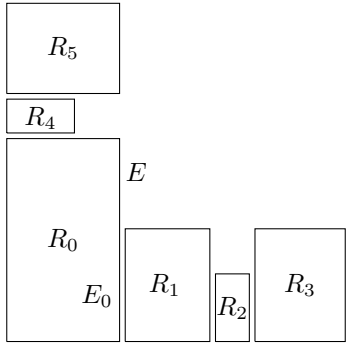

For example, if the rectangles are named as in Figure~\ref{fig:H1c}, we first compute the total charge $q_i$ of each of the rectangles $R_i$ for $i=0,\dots,5$. By Lemma~\ref{lem:rectangle}, $q_i\in\{0,1,-1\}$ for each $i$. Let $q$ be the sum of all $q_i$. Then $|q|\leq 6$. 

For two adjacent rectangles $R$ and $R^*$, we define a set $E_0(R, R^*)$ of unordered pairs $\{x, x^*\}$ where $x\in R$, $x^*\in R^*$, $x$ and $x^*$ are adjacent in the Schreier graph. For example, for $R_2$ and $R_3$, let $E_0(R_2,R_3)$ be a set of $|q_3|$ many unordered pairs $\{ x,x^*\}$ where $x\in R_2$, $x^*\in R_3$, $x$ and $x^*$ are adjacent in the Schreier graph, and $q(x^*)=q_3$ (which implies that $q(x)=-q_3$). Thus $E_0(R_2,R_3)$ is either empty (if $q_3=0$) or contains only one pair (if $q_3=\pm 1$). For $R_1$ and $R_2$, let $E_0(R_1,R_2)$ contain $|q_2+q_3|$ many pairs $\{x, x^*\}$ with $q(x^*)=\sgn (q_2+q_3)$. Similarly one can define $E_0(R_0,R_1)$, $E_0(R_4, R_5)$, $E_0(R_0,R_4)$, etc. By Lemma~\ref{lem:rectangle}, there are perfect matchings of $R_3-E_0(R_2,R_3)$, $R_2-E_0(R_2,R_3)-E_0(R_1,R_2)$, $R_5-E_0(R_4,R_5)$, $R_4-E_0(R_4,R_5)-E_0(R_1,R_4)$, etc. Finally, let $E_1(A)\subseteq E$ be a set of $|q|$ many elements $x$ with $q(x)=\sgn(q)$. Then by Lemma~\ref{lem:rectangle}, there is 
a perfect matching of $R_0-E_0(R_0,R_4)-E_0(R_0,R_1)-E_1(A)$. The desired partial matching $M_1$ is obtained by putting all these perfect matchings together, plus matching points of $E_0(R, R^*)$ according to their natural pairing. We note that Lemma~\ref{lem:rectangle} can be applied since the side lengths of all these rectangles are greater than $6$. Given that the length of $E$ is $\geq \frac{1}{3}\epsilon d_1\gg C_0w^2$, the set $E_1(A)$ can be chosen to be at least $C_0w^2$ from the endpoints of $E$.
\end{proof}

Assuming now that a Borel partial matching $M_{k-1}$ with domain $D_{k-1}$
has been defined which satisfies ${\mathbf H}_{k-1}$, we extend $M_{k-1}$
to a Borel partial matching $M_k$ with domain $D_k$ satisfying ${\mathbf H}_k$.

Consider a $k$-atom $A\in \sA_k$ and the corresponding $w$-buffered atom
$A(w)\subseteq A$. Let $\sS$ denote the set of $(k-1)$-atoms $S\in \sA_{k-1}$
which are contained within $A$. For each $S\in \sS$, we also have its
$w$-buffered counterpart $S(w)$. By ${\mathbf H}_{k-1}$, $M_{k-1}$ matches all
points of $S\in \sS$ except for a finite set $E_{k-1}(S)$ of size at
most $K$. Consider the following ``adjacency graph'' $G$ on $\sS$.
We set $(S_1,S_2)\in G$ iff there are $x_1 \in \partial S_1 \cap X_{k-1}$ and $x_2 \in \partial S_2 \cap X_{k-1}$ with $x_1$ (respectively $x_2$) at least $\frac{1}{6}\epsilon d_1\gg 3C_0w^2$
from the endpoints of the line segment containing it and forming part
of $\partial S_1$ (respectively $\partial S_2$), and with $\rho(x_1,x_2)=1$.
Note that the points $x_1,x_2$ are necessarily not in $\partial_w(A)$
by the above distance requirement. In fact, by our orthogonal marker region construction, two $(k-1)$-atoms are adjacent in $G$ if they contain adjacent line segments of lengths at least $\frac{1}{3}\epsilon d_1\gg 6C_0w^2$ on their $(k-1)$-level boundaries.

We note that $G$ is a connected graph on $\sS$. To see this, consider any two $S_1, S_2\in\sS$. Let $z_1\in S_1$ and $z_2\in S_2$ be arbitrary. Since $A$ is homeomorphic to a disk, there is a path $p$ from $z_1$ to $z_2$ that is contained within $A-\partial A$. Let $x_1, y_1$ be the first pair of consecutive points on $p$ such that $x_1\in S_1$ and $y_1\not\in S_1$. Then we have $\rho(x_1,y_1)=1$ and $x_1, y_1\in X_{k-1}$.  Suppose $y_1\in S_1'\in \sS$. Then we may modify the part of $p$ from $z_1$ to $y_1$ to become a path $p'$ from $S_1$ to $S_1'$ so that any pair of consecutive points on $p'$ crossing the boundaries of $(k-1)$-atoms are at least $\frac{1}{6}\epsilon d_1$ from the endpoints of the line segments containing them respectively and forming part of the boundaries of the respective $(k-1)$-atoms. If $S_1$ and $S_1'$ are adjacent in $G$, then this modification involves only adding points from $S_1\cup S_1'$; otherwise it involves adding points from a third $S_1''\in\sS$ which is adjacent to both $S_1$ and $S_1'$. The modified path gives rise to a path in the adjacency graph $G$ from $S_1$ to $S_1'$. Repeating this process for the rest of the points on $p$ will result in a path $p'$ that satisfies the $\frac{1}{6}\epsilon d_1$ distance requirement for all of its points, and therefore showing that $S_1$ and $S_2$ are connected in the graph $G$.

Choose a particular $S_0 \in \sS$ such that some line segment $E_0$
forming part of $\partial S_0$ is on $X_k$ (i.e. $E_0$ is part of a $k$-level boundary of $S_0$). Let $T\subseteq G$ be a maximal
spanning tree in $G$ with $S_0$ as its root. For any $S_0\neq S\in T$, if $S'$ is the parent node of $S$ in $T$, then there are adjacent line segments $E\in\partial S$ and $E'\in \partial S'$  of lengths at least $\frac{1}{3}\epsilon d_1$ on their $(k-1)$-boundaries. Fix such a
line segment $E_S$ for $\partial S$ and call it the {\em exit segment} for the region $S$.


Now that $T$ is a rooted tree, we may define the {\em rank} of $S\in \sS$ with respect to $T$ by induction as usual: if $S$ is a terminal node
of $T$, define $r(S)=0$; for non-terminal $S$, letting $P_1, \dots, P_m$ be all the child nodes of $S$, define
$$ r(S)=1+\max\{r(P_i)\,:\, i=1, \dots, m\}. $$
Also let $r(T)=r(S_0)$. In addition, for $S\in \sS$, let $t(S)$ be the number of
nodes in $T$ which are equal to $S$ or below $S$ (not just immediately below). Also let $t(T)=|T|$, the size of $T$.

By a subinduction on the rank of $S\in \sS$ with respect to $T$
we  successively extend the partial matching $M_{k-1}$.
We assume the following subinduction hypothesis. 
\medskip

{\bf Subinduction Hypothesis ${\mathbf H}_{k-1}(r-1)$:}
For each $S\in \sS$ with $r(S)\leq r-1$, the partial matching $M_{k-1}$
has been extended to $M_{k-1}(r-1)$, with its domain from 
$S(w)-E_{k-1}(S)$ to $(S\cap A(w))- E'_{k-1}(S)$ where $E_{k-1}'(S) \subseteq E_S$
is a set of size at most $K\cdot t(S)$ and consists of points
of distance at least $C_0w^2$
from the endpoints of $E_S$. If $r(S) <r-1$ then the points
of $E'_{k-1}(S)$ are matched in $M_{k-1}(r-1)$.
\medskip

We first verify the base case for the subinduction.

\begin{lem}\label{lem:Hk0} ${\mathbf H}_{k-1}(0)$ holds.
\end{lem}

\begin{proof} Let $S\in \sS$ be a terminal node of $T$. First we consider the case where all of the boundary of $S$ is on the $k-1$ level. In this case 
$S\cap A(w)=S$, and since $S$ is homeomorphic to a disk and its boundary consists of horizontal and vertical line segments of lengths at least $40w$, $S(w)$ has similar properties. We only need to extend the partial matching $M_{k-1}$ to $M_{k-1}(0)$ with its domain
from $S(w)-E_{k-1}(S)$ to $S-E'_{k-1}(S)$ for some $E'_{k-1}(S)\subseteq E_S$ of size at most $K$. Moreover, every element of $E_{k-1}'(S)$ will be at least $C_0w^2$ from the endpoints of $E_S$. By our inductive hypothesis ${\mathbf H}_{k-1}$, all points of $E_{k-1}(S)$ are at least $C_0w^2$ from the endpoints of their line segment. Note that in this case $\partial_w(S)$ forms a closed loop of rectangular polygonal pipes with thickness $w$ around the region $S(w)$. This is illustrated in Figure~\ref{fig:Hk0}.

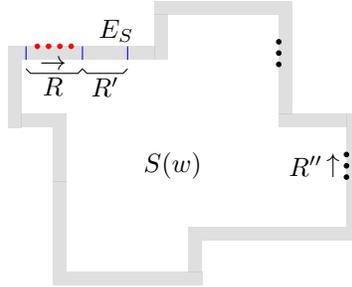
\begin{figure}[h]

\begin{tikzpicture}[scale=0.03]

\pgfmathsetmacro{\a}{3};


\draw[draw=none,fill=lightgray, opacity=0.5] (20-\a,0-\a) rectangle ((20+\a,40+\a);

\draw[draw=none,fill=lightgray, opacity=0.5] (20+\a,0-\a) rectangle ((80+\a,0+\a);

\draw[draw=none,fill=lightgray, opacity=0.5] (80-\a,0+\a) rectangle ((80+\a,20+\a);
\draw[draw=none,fill=lightgray, opacity=0.5] (80+\a,20-\a) rectangle ((150+\a,20+\a);
\draw[draw=none,fill=lightgray, opacity=0.5] (150-\a,20+\a) rectangle ((150+\a,70+\a);
\draw[draw=none,fill=lightgray, opacity=0.5] (120-\a,70-\a) rectangle ((150-\a,70+\a);
\draw[draw=none,fill=lightgray, opacity=0.5] (120-\a,70+\a) rectangle ((120+\a,120+\a);
\draw[draw=none,fill=lightgray, opacity=0.5] (65-\a,120-\a) rectangle ((120-\a,120+\a);
\draw[draw=none,fill=lightgray, opacity=0.5] (65-\a,100-\a) rectangle ((65+\a,120-\a);
\draw[draw=none,fill=lightgray, opacity=0.5] (0-\a,100-\a) rectangle ((65-\a,100+\a);
\draw[draw=none,fill=lightgray, opacity=0.5] (0-\a,70-\a) rectangle ((0+\a,100-\a);
\draw[draw=none,fill=lightgray, opacity=0.5] (0+\a,70-\a) rectangle ((20+\a,70+\a);
\draw[draw=none,fill=lightgray, opacity=0.5] (20-\a,40+\a) rectangle ((20+\a,70-\a);

\draw[fill] (120-\a,105) circle (1);
\draw[fill] (120-\a,100) circle (1);
\draw[fill] (120-\a,95) circle (1);

\draw[fill] (150-\a,55) circle (1);
\draw[fill] (150-\a,50) circle (1);
\draw[fill] (150-\a,45) circle (1);

\draw[fill,red] (10,100+\a) circle (1);
\draw[fill,red] (15,100+\a) circle (1);
\draw[fill,red] (20,100+\a) circle (1);
\draw[fill,red] (25,100+\a) circle (1);

\node at (70, 50) {$S(w)$};
\node at (45, 110) {$E_S$};
\node at (17, 95) {$\to$};
\draw (5,93) to (7,91) to (28,91) to (30,93) to (32,91) to (48,91) to (50,93);
\node at (17, 85) {$R$};
\node at (40, 85) {$R'$};
\node at (150-3*\a, 50) {$\uparrow$};
\node at (150-7*\a, 50) {$R''$};

\draw[blue] (5,100-\a) to (5, 100+\a);
\draw[blue] (30,100-\a) to (30, 100+\a);
\draw[blue] (50, 100-\a) to (50, 100+\a);

\end{tikzpicture}
\caption{The construction of $M_{k-1}(0)$ for a terminal node of $T$. The black dots are umatched points of $M_{k-1}$, the red dots are unmatched points of $M_{k-1}(0)$.}\label{fig:Hk0}
\end{figure}

In preparation for the construction of $M_{k-1}(0)$, we divide the buffer region $\partial_w(S)$ into three regions. First let $R_0$ (not shown in Figure~\ref{fig:H1}) be a rectangular region along the line segment $E_S$ that is at least $C_0w^2$ from the endpoints of $E_S$,  has length at least $6C_0w^2$, and does not contain any point adjacent to a point of $E_{k-1}(S)$. Such an $R_0$ exists since the size of $E_{k-1}(S)$ is at most $K$, but $E_S$ has length at least $\frac{1}{3}\epsilon d_1\gg 4(K+3)C_0w^2$. Next we divide $R_0$ into two adjacent regions $R$ and $R'$, with the length of $R$ at least $2C_0w^2$ and the length of $R'$ at least $2C_0w^2$. Assign an orientation for $R$ so that $R'$ follows $R$ in this orientation. Let $R''=\partial_w(B)-R-R'$ and assign an orientation for $R''$ that is opposite of the orientation of $R$. Note that, by our construction, all the unmatched points of $S(w)$, when matched by $M_{k-1}(0)$ to an adjcent point in $\partial_w(S)$, create ``charges" in $R''$. 

We now define a partial matching of $R''$ by following the orientation assigned to it, starting with the canonical matching, which has current 0. We use Lemma~\ref{lem:ub} as we turn corners, and use Lemma~\ref{lem:ucc} as we absorb new charges. After absorbing a number of charges we always apply Lemma~\ref{lem:standard} to turn the pattern back to the standard form. Since our charges are at least $C_0w_2$ from the endpoints, there is enough room to perform the standardization along a line segment. Because of Lemma~\ref{lem:bounded_charge_atoms}, and because $w>10K$, $R''$ has sufficiently large thickness to allow all charges to be absorbed. Thus we obtain a full partial matching of $R''$, and the end current has absolute value $\leq K$. Suppose the current is $\kappa$. 

Now we turn to the region $R$, and similarly define a partial matching of $R$ by following the orientation assigned to it, starting with the canonical matching, which has current 0. As we move along, we create exactly $|\kappa|$ many charges, with the total charge $-\kappa$, and use Lemma~\ref{lem:ucc} to absorb all these charges to create current $-\kappa$. Since $|\kappa|\leq K$ and the lengths of $R$ is at least $2C_0w^2$, we have enough room to absorb all the charges and to conclude this part of the construction with a partial mathcing of the standard form. 

Finally, we come to the region $R'$, which is of length at least $2C_0w^2$. We use Lemma~\ref{lem:ua} to join the partial matchings on $R$ and $R''$ to obtain a full partial matching of the entire $\partial_w(S)$. This gives the partial matching $M_{k-1}(0)$ as required.

We next consider the case where $\partial S$ contains a part on a higher level $X_\ell$ with $\ell>k-1$. In this case, the key observation is that, from Figure~\ref{fig:possible_shapes}, $S\cap A(w)$ is still a connected rectagular polygonal pipe. In this case we can easily modify the above construction to obtain the extension
$M_{k-1}(0)$ as required.
\end{proof}

For the general subinduction, we assume ${\mathbf H}_{k-1}(r-1)$ for $r<r(T)$, and extend $M_{k-1}(r-1)$
to $M_{k-1}(r)$ satisfying ${\mathbf H}_{k-1}(r)$. Fix an $S\in \sS$ with $r(S)=r$. Let $P_1, \dots, P_m$ be all the child nodes of $S$.
In particular, each of $P_1, \dots, P_m$ is adjacent to $S$ in the graph $G$, and therefore topologically adjacent to $S$ as a region.
By the subinduction hypothesis ${\mathbf H}_{k-1}(r-1)$, $E'_{k-1}(P_j)\subseteq E_{P_j}$ for each $j=1,\dots, m$, and each $E_{P_j}$ is adjacent to a line segment $F_j\subseteq \partial S$. Note that $F_1,\dots, F_m, E_S$ are all disjoint. 

In the first step of the extension from $M_{k-1}(r-1)$ to $M_{k-1}(r)$, we match the points of $E_{P_j}$ to their corresponding points on $F_j$ for all $j=1,\dots, m$. Now we are in a situation similar to Lemma~\ref{lem:Hk0} (see Figure~\ref{fig:qfig}), and with a similar construction as in the proof of Lemma~\ref{lem:Hk0} we obtain the extension $M_{k-1}(r)$ on $S$. 

\begin{figure}[h]

\begin{tikzpicture}[scale=0.03]

\pgfmathsetmacro{\a}{5};


\draw[draw=none,fill=lightgray, opacity=0.5] (0-\a,3) rectangle ((0+\a,40+\a);


\draw[draw=none,fill=lightgray, opacity=0.5] (80-\a,0+\a) rectangle ((80+\a,20+\a);
\draw[draw=none,fill=lightgray, opacity=0.5] (80+\a,20-\a) rectangle ((150+\a,20+\a);
\draw[draw=none,fill=lightgray, opacity=0.5] (150-\a,20+\a) rectangle ((150+\a,70+\a);
\draw[draw=none,fill=lightgray, opacity=0.5] (120-\a,70-\a) rectangle ((150-\a,70+\a);
\draw[draw=none,fill=lightgray, opacity=0.5] (120-\a,70+\a) rectangle ((120+\a,120+\a);
\draw[draw=none,fill=lightgray, opacity=0.5] (65-\a,120-\a) rectangle ((120-\a,120+\a);
\draw[draw=none,fill=lightgray, opacity=0.5] (65-\a,100-\a) rectangle ((65+\a,120-\a);
\draw[draw=none,fill=lightgray, opacity=0.5] (0-\a,100-\a) rectangle ((65-\a,100+\a);
\draw[draw=none,fill=lightgray, opacity=0.5] (0-\a,70-\a) rectangle ((0+\a,100-\a);
\draw[draw=none,fill=lightgray, opacity=0.5] (0+\a,70-\a) rectangle ((20+\a,70+\a);
\draw[draw=none,fill=lightgray, opacity=0.5] (20-\a,40-\a) rectangle ((20+\a,70-\a);
\draw[draw=none,fill=lightgray, opacity=0.5] (0+\a,40-\a) rectangle ((20-\a,40+\a);

\draw[fill] (-\a,25) circle (1);
\draw[fill] (-\a,20) circle (1);
\draw[fill] (-\a,15) circle (1);

\draw[fill] (150+\a,55) circle (1);
\draw[fill] (150+\a,50) circle (1);
\draw[fill] (150+\a,45) circle (1);

\draw[fill,red] (30,100+\a) circle (1);
\draw[fill,red] (35,100+\a) circle (1);
\draw[fill,red] (40,100+\a) circle (1);
\draw[fill,red] (45,100+\a) circle (1);

\node at (75, 60) {$S(w)$};
\node at (40,112) {$E'_{k-1}(S)$};
\node at (-30, 20) {$E_{k-1}'(P_j)$};
\node at (182, 48) {$E'_{k-1}(P_{j'})$};

\end{tikzpicture}
\caption{Constructing the partial matching $M_{k-1}(r)$ for $r<r(T)$. The black dots are initially umatched, and
we extend the matching to them and the shaded region except for the final unmatched red dots.}\label{fig:qfig}
\end{figure}
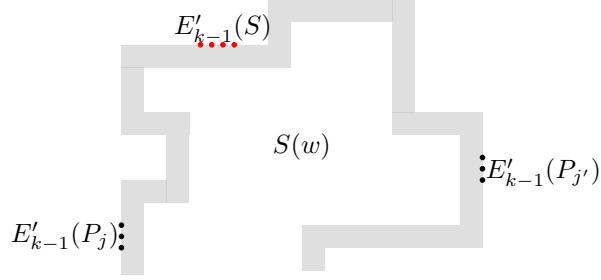 

We verify that the resulting partial matching satisfy the subinductive requirements. By the subinduction hypothesis ${\mathbf H}_{k-1}(r-1)$,
for each $j=1,\dots, m$, the size of each $E'_{k-1}(P_j)$ is at most $K\cdot t(P_j)$. Let $q$ be total charge of $\partial_w(S)$. By Lemma~\ref{lem:bounded_charge_atoms}, $q\leq K$. Thus the size of $E'_{k-1}(S)$ is at most 
$$ q+\sum_{j=1}^m K\cdot t(P_j)\leq K(1+\sum_{j=1}^m t(P_j)=K\cdot t(S). $$
Moreover, since $K\cdot t(S)\leq KC''<\frac{1}{2}w\ll \frac{1}{3}\epsilon d_1$, the thickness of $\partial_w(S)$ is sufficiently large for the
applications of Lemmas~\ref{lem:standard}, \ref{lem:ua}, \ref{lem:ub}, and \ref{lem:ucc}.


This completes the subinduction establishing ${\mathbf H}_{k-1}(r)$ for $r<r(T)$. 

\begin{figure}[h]

\begin{tikzpicture}[scale=0.03]

\pgfmathsetmacro{\a}{5};


\draw[draw=none,fill=lightgray, opacity=0.5] (0-\a,5) rectangle ((0+\a,40+\a);


\draw (-15,-8) to (95, -8);
\draw (-15,5) to (95,5);

\draw[draw=none,fill=lightgray, opacity=0.5] (80-\a,0+\a) rectangle ((80+\a,20+\a);
\draw[draw=none,fill=lightgray, opacity=0.5] (80+\a,20-\a) rectangle ((150+\a,20+\a);
\draw[draw=none,fill=lightgray, opacity=0.5] (150-\a,20+\a) rectangle ((150+\a,70+\a);
\draw[draw=none,fill=lightgray, opacity=0.5] (120-\a,70-\a) rectangle ((150-\a,70+\a);
\draw[draw=none,fill=lightgray, opacity=0.5] (120-\a,70+\a) rectangle ((120+\a,120+\a);
\draw[draw=none,fill=lightgray, opacity=0.5] (65-\a,120-\a) rectangle ((120-\a,120+\a);
\draw[draw=none,fill=lightgray, opacity=0.5] (65-\a,100-\a) rectangle ((65+\a,120-\a);
\draw[draw=none,fill=lightgray, opacity=0.5] (0-\a,100-\a) rectangle ((65-\a,100+\a);
\draw[draw=none,fill=lightgray, opacity=0.5] (0-\a,70-\a) rectangle ((0+\a,100-\a);
\draw[draw=none,fill=lightgray, opacity=0.5] (0+\a,70-\a) rectangle ((20+\a,70+\a);
\draw[draw=none,fill=lightgray, opacity=0.5] (20-\a,40-\a) rectangle ((20+\a,70-\a);
\draw[draw=none,fill=lightgray, opacity=0.5] (0+\a,40-\a) rectangle ((20-\a,40+\a);

\draw[fill] (-\a,85) circle (1);
\draw[fill] (-\a,80) circle (1);
\draw[fill] (-\a,75) circle (1);

\draw[fill] (150+\a,55) circle (1);
\draw[fill] (150+\a,50) circle (1);
\draw[fill] (150+\a,45) circle (1);

\draw[fill,red] (78,\a) circle (1);
\draw[fill,red] (82,\a) circle (1);

\node at (75, 60) {$S_0(w)$};
\node at (52,12) {$E'_{k-1}(S_0)$};
\node at (-30, 80) {$E_{k-1}'(P_j)$};
\node at (182, 48) {$E'_{k-1}(P_{j'})$};
\node at (20,-2) {$\partial_w(A)$};

\end{tikzpicture}
\caption{Constructing the partial matching $M_{k-1}(r)$ for $r=r(T)$. }\label{fig:qfigr}
\end{figure}
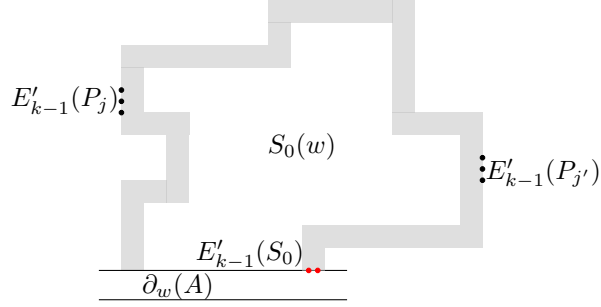 

For $r=r(T)$, there is no exit segment $E_{S_0}$; instead, the ``exit points" $E'_{k-1}(S_0)$ need to be a part of $\partial A(w)$ to maintain
the induction hyothesis ${\mathbf H}_k$ (see Figure~\ref{fig:qfigr}). Here the extension $M_{k-1}(r)$ is still constructed similar to the proof of Lemma~\ref{lem:Hk0}. In fact, noting that $\partial_w(S_0)\cap A(w)$ is a connected rectangular polygonal pipe, we arbitrarily fix an orientation for it. Then we define the partial matching $M_{k-1}(r)$ again by starting from the canonical matching, turning corners, and absorbing charges. Along the way we also put the partial matching into the standard form whenever we can. At the end of the construction the charges will show up at the end of the pipe as unmatched points. By a similar computation as above, we have $K\cdot t(T)=K\cdot |T|\leq KC''<\frac{1}{2}{w}\ll \frac{1}{3}\epsilon_1d_1$, and thus the thickness of $\partial_w(S_0)$, which is $w$, is large enough for the algorithms to be applied. 

To show that the inductive hypothesis ${\mathbf H}_k$ is maintained, we need to verify that $E'_{k-1}(S_0)=E_k(A)$ has size $\leq K$ and that the charges in $E_{k-1}'(S_0)$ are at least $C_0w^2$ from the endpoints of their line segment as a part of $\partial A(w)$. 
For the size of $E_{k-1}'(S_0)$, note that the total charge of $A$ is at most $K$ from Lemma~\ref{lem:bounded_charge_atoms}. Since $A(w)$ has the same shape, the total charge of $A(w)$ is at most $K$. Now $M_{k-1}(r)$ has matched all the points of $A(w)$ except the points of $E'_{k-1}(S_0)$, we have that $E_{k-1}' (S_0)$ has size $\leq K$.

We next verify that the charges in $E_{k-1}'(S_0)$ are at least $C_0w^2$ from the endpoints of their line segment as a part of $\partial A(w)$. For this pick any point $x_0\in E'_{k-1}(S)$, and note that any other point of $E'_{k-1}(S)$ is within distance $w$ of $x_0$. Let $F$ be a line segment of $\partial S$ that is within distance $w$ of $x_0$ and $E$ be a line segment of $\partial A$ that is within distance $w$ of $x_0$. Since $w\ll \epsilon d_1$, $F$ and $E$ meets perpendicularly at a point $y_0$ that is within distance $w$ of $x_0$. By the construction in the proof of Lemma~\ref{lem:bounded_charge}, since $E$ is on level $k$ and $F$ is on level $k-1$, the endpoints of $E$ is at least $\frac{1}{3}\epsilon d_1$ from $y_0$. Since $\frac{1}{3}\epsilon d_1\gg C_0w^2$, it follows that all points of $E_{k-1}'(S)$ are at least $C_0w^2$ from the endpoints of $E$.

This completes the subinduction and hence the induction establishing ${\mathbf H}_k$ for all $k$, giving the Borel
partial matchings $M_k$. Letting $M=\bigcup_k M_k$ gives the desired matching of
$F(2^{\Z^2})$, and completes the proof of Theorem~\ref{thm:mt}.

\section{Borel Linings} \label{sec:cls}

In this section we show that there exists a Borel lining of $F(2^{\Z^2})$. By a {\em lining} we mean a subgraph of the Schreier graph of $F(2^{\Z^2})$ which is a Hamiltonian path on each equivalence class. More generally, consider a free action of $\Z^n$ on a Polish space $X$ and the associated Schreier graph on $X$. Following \cite{GJKSinf}, a {\em line section} is a subgraph $S$ of the Schreier graph on $X$ where each vertex in $S$ has degree $2$. If $S$ is a line section, we call each connected component of $S$ an {\em $S$-line}. A line section $S$ is {\em complete} if $S\cap [x]\neq \varnothing$ for all $x\in X$. A line section $S$ is {\em single} if for each $x\in X$, $S\upharpoonright [x]$ is a nonempty single $S$-line. Thus a {\em lining} $S$ is a single line section where the vertex set of $S$ is the entire space, and in particular it is a complete line section. 

A line section $S$ on $X$ is {\em Borel} ({\em clopen}, resp.) if for each generator $e$ of $\Z^n$, the set $\{x\in X\,:\, (x, e\cdot x)\in S\}$ is Borel (clopen, resp.). In \cite{GJKSinf} the authors showed that there do not exist clopen single line sections of $F(2^{\Z^2})$; in particular, there are no clopen linings of $F(2^{\Z^2})$. Here, in contrast, we prove the following theorem.


\begin{thm} \label{thm:cls}
Every free action of the group $\Z^2$ on a Polish space has a Borel lining.
\end{thm}



The rest of this section is devoted to a proof of Theorem~\ref{thm:cls}.
The proof is similar in spirit to that for Borel matchings, Theorem~\ref{thm:mt}, and
as with that proof we must establish some combinatorial facts about line sections that we 
will need for the proof.

The notions of line section and lining extend naturally to the Cayley graph of $\Z^2$, and more generally finite subsets of $\Z^2$. So for a finite subset $S$ of $\Z^2$, a {\em lining} of $S$ means a finite simple path in the Cayley graph of $\Z^2$ 
which passes through each element of $S$.

As is the case with matchings, we fix a parity on the points of $\Z^2$ in discussing the combinatorial facts below. In our construction of a Borel lining, all of our arguments 
will be local and not depend on which of the two possible parities on the class we consider. We fix odd positive integers $p_0$, $q_0>1$ and consider a rectangle of dimensions $p_0 \times q_0$. This will have 
a total charge of $+ 1$ or $-1$. Such a rectangle will be called a ``parity change'' region, and will be used frequently in the 
following arguments. So, by adding or subtracting such a region from a rectangle $R$, we can change the 
total charge of the region by either $1$ or $-1$, depending on the placement of the parity change region.

The following fact is the main techical tool we need.

\begin{thm} \label{thm:mtt}
There is a constant $a_0$ such that the following holds. Suppose $R\subseteq \Z^2$ is a finite rectangle 
and $R'$ is obtained from $R$ by adding or subtracting 
at most $k$ many $p_0\times q_0$ parity change regions along the boundary of $R$,
and with the distance between any two such parity change regions at least $a_0$,
and the distance between any paritty change region and a corner of $R$ at least $k a_0$.
Suppose $R'$ has total charge $0$. Then given two adjacent points $x$, $y$ on the 
boundary of $R'$ at least $a_0$ away from a corner, 
there is a lining of $R'$ which starts at $x$ and ends at $y$. 
\end{thm}

Figure~\ref{fig:mtt} illustrates the statement of Theorem~\ref{thm:mtt}.

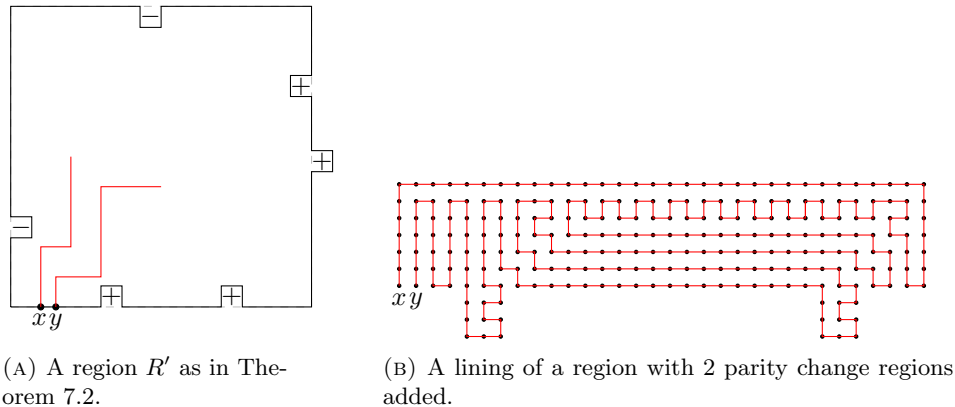
\begin{figure}[h] 
\centering

\begin{subfigure}{0.3\textwidth}

\begin{tikzpicture}[scale=0.04]
\pgfmathsetmacro{\w}{7};

\draw[lightgray, dashed] (0,0) to (100,0) to (100,100) to (0,100) to (0,0);

\draw (0,0) to (30,0) to (30,\w) to (30+\w,\w) to (30+\w,0) to (70,0) to (70,\w) to (70+\w,\w)
to (70+\w,0) to (100,0) to (100,45) to (100+\w,45) to (100+\w,45+\w) to (100,45+\w)
to (100,70) to (100-\w,70) to (100-\w,70+\w) to (100,70+\w) to (100,100)
to (50,100) to (50,100-\w) to (50-\w,100-\w) to (50-\w,100) 
to (0,100) to (0,30) to (\w,30) to (\w,30-\w) to (0,30-\w) to (0,0);

\node (a) at (30+\w/2,\w/2) {$+$};
\node (b) at (70+\w/2,\w/2) {$+$};
\node (c) at (100+\w/2,45+\w/2) {$+$};
\node (d) at (100-\w/2,70+\w/2) {$+$};
\node (e) at (50-\w/2,100-\w/2) {$-$};
\node (f) at (\w/2,30-\w/2) {$-$};

\draw[fill]  (10,0) circle  (1);
\draw[fill]  (15,0) circle  (1);

\draw[red] (10,0) to (10,20) to (20,20) to (20,50);
\draw[red] (15,0) to (15,10) to (30,10) to (30,40) to (50,40);

\node(g) at (9.5,-4) {$x$};
\node(h) at (15,-4.9) {$y$};

\end{tikzpicture} 

\caption{A region $R'$ as in Theorem~\ref{thm:mtt}.}
\label{fig:imtt}
\end{subfigure}
\hfill
\begin{subfigure}{0.6\textwidth}

\begin{tikzpicture}[scale=0.225]

\foreach \i in {0,...,30} 
\foreach \j in {0,...,5} 
{
\draw[fill] (\i,\j) circle (0.1);
}

\foreach \i in {3,...,5} 
\foreach \j in {-3,...,-1} 
{
\draw[fill] (\i,\j) circle (0.1);
}

\foreach \i in {24,...,26} 
\foreach \j in {-3,...,-1} 
{
\draw[fill] (\i,\j) circle (0.1);
}

\foreach \j in { 0,...,5}
\draw[fill] (-1,\j) circle(0.1);

\foreach \j in { -1,...,30}
\draw[fill] (\j,6) circle(0.1);

\draw[red] (0,0) to (0,5) to (1,5) to (1,0) to (2,0) to (2,5) to (3,5) to (3,-3)
to (5,-3) to (5,-2) to (4,-2) to (4,-1) to (5,-1) to (5,0) to (4,0)
to (4,5) to (5,5) to (5,1) to (6,1) to (6,0)
to (24,0) to (24,-3) to (26,-3) to (26,-2) to (25,-2) to (25,-1) to (26,-1) to (26,0) to (25,0)
to (25,1) to (7,1) to (7,2) to (6,2) to (6,5)
to (8,5) to (8,4) to (7,4) to (7,3) to (8,3) to (8,2)
to (26,2) to (26,1) to (27,1) to (27,0) to (28,0) to (28,2) to (27,2) to (27,3) to (9,3)
to (9,4) to (9,5);
\foreach \k in {9,11,...,21} 
{\draw[red] (\k,5) to (\k+1,5) to (\k+1,4) to (\k+2,4) to (\k+2,5);
}

\draw[red] (23,4) to (23,5) to (24,5) to (24,4) to (25,4) to (25,5) to (26,5) to (26,4)
to (27,4) to (27,5) to (29,5) to (29,4) to (28,4) to (28,3) to (29,3) to (29,0) to (30,0) to (30,6) to (-1,6) to (-1,0);

\node(g) at (-1,-.8) {$x$};
\node(h) at (0,-.9) {$y$};

\end{tikzpicture}

\caption{A lining of a region with $2$ parity change regions added.}
\label{fig:mttb}
\end{subfigure}

\caption{An illustration of Theorem~\ref{thm:mtt}.}
\label{fig:mtt}
\end{figure}

An easy consequence of Theorem~\ref{thm:mtt} is the following. 

\begin{cor} \label{cor:mtt}
There is a constant $a_0$ such that if $R$, $R'$ are as in 
Theorem~\ref{thm:mtt}, and if the total 
charge of $R'$ is $0$,  then for any two adjacent points $x,y$ on the boundary of $R'$
at least $k a_0$ away from a corner of $R$, 
and a point $z$ on the boundary at least $k a_0$ away from a corner or a point in $\{ x,y\}$,
there are two disjoint line sections $L_1$, $L_2$ of $R'$ such that $L_1$ starts at $x$, $L_2$ starts at $y$,
every vertex of $R'$ is on $L_1 \cup L_2$, and $L_1$, $L_2$ end at adjacent vertices on the boundary of $R'$
one of which is the point $z$.  
\end{cor}

\begin{proof}
Let $L$ be a line section for $R'$ which begins at $x$ and ends at $y$. One of the edges of $L$
contains the point $z$ and connects $z$ to an adjacent point on the boundary of $R'$. If we remove this 
edge from $L$, we obtain two disjoint line section of $R'$, say $L_1$ and $L_2$ which start at $x$ and $y$
respectively, and one of which ends at $z$. 
\end{proof}

We note that the statement of Corollary~\ref{cor:mtt} immediately implies the statement 
of Theorem~\ref{thm:mtt} as well, simply by adding an edge between the terminating vertices of 
$L_1$ and $L_2$. So, in the following constructions it suffices to show either version.

First we note that if $R'$ is a rectangle (equivalently $k=0$), then the result is easy to see directly. In this case, since $R'$ has total charge $0$, it must have an even number of nodes. A proof by picture is given in Figure~\ref{fig:rc}. Without loss of generality we assume $x, y$ are on the lower horizontal edge of $R'$. Note that $x, y$ can be any pair of adjacent points on that edge. In the case of an even-by-odd rectangle, also note that the method works equally well if the adjacent points $x, y$ are on either of the vertical edges.

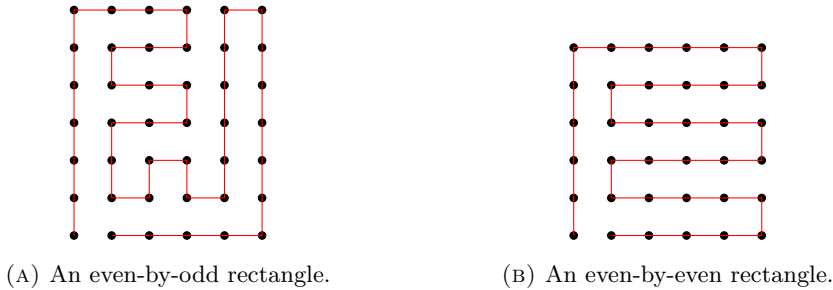
\begin{figure}[h]
\centering

\begin{subfigure}{0.5\textwidth}

\begin{center}
\begin{tikzpicture}[scale=0.5]

\foreach \i in {0,...,5} 
\foreach \j in {0,...,6} 
{
\draw[fill] (\i,\j) circle (0.1);
}

\draw[red] (0,0) to (0,6) to (3,6) to (3,5) to (1,5) to (1,4) to (3,4) to (3,3) to (1,3) to (1,1) to (2,1) to (2,2) to (3,2) to (3,1) to (4,1)
to (4,6) to (5,6) to (5,0) to (1,0);

\end{tikzpicture}
\end{center}
\caption{An even-by-odd rectangle.}
\end{subfigure}
\hfill
\begin{subfigure}{0.45\textwidth}
\begin{center}
\begin{tikzpicture}[scale=0.5]

\foreach \i in {0,...,5} 
\foreach \j in {0,...,5} 
{
\draw[fill] (\i,\j) circle (0.1);
}

\draw[red] (0,0) to (0,5) to (5,5) to (5,4) to (1,4) to (1,3) to (5,3) to (5,2)
to (1,2) to (1,1) to (5,1) to (5,0) to (1,0);

\end{tikzpicture}
\end{center}
\caption{An even-by-even rectangle.}
\end{subfigure}
\caption{Proof of Theorem~\ref{thm:mtt} for a rectangle $R'$.} \label{fig:rc}
\end{figure}

Although not needed for the proof of Theorem~\ref{thm:mtt}, for this $k=0$ case we would like to state a stronger result which we will need later.

\begin{lem} \label{lem:k=0} Let $R$ be a rectangle of total charge $0$. Given any two disjoint pairs of adjacent points $x, y$ and $x',y'$ on the boundary of $R$, there exist two disjoint line sections $L_1$ and $L_2$ of $R$ such that every vertex of $R$ is on $L_1\cup L_2$, $L_1$ starts with $x$ and ends with one of $x',y'$ and $L_2$ starts with $y$ and ends with the other one of $x',y'$.
\end{lem}

\begin{proof} For this lemma we need more symmetry than noted above. In Figure~\ref{fig:rc} (a), by symmetry, we may assume neither pair is the pair of points on the upper edge that are not adjacent in the lining. Then it is obvious that $L_1$ and $L_2$ can be obtained by omitting the edges in between $x$ and $y$ and in between $x'$ and $y'$. For Figure~\ref{fig:rc} (b), by symmetry we may assume neither pair is one of the pairs of adjacent points on the right edge that are not adjacent in the lining. For bigger values of side lengths the proof can be generalized. For smaller values of side lengths the statement can be checked directly.
\end{proof}

Next we consider the case $k=1$. In this case $R'$ is obtained from an odd-by-odd rectangle $R$ with one parity change region added or subtracted
(so the total charge of $R'$ is $0$). A proof by picture is presented in Figure~\ref{fig:orc}. 
The red lines indicate linings of $R'$ with end points $x, y$ on the lower horizontal edges. 
When $x, y$ are at a different position on the edge of $R'$, we use blue lines to illustrate alternative linings. 

\begin{figure}[h]

\begin{subfigure}{0.45\textwidth}
\begin{center}
\begin{tikzpicture}[scale=0.4]

\foreach \i in {0,...,3} 
\foreach \j in {0,...,10} 
{
\draw[fill] (\i,\j) circle (0.1);
}

\foreach \i in {7,...,10} 
\foreach \j in {0,...,10} 
{
\draw[fill] (\i,\j) circle (0.1);
}

\foreach \i in {4,...,6} 
\foreach \j in {3,...,10} 
{
\draw[fill] (\i,\j) circle (0.1);
}

\foreach \i in {0,...,10}
\foreach \j in {11,12}
{
\draw[fill] (\i,\j) circle (0.1);
}

\foreach \i in {4,...,6} 
\foreach \j in {0,...,2} 
{
\draw[fill, lightgray] (\i,\j) circle (0.1);
}

\draw[red] (0,0) to (0,3) to (2,3) to (2,4) to (0,4) to (0,5) to (2,5) to (2,6)
to (0,6) to (0,7) to (2,7) to (2,8) to ((0,8) to (0,9) to (2,9) to (2,10) to (0,10);

\draw[red] (1,0) to (3,0) to (3,1) to (1,1) to (1,2) to (3,2) to (3,3) to (7,3) to (7,0)
to (10,0) to (10,1) to (8,1) to (8,3) to (9,3) to (9,2) to (10,2) to (10,4) to (3,4) to (3,5) to (10,5)
to (10,6) to (3,6) to (3,7) to (10,7) to (10,8) to (3,8) to (3,9) to (10,9) to (10,10) to (3,10);

\draw[red] (0,10) to (0,12);

\draw[red] (3,10) to (3,11) to (10,11) to (10,12) to (2,12) to (2,11) to (1,11) to (1,12) to (0,12);

\draw[blue] (-0.2,9.8) to (1,9.8) to (1,9.2) to (1.8,9.2) to (1.8,10.8) to (0.2,10.8) to (0.2,12.2) to (2,12.2);

\draw[fill, blue] (-0.2,9.8) circle (0.1);
\draw[fill, blue] (-0.2,9.2) circle (0.1);
\draw[blue] (-0.2,9.2) to (-0.2,8);

\end{tikzpicture}
\end{center}
\end{subfigure}
\hfill
\begin{subfigure}{0.45\textwidth}
\begin{center}
\begin{tikzpicture}[scale=0.4]

\foreach \i in {0,...,10} 
\foreach \j in {3,...,11} 
{
\draw[fill] (\i,\j) circle (0.1);
}




\foreach \i in {4,...,6} 
\foreach \j in {0,...,2} 
{
\draw[fill] (\i,\j) circle (0.1);
}

\draw[red] (0,3) to (0,5) to (2,5) to (2,6)
to (0,6) to (0,7) to (2,7) to (2,8) to ((0,8) to (0,9) to (2,9) to (2,10) to (0,10);

\draw[red]  (10,5) to (3,5)
 to (3,6) to (10,6) to (10,7) to (3,7) to (3,8) to (10,8) to (10,9) to (3,9) to (3,10) to (10, 10);

\draw[red] (0,10) to (0,11) to (10,11) to (10,10);

\draw[red] (1,3) to (1,4) to (2,4) to (2,3) to (3,3) to (3,4) to (4,4) to (4,0) to (6,0) to (6,1) to (5,1) to (5,2) to (6,2) to (6,3) to (5,3) to (5,4) to (7,4) to (7,3) to (8,3) to (8,4) to (9,4) to (9,3) to (10,3) to (10,5);

\draw[blue] (1.2, 2.8) to (0.2,2.8) to (0.2,4.8) to (1,4.8) to (1,4.2) to (2,4.2) to (2,4.8);

\draw[fill, blue] (1.2,2.8) circle (0.1);
\draw[fill, blue] (1.8,2.8) circle (0.1);
\draw[blue] (1.8,2.8) to (3,2.8);
\end{tikzpicture}
\end{center}
\end{subfigure}

\caption{An odd-by-odd rectangle with one parity region subtracted or added.}

\label{fig:orc}
\end{figure}
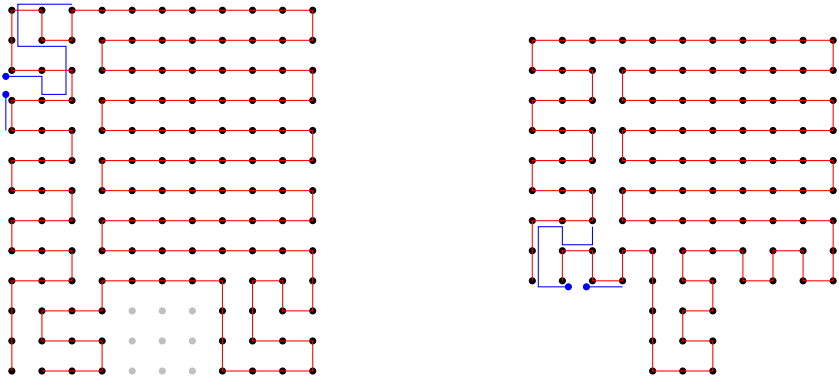

The case $k=2$, where $R'$ is an even rectangle with two parity change regions added or removed (these
parity change regions must have opposite parities so that $R'$ has total charge $0$), 
is also not difficult, and is illustrated in Figure~\ref{fig:mttb}. 
Finding alternative linings for different pairs of endpoints $x, y$ 
takes a moment of reflection but is not difficult. We note that this result holds 
provided the minimum side length of the rectangle $R$ 
exceeds a fixed small value $a_0$.

Next we treat the general case $k>2$. As a tool, we show that if $R'$ is a rectangle with a number of parity change regions along 
one side, then we may ``rearrange'' the locations of the parity change regions along that side. 
More precisely, consider a region $R'$ as shown in Figure~\ref{fig:prearra}. Along the bottom
edge of $R'$ we have parity change regions in arbitrary locations, but along the top edge we have 
placed the regions far to the left (but still consistent with the requirements on $R'$). In
Figure~\ref{fig:prearrb} we have just moved one of the parity change regions to the left. We show that in these cases we can prove Theorem~\ref{thm:mtt}.

\begin{figure}[h] 

\begin{subfigure}{0.4\textwidth}

\begin{tikzpicture}[scale=0.05]

\draw (0,0) to (50,0) to (50,-5) to (55,-5) to (55,0) to (70,0)
to (70,-5) to (75,-5) to (75,0) to (85,0) to (85,-5) to (90,-5) to (90,0) to (100,0);

\draw (0,0) to (0,30) to (5,30) to (5,25) to (10,25) to (10,30) to (15,30) 
to (15,25) to (20,25) to (20,30) to (25,30) to (25,25) to (30,25) to (30,30) to 
(100,30) to (100,0);

\end{tikzpicture}
\caption{Rearranging the parity change regions.}
\label{fig:prearra}
\end{subfigure}
\hfill
\begin{subfigure}{0.4\textwidth}
\begin{tikzpicture}[scale=0.05]

\draw (0,0) to (50,0) to (50,-5) to (55,-5) to (55,0) to (70,0)
to (70,-5) to (75,-5) to (75,0) to (85,0) to (85,-5) to (90,-5) to (90,0) to (100,0);

\draw (0,0) to (0,30) to (5,30) to (5,25) to (10,25) to (10,30) to (70,30) 
to (70,25) to (75,25) to (75,30) to (85,30) to (85,25) to (90,25) to (90,30) to 
(100,30) to (100,0);

\end{tikzpicture}

\caption{Moving just one parity change region.}
\label{fig:prearrb}
\end{subfigure}
\caption{}
\label{fig:prearr}
\end{figure}

Note the total charge of the regions $R'$ shown in Figure~\ref{fig:prearr}
is the same as the underlying rectangle $R$. We assume the total charge of $R$
is $0$, that is, $R$ has even size (this will be the case in our applications; if $R$
has odd size, then a similar argument can be given where we add one more parity change 
region to the top edge of $R'$).

First we consider the case of moving a single parity change region as shown in Figure~\ref{fig:prearrb}. 
A lining for such an $R'$ is easily obtained from the argument
for $\leq 2$ parity change regions as illustrated in Figure~\ref{fig:osr}. We start with the terminals 
$x,y$ and move left to right through the subregions using Corollary~\ref{cor:mtt} 
to end in a pair of adjacent terminals of the boundary of the next subregion. For the last subregion,
we use Theorem~\ref{thm:mtt} for the $\leq 2$ parity change region case to get a complete 
line section for it.

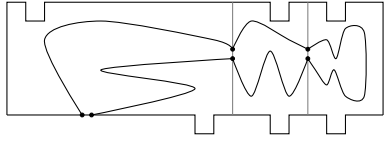
\begin{figure}[h]
\begin{tikzpicture}[scale=0.05]

\draw (0,0) to (50,0) to (50,-5) to (55,-5) to (55,0) to (70,0)
to (70,-5) to (75,-5) to (75,0) to (85,0) to (85,-5) to (90,-5) to (90,0) to (100,0);

\draw (0,0) to (0,30) to (5,30) to (5,25) to (10,25) to (10,30) to (70,30) 
to (70,25) to (75,25) to (75,30) to (85,30) to (85,25) to (90,25) to (90,30) to 
(100,30) to (100,0);

\draw[gray] (60,0) to (60,30);
\draw[gray] (80,0) to (80,30);

\draw[fill] (20,0) circle (0.5);
\draw[fill] (22.5,0) circle (0.5);
\draw[fill] (60,15) circle (0.5);
\draw[fill] (60,17.5) circle (0.5);
\draw[fill] (80,15) circle (0.5);
\draw[fill] (80,17.5) circle (0.5);

\draw plot[smooth] coordinates { (20,0) (10,20) (25,25) (55,20) (60,17.5)};
\draw plot[smooth] coordinates { (22.5,0) (50,7) (25,12) (60,15)};

\draw plot[smooth] coordinates { (60,17.5) (65, 25) (75,20) (80,17.5)};
\draw plot[smooth] coordinates { (60,15) (65, 5) (70,17) (75, 5) (80,15)};

\draw plot[smooth] coordinates { (80,17.5) (85,20) (87.5, 15) (90, 23) (95,22) (95, 5) (90,5) (87, 12)
(85,8) (82, 12) (80, 15)};

\end{tikzpicture}

\caption{Moving just one parity change region.}
\label{fig:osr}
\end{figure}

To accomplish the more general parity change rearrangement as in Figure~\ref{fig:prearra},
we vertically stack the appropriate number of rectangles (equal to the numer of parity change regions) 
and appply the algorithm of Figure~\ref{fig:prearrb} in succession. The first two steps of this procedure
for the region $R'$ of Figure~\ref{fig:prearra} are shown in Figure~\ref{fig:prearrc}.

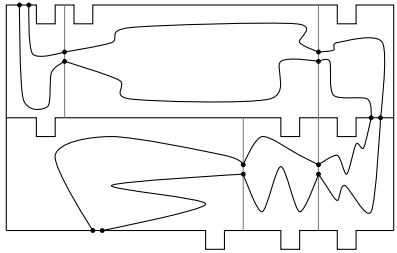
\begin{figure}[h]
\begin{tikzpicture}[scale=0.05]

\draw (-3,0) to (50,0) to (50,-5) to (55,-5) to (55,0) to (70,0)
to (70,-5) to (75,-5) to (75,0) to (85,0) to (85,-5) to (90,-5) to (90,0) to (100,0);

\draw (-3,0) to (-3,30) to (5,30) to (5,25) to (10,25) to (10,30) to (70,30) 
to (70,25) to (75,25) to (75,30) to (85,30) to (85,25) to (90,25) to (90,30) to 
(100,30) to (100,0);

\draw (-3,30) to (-3,60) to (5,60) to (5,55) to (10,55) to (10,60) to (15,60) 
to (15,55) to (20,55) to (20,60) to (85,60) to (85,55) to (90,55) to (90,60)
to (100,60) to (100,30);

\draw[gray] (60,0) to (60,30);
\draw[gray] (80,0) to (80,30);

\draw[gray] (12.5,30) to (12.5,60);
\draw[gray] (80,30) to (80,60);

\draw[fill] (20,0) circle (0.5);
\draw[fill] (22.5,0) circle (0.5);
\draw[fill] (60,15) circle (0.5);
\draw[fill] (60,17.5) circle (0.5);
\draw[fill] (80,15) circle (0.5);
\draw[fill] (80,17.5) circle (0.5);
\draw[fill] (94,30) circle (0.5);
\draw[fill] (96.5,30) circle (0.5);
\draw[fill] (80,45) circle (0.5);
\draw[fill] (80,47.5) circle (0.5);
\draw[fill] (12.5,47.5) circle (0.5);
\draw[fill] (12.5,45) circle (0.5);
\draw[fill] (3,60) circle (0.5);
\draw[fill] (0.5,60) circle (0.5);

\draw plot[smooth] coordinates { (20,0) (10,20) (25,25) (55,20) (60,17.5)};
\draw plot[smooth] coordinates { (22.5,0) (50,7) (25,12) (60,15)};

\draw plot[smooth] coordinates { (60,17.5) (65, 25) (75,20) (80,17.5)};
\draw plot[smooth] coordinates { (60,15) (65, 5) (70,17) (75, 5) (80,15)};

\draw plot[smooth] coordinates { (80,17.5) (85,20) (87.5, 15) (90, 23) (92,22) (94, 30)};
\draw plot[smooth] coordinates { (96.5,30) (94,5) (87, 12) (85,8) (82, 12) (80, 15)};

\draw plot[smooth] coordinates { (94,30) (93,35) (84,36) (83,45) (80,45)};
\draw plot[smooth] coordinates { (96.5,30) (97,50) (85,50) (84,48) (80, 47.5)};

\draw plot[smooth] coordinates { (80,47.5) (75,50) (74,55) (30,54) (25,50) (12.5, 47.5)};
\draw plot[smooth] coordinates { (80,45) (70,45) (67, 35) (30, 35) (27, 40) (12.5, 45)};

\draw plot[smooth] coordinates { (12.5,47.5) (4,47) (3,60)};
\draw plot[smooth] coordinates { (12.5,45) (9,42) (8,33) (1.5,35) (0.5, 60)};

\end{tikzpicture}

\caption{The general parity rearrangement algorithm.}
\label{fig:prearrc}
\end{figure}

We summarize this discussion into the following lemma.

\begin{lem} \label{lem:parearr}
There is a constant $a_0$ such that the following holds. Suppose 
$R$ is a rectangle with total charge $0$ and with minimum side length $\geq a_0$, 
and $R'$ is obtained from $R$ by adding or subtracting $\leq k$ parity change regions from opposite
edges of $R$, and that the parities of these corresponding regions (from left to right) 
are the same for each of these two edges. If all of the parity change regions are at least $ k a_0$ 
from a corner of $R$, then there is a lining of $R'$. 
\end{lem}

\begin{proof}
We use $\leq k$ applications of the algorithm described above as shown in Figure~\ref{fig:prearrc}.
If the width of the rectangle $R$ is at least $k a_0$, then the procedure will give a lining of $R'$.
\end{proof}

We next prove a parity cancellation lemma.

\begin{lem} \label{lem:parc}
There is a constant $a_0$ such that the following holds. Suppose 
$R$ is a rectangle with total charge $0$ and with minimum side length $\geq a_0$, 
and $R'$ is obtained from $R$ by adding or subtracting $\leq k$ parity change regions
from one side of $R$, and subtracting or adding parity change regions on the opposite edge 
of the same parities (left to right) except that we delete a pair of adjacent parity change regions of 
opposite parities. Then there is a lining of $R'$. 
\end{lem}

\begin{proof}
Let $R$ and $R'$ be as in the statement, see Figure~\ref{fig:parc} for an illustration. 
As in the proof of Lemma~\ref{lem:parearr}, we divide $R'$ into subregions and proceed 
from left to right using the result for regions with $\leq 2$ parity change regions. 

\end{proof}

\begin{figure}[h] 

\begin{subfigure}{0.4\textwidth}

\begin{tikzpicture}[scale=0.05]

\draw (0,0) to (20,0) to (20,-5) to (25,-5) to (25,0) to (50,0) to (50,-5) to (55,-5) to (55,0)
to (65,0) to (65,-5) to (70,-5) to (70,0) to (80,0) to (80,-5) to (85,-5) to (85,0) to (100,0)
to (100,30);

\draw (0,0) to (0,30) to (20,30) to (20,25) to (25,25) to (25,30) to  
(80,30) to (80,25) to (85,25) to (85,30) to (100,30);

\node (a) at (52.5,-2) {$+$};
\node (b) at (67.5,-2) {$-$};

\end{tikzpicture}
\caption{A region $R'$ as in Lemma~\ref{lem:parc}.}
\label{fig:parca}
\end{subfigure}
\hfill
\begin{subfigure}{0.4\textwidth}
\begin{tikzpicture}[scale=0.05]

\draw (0,0) to (20,0) to (20,-5) to (25,-5) to (25,0) to (50,0) to (50,-5) to (55,-5) to (55,0)
to (65,0) to (65,-5) to (70,-5) to (70,0) to (80,0) to (80,-5) to (85,-5) to (85,0) to (100,0)
to (100,30);

\draw (0,0) to (0,30) to (20,30) to (20,25) to (25,25) to (25,30) to  
(80,30) to (80,25) to (85,25) to (85,30) to (100,30);

\node (a) at (52.5,-2) {$+$};
\node (b) at (67.5,-2) {$-$};

\draw[gray] (37.5,0) to (37.5,30);
\draw[gray] (75,0) to (75,30);

\draw[fill] (10,0) circle (0.5);
\draw[fill] (12.5,0) circle (0.5);
\draw[fill] (37.5,15) circle (0.5);
\draw[fill] (37.5,17.5) circle (0.5);
\draw[fill] (75,15) circle (0.5);
\draw[fill] (75,17.5) circle (0.5);

\draw plot[smooth] coordinates { (10,0) (8,25) (15,20) (25,20) (37.5, 17.5)};
\draw plot[smooth] coordinates { (12.5,0) (20, 10) (37.5, 15)};

\draw plot[smooth] coordinates { (37.5,17.5) (45,25) (65,25) (75,17.5)};
\draw plot[smooth] coordinates { (37.5, 15) (45,5) ( 70,5) (75,15)};

\draw plot[smooth] coordinates { (75,17.5)   (80, 22.5) (95,22.5) (90,5) ( 85, 12) (75,15) };

\end{tikzpicture}

\caption{The proof of Lemma~\ref{lem:parc}.}
\label{fig:parcb}
\end{subfigure}
\caption{}
\label{fig:parc}
\end{figure}

As with Lemma~\ref{lem:parearr}, we also have the version of Lemma~\ref{lem:parc} where 
we exit the final subregion at adjacent points on the boundary of $R'$. In this way,
as with Lemma~\ref{lem:parearr}, we may vertically stack such regions to perform a 
succession of these parity cancellation operations.

We now give the proof of Theorem~\ref{thm:mtt}. We assume first that $R$ has total charge $0$ 
(which will be the case in our applications), and let $R'$ be obtained from $R$
by adding or subtracting at most $k$ parity change regions along the boundary of $R$
as in Figure~\ref{fig:imtt}. We are assuming that $R'$ has total charge $0$. 
We proceed by induction of the number of parity change regions that are used in forming $R'$ from $R$,
which we call $k$. The base cases when $k\leq 2$ have been verified. We consider the inductive case.

First assume that for some edge of $R$ there are two adjacent parity change regions of 
opposite parity. Partition $R'$ into two subregions as shown in Figure~\ref{fig:opc},
and add parity change regions so that the bottom subregion has total charge $0$,
and the top edge of the bottom subregion has fewer parity change regions than the bottom edge.

\begin{figure}[h]

\begin{tikzpicture}[scale=0.05]

\draw (0,0) to (20,0) to (20,-5) to (25,-5) to (25,0) to (50,0) to (50,-5) to (55,-5) to (55,0)
to (65,0) to (65,-5) to (70,-5) to (70,0) to (80,0) to (80,-5) to (85,-5) to (85,0) to (100,0);
to (100,30);

\draw (0,0) to (0,30) to (-5,30) to (-5,35) to (0,35) to (0,65) to (5,65) to (5,70) to 
(0,70) to (0,90) to (40,90) to (40,85) to (45,85) to (45,90) to (60,90) to (60,95) to (65,95)
to (65,90) to (100,90) to (100,80) to (95,80) to (95,75) to (100,75) to (100,50)
to (95,50) to (95,45) to (100,45) to (100,0);

\node (a) at (52.5,-2) {$+$};
\node (b) at (67.5,-2) {$-$};

\draw[gray] (0,20) to (20,20) to (20,15) to (25,15) to (25,20) to (80,20) to (80,15)
to (85,15) to (85,20) to (100,20);

\end{tikzpicture}
\caption{The case of adjacent regions of opposite parity.}
\label{fig:opc}
\end{figure}
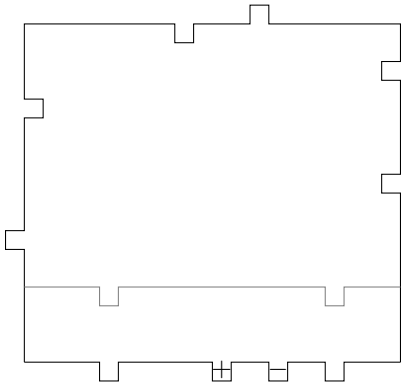

By Lemma~\ref{lem:parc} for the bottom region, and by induction for the top region, each
of these two subregions has a lining. As in Corollary~\ref{cor:mtt},
these can be combined into a lining for the region $R'$, and we may 
start and exit at any pair of adjacent points on the boundary of $R'$
(possibly adjusting the location of the subdividing line). 
Note that the top region still satisfies the hypotheses of Theorem~\ref{thm:mtt}
as the side lengths have decreased by at most $a_0$ and number of parity 
change regions for the top has strictly decreased.

Now suppose that all of the parity change regions along a given edge of $R$ have the same parity. 
There must be two adjacent edges of $R$ where these parity change regions along these two edges have 
opposite parities. Suppose, for example, the regions along the bottom and left edges have opposite 
parities, as illustrated in Figure~\ref{fig:opc2}. Using Lemma~\ref{lem:parearr} we may rearrange 
the bottom-most parity change region on the left edge 
to be within $a_0$ of the bottom edge. We subdivide $R'$ into subregions as shown in Figure~\ref{fig:opc2}. 
The bottom-left subregion is handled by the $\leq 2$ parity region case, the bottom-right subregion
by Lemma~\ref{lem:parearr}, and the top subregion by induction. Note that the number of parity change 
regions for the top subregion is strictly less than $k$. The hypotheses on the region of 
Theorem~\ref{thm:mtt} are easily satisfied for the top subregion since the the height of the 
bottom subregions is $a_0$.

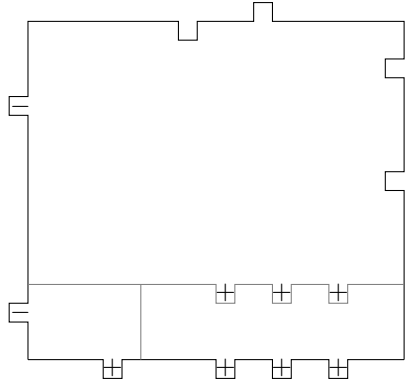
\begin{figure}[h]
\begin{tikzpicture}[scale=0.05]

\draw (0,0) to (20,0) to (20,-5) to (25,-5) to (25,0) to (50,0) to (50,-5) to (55,-5) to (55,0)
to (65,0) to (65,-5) to (70,-5) to (70,0) to (80,0) to (80,-5) to (85,-5) to (85,0) to (100,0);
to (100,30);

\draw (0,0) to (0,10) to (-5,10) to (-5,15) to (0,15) to (0,65) to (-5,65) to (-5,70) to 
(0,70) to (0,90) to (40,90) to (40,85) to (45,85) to (45,90) to (60,90) to (60,95) to (65,95)
to (65,90) to (100,90) to (100,80) to (95,80) to (95,75) to (100,75) to (100,50)
to (95,50) to (95,45) to (100,45) to (100,0);

\node (a) at (52.5,-2) {$+$};
\node (b) at (67.5,-2) {$+$};
\node (c) at (22.5, -2) {$+$};
\node (d) at (82.5, -2) {$+$};
\node (e) at (-2, 12.5) {$-$};
\node (f) at (-2, 67.5) {$-$};

\node (a) at (52.5,18) {$+$};
\node (b) at (67.5,18) {$+$};
\node (d) at (82.5, 18) {$+$};

\draw[gray] (0,20) to (50,20) to (50,15) to (55,15) to (55,20)  
to (65,20) to (65,15) to (70,15) to (70,20) to
(80,20) to (80,15)
to (85,15) to (85,20) to (100,20);

\draw[gray] (30,0) to (30,20);

\end{tikzpicture}
\caption{The proof of Theorem~\ref{thm:mtt}.}
\label{fig:opc2}
\end{figure}

This completes the proof of Theorem~\ref{thm:mtt}. 

We now turn to the proof of 
Theorem~\ref{thm:cls}. The proof is similar to that for matchings, Theorem~\ref{thm:mt}, using 
Theorem~\ref{thm:mtt} instead of the lemmas on matchings. As the argument is similar, we will 
just sketch the argument and highlight the differences.

We begin with the sequence of distances $d_1<d_2<\cdots$ of intermediate growth as before, that is, 
$C_1< \frac{d_{k+1}}{d_k}<C_2$ for some fixed constants $C_1$, $C_2$. We fix a sufficiently large even
integer $w$ which, as before, will be the width of the buffer regions in our construction. 
In particular, we take $w> 4 a_0 (p_0+q_0) (C'' K)^2$ where $C''$, as before, is an upper bound on the number of $k$-atoms that can be contained in a $(k+1)$-atom, and $K$ is as in Lemma~\ref{lem:bounded_charge_atoms}. As in the proof of 
Theorem~\ref{thm:mt}, $C''K$ represents the possible accumulation of charge as we extend the linings from the $k$-atoms
in a given $(k+1)$-atom to the buffered $(k+1)$-atom. Loosely speaking, now this extra charge will be carried
by a set of $p_0\times q_0$ parity change regions at certain locations. 
We again choose $d_1$ so that $\epsilon d_1 \gg w$, where $\epsilon=\alpha \epsilon_2$ is the 
orthogonality constant as explained in Remark~\ref{rem:oc}.
We define the buffered $k$-atoms as in Theorem~\ref{thm:mt}, that is, 
for every $k$-atom $A$ we let $A(w)=\{ x \in A\colon \rho(x, \partial A) \geq w\}$ be the 
$w$-buffered atom.

In our construction below, a {\em partial lining} is a Borel subgraph $L$ of the Schreier graph of $F(2^{\Z^2})$ 
such that each connected component of $S$ is a finite simple path of length at least $1$. Let $V(L)$ denote the vertex set of $L$. Then for each $x\in V(L)$, the degree of $x$ is either $1$ or $2$. Let $T(L)$ be the set of end points of $L$, that is, all $x\in V(L)$ where the degree of $x$ is $1$.

Our slightly modified induction hypothesis is as follows.

{\bf Induction Hypothesis ${\mathbf H}_{k-1}$}: 
We have defined a partial lining $L_{k-1}$ with each connected component of $L_{k-1}$ 
being a subset of a $(k-1)$-atom. For each $(k-1)$-atom $A\in \sA_{k-1}$, denote 
$E_{k-1}(A)=A(w) \sm V(L_{k-1})$. Then the following conditions hold for all $A\in\sA_{k-1}$:
\begin{enumerate}
\item[(i)] $L_{k-1}$ restricted to $A$ is a single connected component of $L_{k-1}$ with the two endpoints (forming the set $T(L_{k-1})\cap A$) being adjacent points on the boundary of $A(w)$; 
\item[(ii)] $E_{k-1}(A)$ is a disjoint union of at most $K$ many parity change regions 
along the boundary of $A(w)$;
\item[(iii)] Each of the parity change regions in $E_{k-1}(A)$ is at least 
$C'' K a_0$ from a corner of $A(w)$ and at least $C''K a_0$ from $T(L_{k-1})\cap A$, and the distance between them is at least $C'' K a_0$, where 
$a_0$ is as in Theorem~\ref{thm:mtt};
\item[(iv)] Each of the parity change regions in $E_{k-1}(A)$ is within distance $w$ of $X_{k-1}=\partial \sR_{k-1}$.
\end{enumerate}

\begin{lem}\label{lem:H1l} ${\mathbf H}_1$ holds.
\end{lem}

\begin{proof} Similar to Lemma~\ref{lem:H1}, we consider a buffered $1$-atom $B$ and define a partial lining $L$ of $B$ to satisfy the above conditions (i)--(iv). Our strategy of proof is still to divide $B$ into rectangular subregions (as illustrated in the example in Figure~\ref{fig:H1c}), define partial linings on each of them, and then connect these linings together. By Lemma~\ref{lem:bounded_charge} the total charge of $B$ is bounded by $K$, thus we need to subtract no more than $K$ many parity change regions to make the resulting region $B'$ of total charge $0$. Let $E$ be an edge of a subregion at the end (say $R_5$ in Figure~\ref{fig:H1c}) so that $E$ is within distance $w$ to $X_1$. Pick an appropriate number of parity change regions and pick adjacent points $x, y$ along $E$ so that the parity change regions are at least $C''Ka_0$ from the end of $E$ and are at least $C''Ka_0$ from $x, y$, and the distance between any two of them are at least $C''Ka_0$. Now we are ready to apply Theorem~\ref{thm:mtt} and Corollary~\ref{cor:mtt} to obtain a lining of $B'$ with $x, y$ as endpoints. More specifically, we start with the rectangular subregion that contains $x, y$ and extend the partial linings through the rest of the subregions. In the example in Figure~\ref{fig:H1c} we follow the order $R_5, R_4, R_0, R_1, R_2, R_3$. As we extend the partial lining from one subregion to the next, we substract a number of parity change regions as appropriate to make the current subregion of total charge $0$. These parity change regions will be added or subtracted along the common edge between the current subregion and the next one. Then we use Corollary~\ref{cor:mtt} to extend the partial lining to the next one. Repeating this until we reach the rectangular region on the other end (in our case $R_3$ in Figure~\ref{fig:H1c}), in which case the region must have total charge $0$ already, we apply Theorem~\ref{thm:mtt} to complete the lining. The strategy works because the lengths of the common edges between any two neighboring subregions is at least $\epsilon d_1\gg w> 2(K+1)\cdot (2C''Ka_0)$.
\end{proof}

Consider the inductive step where we assume ${\mathbf H}_{k-1}$ and show ${\mathbf H}_k$. 
As in the proof of Theorem~\ref{thm:mt} we let $\sS$ denote the set of $(k-1)$-atoms in a given 
$k$-atom $A$ which we are considering. We again define 
the adjacency graph $G$ on $\sS$ and a rooted spanning tree $T$ for the graph $G$. For each $S\in \sS$ we also fix an exit segment $E_S$, which has distance $1$ to the atom $S'\in\sS$ which is the parent node of $S$ in $T$. Also, for each $S\in \sS$ let $T_S$ be the subtree of $T$ with the root $S$, that is, $T_S$ consists of $S$ together with all nodes of $T$ below $S$ (not just immediately below). We denote $t(S)=|T_S|$.
As before, we proceed by induction on the rank $r(S)$ of the atom $S\in \sS$ in $T$. 
We maintain the following subinduction hypothesis.

{\bf Subinduction Hypothesis ${\mathbf H}_{k-1}(r-1)$}: For each $S\in \sS$ with $r(S)\leq r-1$, the partial lining $L_{k-1}$ has been extended to $L_{k-1}(r-1)$ with $V(L_{k-1}(r-1))\cap S=(A(w)\cap S)-E_{k-1}'(S)$, where the following conditions hold:
\begin{enumerate}
\item[(i')] $L_{k-1}(r-1)$ restricted to 
$$ Q_S=\bigcup_{S'\in T_S}S'\cap A(w)  $$
is a single connected component of $L_{k-1}(r-1)$ with the two endpoints $x_S, y_S$ being adjacent points on $E_S$;
\item[(ii')] $E_{k-1}'(S)$ is a disjoint union of at most $K\cdot t(S)$ many parity change regions along $E_S$;
\item[(iii')] Each of the parity change regions in $E_{k-1}'(S)$ is at least $C''Ka_0$ from the end of $E_S$, at least $C''Ka_0$ from $x_S, y_S$, and the distance between them is at least $C''Ka_0$.
\end{enumerate}

\begin{lem}\label{lem:Hk0l} ${\mathbf H}_{k-1}(0)$ holds.
\end{lem}

\begin{proof} Similar to the proof of Lemma~\ref{lem:Hk0}, we consider two cases. 
The first case is when all of the boundary of $S$ are on the $k-1$ level. 
By the induction hypothesis ${\mathbf H}_{k-1}$, $E_{k-1}(S)$ is a disjoint union of 
at most $K$ many parity change regions along the boundary of $S(w)$. 
We need to extend the partial lining $L_{k-1}$ to $L_{k-1}(0)$ with $V(L_{k-1}(0))\cap S=S-E_{k-1}'(S)$ 
where (i')--(iii') hold. 
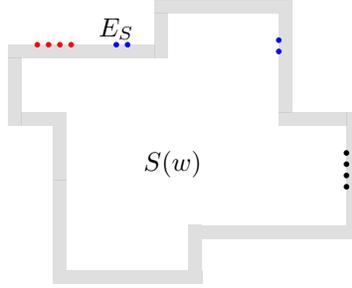
\begin{figure}[h]

\begin{tikzpicture}[scale=0.03]

\pgfmathsetmacro{\a}{3};


\draw[draw=none,fill=lightgray, opacity=0.5] (20-\a,0-\a) rectangle ((20+\a,40+\a);

\draw[draw=none,fill=lightgray, opacity=0.5] (20+\a,0-\a) rectangle ((80+\a,0+\a);

\draw[draw=none,fill=lightgray, opacity=0.5] (80-\a,0+\a) rectangle ((80+\a,20+\a);
\draw[draw=none,fill=lightgray, opacity=0.5] (80+\a,20-\a) rectangle ((150+\a,20+\a);
\draw[draw=none,fill=lightgray, opacity=0.5] (150-\a,20+\a) rectangle ((150+\a,70+\a);
\draw[draw=none,fill=lightgray, opacity=0.5] (120-\a,70-\a) rectangle ((150-\a,70+\a);
\draw[draw=none,fill=lightgray, opacity=0.5] (120-\a,70+\a) rectangle ((120+\a,120+\a);
\draw[draw=none,fill=lightgray, opacity=0.5] (65-\a,120-\a) rectangle ((120-\a,120+\a);
\draw[draw=none,fill=lightgray, opacity=0.5] (65-\a,100-\a) rectangle ((65+\a,120-\a);
\draw[draw=none,fill=lightgray, opacity=0.5] (0-\a,100-\a) rectangle ((65-\a,100+\a);
\draw[draw=none,fill=lightgray, opacity=0.5] (0-\a,70-\a) rectangle ((0+\a,100-\a);
\draw[draw=none,fill=lightgray, opacity=0.5] (0+\a,70-\a) rectangle ((20+\a,70+\a);
\draw[draw=none,fill=lightgray, opacity=0.5] (20-\a,40+\a) rectangle ((20+\a,70-\a);

\draw[fill, blue] (120-\a,105) circle (1);
\draw[fill, blue] (120-\a,100) circle (1);

\draw[fill] (150-\a,55) circle (1);
\draw[fill] (150-\a,50) circle (1);
\draw[fill] (150-\a,45) circle (1);
\draw[fill] (150-\a,40) circle (1);

\draw[fill,red] (10,100+\a) circle (1);
\draw[fill,red] (15,100+\a) circle (1);
\draw[fill,red] (20,100+\a) circle (1);
\draw[fill,red] (25,100+\a) circle (1);

\draw[fill,blue] (45,100+\a) circle (1);
\draw[fill,blue] (50,100+\a) circle (1);

\node at (70, 50) {$S(w)$};
\node at (45, 110) {$E_S$};


\end{tikzpicture}
\caption{The construction of $L_{k-1}(0)$ for a terminal node of $T$. The black dots represent parity change regions in $E_{k-1}(S)$, the red dots represent parity change regions $E_{k-1}'(S)$, and the blue dots represent endpoints of partial linings.}\label{fig:Hk0l}
\end{figure}

The strategy of our construction is to divide $\partial_w(S)$ into rectangular regions and apply Theorem~\ref{thm:mtt} and Corollary~\ref{cor:mtt} to ``transport" the necessary parity change regions to be along $E_S$. In order to keep each region of total charge $0$, we add or substract parity change regions as appropriate along the common edge between the current region and the next one. This strategy works because $w$ is even  and so the total charge of $\partial_w(S)$ is $0$. Thus at the end there are no more than $K$ many parity change regions needed to be placed along $E_S$ to form $E_{k-1}'(S)$. Also, because $w>4a_0(p_0+q_0)(C''K)^2$, up to $K$ many parity change regions can be placed on the common edge (of length $w$) between any two rectangular regions in the middle of our construction, which is sufficient to accommodate the acumulation of charges along any path segment of $\partial X_{k-1}$. 

The second case is when some part of the boundary of $S$ are on higher levels. In this case we work with the region $\partial_w(S)-\partial_w(A)$, which is not a loop but still connected. By the induction hypothesis ${\mathbf H}_{k-1}$ the parity change regions in $E_{k-1}(S)$ are connected with $\partial_w(S)-\partial_w(A)$. Hence our strategy of construction above still works.
\end{proof}

Next we assume the subinduction hypothesis ${\mathbf H}_{k-1}(r-1)$ and show ${\mathbf H}_{k-1}(r)$. Suppose $r(S)=r$. 
Let $S'$ denote the parent atom of $S$ in the tree $T$, and $P_1,\dots, P_m$
denote the children of $S$ in $T$. Let $L_1, \dots, L_m$ be the partial linings of $Q_{P_1}, \dots, Q_{P_m}$ given by ${\mathbf H}_{k-1}(r-1)$. Let $L_{m+1}$ be the partial lining $L_{k-1}$ restricted to $S(w)$. We need to extend 
$L_1,\dots, L_m, L_{m+1}$ to a partial lining of 
$Q_S$. 

We give the construction first for the case that all of the boundary of $S$ is on the $k-1$ level. Figure~\ref{fig:isl} illustrates this step. 
In this figure, the shaded area in gray is the region to which we must extend the partial 
linings $L_0\cup L_1 \cup \cdots \cup L_m$, where in the figure we use $m=2$. 
The central region in white is the buffered atom $S(w)$ except for some parity change regions
along the top edge of $S(w)$ shown in gray. The small gray squares along the left and bottom
represent the parity change regions from the child atoms $P_1$ and $P_2$. 
The three pairs of red dots represent the endpoints of the partial linings 
of $P_1$, $P_2$, and $S$. The pair of green dots represents the terminal points 
of the partial lining we will construct at this stage of the subinduction.

\begin{figure}[h]
\begin{tikzpicture}[scale=0.05]

\pgfmathsetmacro{\a}{5};
\pgfmathsetmacro{\b}{3};


\draw[draw=none,fill=lightgray, opacity=0.5] (20-\a,0-\a) rectangle (20+\a,40+\a);
\draw[draw=none, fill=lightgray, opacity=1.0] (20-\a-\b, 40) rectangle (20-\a, 40+\b);
\draw[draw=none, fill=lightgray, opacity=1.0] (20-\a-\b, 40+2*\b) rectangle (20-\a, 40+3*\b);
\draw[fill=red] (20-\a, 40-3*\b) circle(1.0);
\draw[fill=red]  (20-\a, 40-5*\b) circle(1.0);
\draw plot[smooth] coordinates { (20-\a, 40-3*\b) (-\a, 40-3*\b) (-10-\a, 40-3*\b+10)};
\draw plot[smooth] coordinates { (20-\a, 40-5*\b) (-\a, 40-5*\b) (-10-\a, 40-3*\b-10)};

\draw[draw=none,fill=lightgray, opacity=0.5] (20+\a,0-\a) rectangle (80+\a,0+\a);

\draw[draw=none,fill=lightgray, opacity=0.5] (80-\a,0+\a) rectangle (80+\a,20+\a);

\draw[draw=none,fill=lightgray, opacity=0.5] (80+\a,20-\a) rectangle (150+\a,20+\a);
\draw[draw=none, fill=lightgray, opacity=1.0] (100, 20-\a-\b) rectangle (100+\b, 20-\a);
\draw[draw=none, fill=lightgray, opacity=1.0] (100+2*\b, 20-\a-\b) rectangle (100+3*\b, 20-\a);
\draw[fill=red] (100+5*\b, 20-\a) circle(1.0);
\draw[fill=red]  (100+7*\b, 20-\a) circle(1.0);
\draw plot[smooth] coordinates { (100+5*\b, 20-\a) (100+5*\b+5, 20-\a-10) (140+5*\b-5, 20-\a-20)
(140+7*\b-5, 20-\a-20)};
\draw plot[smooth] coordinates { (100+7*\b, 20-\a) (100+7*\b+10, 20-\a-7) (140+7*\b-5, 20-\a-10)};

\draw[draw=none,fill=lightgray, opacity=0.5] (150-\a,20+\a) rectangle (150+\a,70+\a);
\draw[draw=none,fill=lightgray, opacity=0.5] (120-\a,70-\a) rectangle (150-\a,70+\a);
\draw[draw=none,fill=lightgray, opacity=0.5] (120-\a,70+\a) rectangle (120+\a,120+\a);

\draw[draw=none,fill=lightgray, opacity=0.5] (65-\a,120-\a) rectangle (120-\a,120+\a);
\draw[draw=none, fill=lightgray, opacity=1.0] (80, 120-\a-\b) rectangle (80+\b, 120-\a);
\draw[draw=none, fill=lightgray, opacity=1.0] (80+2*\b, 120-\a-\b) rectangle (80+3*\b, 120-\a);
\draw[fill=red] (80+5*\b, 120-\a) circle (1.0);
\draw[fill=red] (80+7*\b, 120-\a) circle (1.0);

\node at (100, 40) {$S(w)$};
\node at (0, 10) {$P_1$};
\node at (100, 0) {$P_2$};
\node at (100, 130) {$E_S$};

\draw[draw=none,fill=lightgray, opacity=0.5] (65-\a,100-\a) rectangle (65+\a,120-\a);
\draw[draw=none,fill=lightgray, opacity=0.5] (0-\a,100-\a) rectangle (65-\a,100+\a);
\draw[draw=none,fill=lightgray, opacity=0.5] (0-\a,70-\a) rectangle (0+\a,100-\a);
\draw[draw=none,fill=lightgray, opacity=0.5] (0+\a,70-\a) rectangle (20+\a,70+\a);
\draw[draw=none,fill=lightgray, opacity=0.5] (20-\a,40+\a) rectangle (20+\a,70-\a);

\draw plot[smooth] coordinates { (80+5*\b, 120-\a) (80+5*\b-10, 120-\a-30)
 (30+5*\b-10, 95-\a-10) (60+5*\b-10, 60-\a-10) (95,60)     (80+7*\b, 120-\a)
};

\end{tikzpicture}
\caption{The inductive proof of ${\mathbf H}_{k-1}(r)$ from ${\mathbf H}_{k-1}(r-1)$.}
\label{fig:isl}

\end{figure}
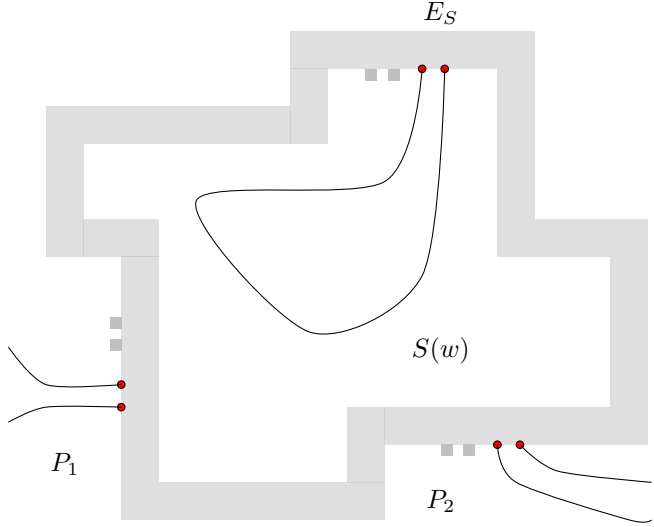


Let 
$$ B(w)=\partial_w(S)\cup E_{k-1}(S)\cup\bigcup_{i=1}^{m} E_{k-1}'(P_i).$$
$B(w)$ is illustrated by the grey area in Figure~\ref{fig:isl}, which is the $w$-buffer of the $(k-1)$-atom $S$ together with the parity change regions from $S(w)$, $P_1$ and $P_2$.

By the subinduction hypothesis ${\mathbf H}_{k-1}(r-1)$, there are at most $K\cdot t(P_i)$ many parity change regions in $E_{k-1}'(P_i)$. 
Also, by the induction hypothesis ${\mathbf H}_{k-1}$, there are at most $K$ many parity change regions in $E_{k-1}(S)$. 

We partition $\partial_w(S)$ into rectangular subregions by imposing vertical or horizontal boundary lines in between neighboring subregions. This is illustrated in Figure~\ref{fig:wcon}. The partition is done in a way so that each subregion in the partition does not split the endpoints of $L_i$ for $i=1,\dots, m,m+1$, no subregion contains more than one pair of endpoints of $L_i$, and the dividing lines do not come within distance $C''Ka_0$ of any of the parity change regions in $E_{k-1}(S)$ or $E_{k-1}'(P_i)$. Without loss of generality we may assume that the endpoints of $L_i$ for $i=1,\dots, m, m+1$ appear on the boundary of $\partial_w(S)$ in the counter clockwise order as $L_1, \dots, L_m, L_{m+1}$. For each $i=1,\dots, m, m+1$, let $W_i$ denote the rectangular subregion containing the endpoints of $L_i$ together with any parity change regions connecting to it. By refining the partition if necessary, we assume in addition that in between $W_i$ and $W_{i+1}$ for $1\leq i\leq m$, there is a rectangular subregion $W_i'$ in the partition not connecting to any parity change regions (in other words $W_i'$ is a rectangle).


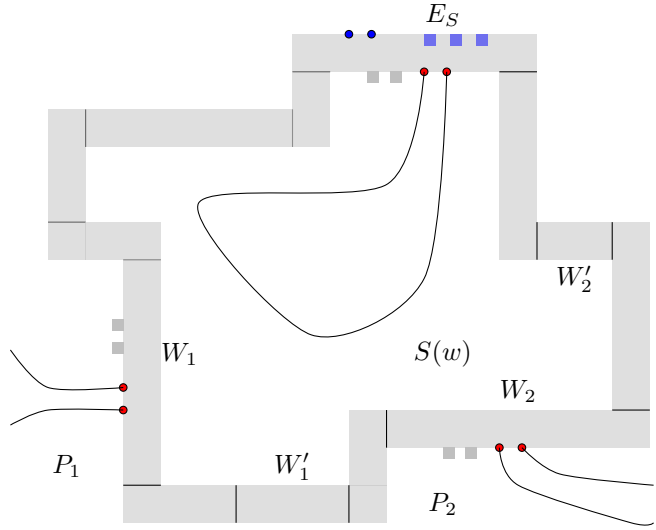
\begin{figure}[h]
\begin{tikzpicture}[scale=0.05]

\pgfmathsetmacro{\a}{5};
\pgfmathsetmacro{\b}{3};


\draw[draw=none,fill=lightgray, opacity=0.5] (20-\a,0-\a) rectangle (20+\a,40+\a);
\draw[draw=none, fill=lightgray, opacity=1.0] (20-\a-\b, 40) rectangle (20-\a, 40+\b);
\draw[draw=none, fill=lightgray, opacity=1.0] (20-\a-\b, 40+2*\b) rectangle (20-\a, 40+3*\b);
\draw[fill=red] (20-\a, 40-3*\b) circle(1.0);
\draw[fill=red]  (20-\a, 40-5*\b) circle(1.0);
\draw plot[smooth] coordinates { (20-\a, 40-3*\b) (-\a, 40-3*\b) (-10-\a, 40-3*\b+10)};
\draw plot[smooth] coordinates { (20-\a, 40-5*\b) (-\a, 40-5*\b) (-10-\a, 40-3*\b-10)};


\draw[draw=none,fill=lightgray, opacity=0.5] (20+\a,0-\a) rectangle (80+\a,0+\a);

\draw[draw=none,fill=lightgray, opacity=0.5] (80-\a,0+\a) rectangle (80+\a,20+\a);

\draw[draw=none,fill=lightgray, opacity=0.5] (80+\a,20-\a) rectangle (150+\a,20+\a);
\draw[draw=none, fill=lightgray, opacity=1.0] (100, 20-\a-\b) rectangle (100+\b, 20-\a);
\draw[draw=none, fill=lightgray, opacity=1.0] (100+2*\b, 20-\a-\b) rectangle (100+3*\b, 20-\a);
\draw[fill=red] (100+5*\b, 20-\a) circle(1.0);
\draw[fill=red]  (100+7*\b, 20-\a) circle(1.0);


\draw plot[smooth] coordinates { (100+5*\b, 20-\a) (100+5*\b+5, 20-\a-10) (140+5*\b-5, 20-\a-20)
(140+7*\b-5, 20-\a-20)};
\draw plot[smooth] coordinates { (100+7*\b, 20-\a) (100+7*\b+10, 20-\a-7) (140+7*\b-5, 20-\a-10)};

\draw[draw=none,fill=lightgray, opacity=0.5] (150-\a,20+\a) rectangle (150+\a,70+\a);
\draw[draw=none,fill=lightgray, opacity=0.5] (120-\a,70-\a) rectangle (150-\a,70+\a);
\draw[draw=none,fill=lightgray, opacity=0.5] (120-\a,70+\a) rectangle (120+\a,120+\a);

\draw[draw=none,fill=lightgray, opacity=0.5] (65-\a,120-\a) rectangle (120-\a,120+\a);
\draw[draw=none, fill=lightgray, opacity=1.0] (80, 120-\a-\b) rectangle (80+\b, 120-\a);
\draw[draw=none, fill=lightgray, opacity=1.0] (80+2*\b, 120-\a-\b) rectangle (80+3*\b, 120-\a);
\draw[fill=red] (80+5*\b, 120-\a) circle (1.0);
\draw[fill=red] (80+7*\b, 120-\a) circle (1.0);


\draw[fill=blue] (75, 120+\a) circle(1.0);
\draw[fill=blue]  (75+2*\b, 120+\a) circle(1.0);

\draw (55+\a, 110+\a) to (65+\a, 110+\a);
\draw (110+\a, 110+\a) to (120+\a, 110+\a);
\draw (120+\a, 70+\a) to (120+\a, 60+\a);
\draw (140+\a, 70+\a) to (140+\a, 60+\a);
\draw (140+\a, 20+\a) to (150+\a, 20+\a);
\draw (80+\a, 20+\a) to (80+\a, 10+\a);
\draw (70+\a, \a) to (70+\a, -\a);
\draw (10+\a, \a) to (20+\a, \a);
\draw (10+\a, 60+\a) to (20+\a, 60+\a);
\draw (-\a, 70+\a) to (10-\a,70+\a);
\draw (10-\a, 90+\a) to (10-\a, 100+\a);
\draw (55+\a, 90+\a) to (55+\a, 100+\a);
\draw (40+\a, \a) to (40+\a, -10+\a);

\draw[draw=none,fill=blue, opacity=0.5] (95,131-2*\b) rectangle (95+\b,131-3*\b);
\draw[draw=none,fill=blue, opacity=0.5] (102,131-2*\b) rectangle (102+\b,131-3*\b);
\draw[draw=none,fill=blue, opacity=0.5] (109,131-2*\b) rectangle (109+\b,131-3*\b);

\draw[draw=none,fill=lightgray, opacity=0.5] (65-\a,100-\a) rectangle (65+\a,120-\a);
\draw[draw=none,fill=lightgray, opacity=0.5] (0-\a,100-\a) rectangle (65-\a,100+\a);
\draw[draw=none,fill=lightgray, opacity=0.5] (0-\a,70-\a) rectangle (0+\a,100-\a);
\draw[draw=none,fill=lightgray, opacity=0.5] (0+\a,70-\a) rectangle (20+\a,70+\a);
\draw[draw=none,fill=lightgray, opacity=0.5] (20-\a,40+\a) rectangle (20+\a,70-\a);

\draw plot[smooth] coordinates { (80+5*\b, 120-\a) (80+5*\b-10, 120-\a-30)
 (30+5*\b-10, 95-\a-10) (60+5*\b-10, 60-\a-10) (95,60)     (80+7*\b, 120-\a)
};

\node at (100, 40) {$S(w)$};
\node at (0, 10) {$P_1$};
\node at (100, 0) {$P_2$};
\node at (100, 130) {$E_S$};
\node at (30,40) {$W_1$};
\node at (60,10) {$W_1'$};
\node at (120, 30) {$W_2$};
\node at (135, 60) {$W_2'$};


\end{tikzpicture}
\caption{A partition of $B(w)$ into rectangular subregions with parity change regions.}
\label{fig:wcon}

\end{figure}


Because $w$ is even, the total charge of $\partial_w(S)$ is $0$. Thus the total charge of $B(w)$ is bounded by 
$$ K+\sum_{i=1}^m K\cdot t(P_i)=K\cdot t(S). $$
We place an appropriate number of parity change regions along $E_S$ so that the adjusted region $B'(w)$ has total charge $0$. In Figure~\ref{fig:wcon} these parity change regions are illustrated as small blue regions.

We are now ready to construct a partial lining $L'$ of $B'(w)$ that connects the existing linings $L_1,\dots, L_m, L_{m+1}$. After this lining $L'$ is constructed, we find two adjacent points on $E_S$ that are also adjacent in $L'$ and have distance at least $C''Ka_0$ from the ends of $E_S$ and at least $C''Ka_0$ from the parity change regions, and designate them $x_S,y_S$ as required.

The basic strategy of our construction is still to start from $W_1$ and, going counter clockwise, successively extend the partial lining from one subregion in the partition of $B’(w)$ to the next, while ``transporting" the total charges to the next region by adding or subtracting an appropriate number of parity change regions along the common edge between the current subregion and the next one. Our parameters are sufficiently large to carry out this strategy. For instance, by Lemma~\ref{lem:bounded_charge} the accumulation of charges along the way is bounded by $C''K+K$, and thus there are no more than $(C''+1)K$ many parity change regions needed during the construction of the extension of the partial lining from one subregion to another. Because $w> 4a_0(p_0+q_0)(C''K)^2$, the necessary number of parity change regions can be appropriately placed on the edge between the current subregion and the next one. However, this strategy will become problematic when the next subregion is $W_2$, since Corollary~\ref{cor:mtt} does not guarantee the existence of a lining of $W_2$ with the endpoints of $L_2$ connected by an edge. For this we will adopt an modification as follows.

To simplify notation we assume that $W_1, W_1'$ and $W_2$ are adjacent in the partition of $B'(w)$. This is illustrated in Figure~\ref{fig:W1W2}. We first divide $W_1'$ further into two rectangular regions both connected to $W_1$ and $W_2$. The side lengths of $U_1$ only need to be at least $4$. This will make the side lengths of $U_0$ sufficiently large for our construction below. We first place an appropriate number of parity change regions on the edge between $W_1$ and $U_0$ so that the modified region $W_1$ has total charge $0$. Pick an end point $z$ on the boundary between $W_1$ and $U_0$ and apply Corollary~\ref{cor:mtt} to obtain two line sections in $W_1$ with exit points $z$ and $z'$, where $z'$ is adjacent to $z$ and is also on the boundary between $W_1$ and $U_0$. Then we place an appropriate number of parity change regions on the boundary between $U_0$ and $W_2$ so that the total charge of the modified $U_0$ is $0$. Apply Theorem~\ref{thm:mtt} to obtain a lining of $U_0$ with $z, z'$ as endpoints.

In $W_2$ place an appropriate number of parity change regions on the edge between $W_2$ and the next subregion in the partition so that the total charge of the modified $W_2$ region is $0$. Apply Theorem~\ref{thm:mtt} to obtain a lining of $W_2$ with the endpoints of $L_2$ as endpoints.

\begin{center}
\begin{figure}[h]
\begin{tikzpicture}[scale=0.1]

\draw (0,0) to (20,0) to (20,-2) to (22,-2) to (22,0) to (25,0) to (25,-2) to (27,-2) to (27,0) to (80,0) to (80,-2) to (82,-2) to (82,0) to (85,0) to (85,-2) to (87,-2) to (87,0) to (100,0);

\draw (0,0) to (0,30) to (100,30);


\draw (37.5,0) to (37.5,21) to (35.5, 21) to (35.5, 23) to (37.5, 23) to (37.5, 26) to (35.5, 26) to (35.5, 28) to (37.5, 28) to (37.5, 30);
\draw (75,0) to (75,21) to (73, 21) to (73, 23) to (75, 23) to (75, 26) to (73, 26) to (73, 28) to (75, 28) to (75, 30);
\draw (100,0) to (100,11) to (98,11) to (98,13) to (100, 13) to (100, 16) to (98, 16) to (98,18) to (100, 18) to (100,21) to (98, 21) to (98, 23) to (100, 23) to (100, 26) to (98, 26) to (98, 28) to (100, 28) to (100, 30);

\draw[fill] (10,0) circle (0.5);
\draw[fill] (12.5,0) circle (0.5);
\draw[fill] (37.5,15) circle (0.5);
\draw[fill] (37.5,17.5) circle (0.5);
\draw[fill] (92,0) circle (0.5);
\draw[fill] (94.5,0) circle (0.5);
\draw[fill, blue] (37,2) circle(0.5);
\draw[fill, blue] (37,4) circle (0.5);
\draw[fill, blue] (38,2) circle (0.5);
\draw[fill, blue] (38,4) circle (0.5);
\draw[fill, blue] (75.5,4.5) circle (0.5);
\draw[fill, blue] (75.5,6.5) circle (0.5);
\draw[fill, blue] (74.5,4.5) circle (0.5);
\draw[fill, blue] (74.5,6.5) circle (0.5);

\draw plot[smooth] coordinates { (10,0) (8,25) (15,20) (25,20) (37.5, 17.5)};
\draw plot[smooth] coordinates { (10,0) (11,-3) (10,-7)};
\draw plot[smooth] coordinates { (12.5,0) (12,-2) (15,-5)};
\draw plot[smooth] coordinates { (12.5,0) (20, 10) (37, 2)};
\draw plot[smooth] coordinates {(37,4) (30, 8) (37.5, 15)};

\draw plot[smooth] coordinates { (37.5,17.5) (45,25) (65,25) (72,17.5)
 ( 70,10) (45,17)(37.5, 15) };

\draw[dashed] (37.5, 8) to (75,8);
\draw (37,2) to (37,4);
\draw (75.5,4.5) to (75.5,6.5);

\draw plot[smooth] coordinates { (75.5,6.5)   (80, 22.5) (95,22.5) (94.5,0)};
\draw plot[smooth] coordinates {(92,0) (90,5) ( 85, 12) (75.5, 4.5) };
\draw plot[smooth] coordinates {(92,0) (90,-2) (92,-5)};
\draw plot[smooth] coordinates {(94.5,0) (96,-3) (98,-4)};

\node at (7, -5) {$L_1$};
\node at (95,-5) {$L_2$};
\node at (25, 32) {$W_1$};
\node at (87, 32) {$W_2$};
\node at (60, 32) {$W_1'$};
\node at (43,27) {$U_0$};
\node at (43,3) {$U_1$};
\node at (35,16) {$z$};

\end{tikzpicture}

\caption{Connecting $L_1$ with $L_2$.}
\label{fig:W1W2}
\end{figure}
\end{center}

Now find two adjacent points $x_0,x_1$ on the boundary between $W_1$ and $U_1$ which are also adjacent in the lining. Let $y_0, y_1$ be the corresponding points on the boundary of $U_1$. Similarly find two adjacent points $x_0', x_1'$ on the boundary between $W_2$ and $U_1$. Let $y_0', y_1'$ be the corresponding points on the boundary of $U_1$. Apply Lemma~\ref{lem:k=0} to obtain two line sections in $U_1$ with $x_0', x_1'$ and $y_0', y_1'$ as endpoints. Remove the edge in between $x_0$ and $x_1$, and connect $x_0$ with $x_0'$ and $x_1$ with $x_1'$. Similarly, remove the edge in between $y_0$ and $y_1$, and connect $y_0$ with $y_0'$ and $y_1$ with $y_1'$. The resulting lining is as required.

Finally, find two adjacent points (not shown in Figure~\ref{fig:W1W2}) on the boundary between $W_2$ and the next subregion, remove the edge in between them, and designate them as the terminal points of the lining. We are now ready to continue the algorithm and extend the lining to the next subregion.

Repeating the same algorithm, we obtain a lining of the entire region from $W_1$ to $W_{m+1}$. If there are more subregions between $W_{m+1}$ and $W_1$, we continue the algorithm to cover them as well, except that there will not be any more parity change regions necessary beyond $W_{m+1}$. Hence we are able to extend the lining to $B'(w)$ as required. This finishes the construction of the lining as a subinduction step. It is easy to see that the subinduction hypothesis is maintained.

Next we consider the case where some part of the boundary of $S$ is not on the $k-1$ level. In this case we need to extend the linings $L_1, \dots, L_m, L_{m+1}$ to the cover the region $\partial_w(S)-\partial_w(A)$, which is not topologically a loop but still connected. In this case we need to modify our algorithm as follows. After we divide the region $\partial_w(S)-\partial_w(A)$ into rectangular subregions, the one containing the endpoints of $L_{m+1}$, $W_{m+1}$, might be in the middle rather than at one end. In this case we apply our algorithm from the two far ends of the partition and work our way toward the middle, ending in $W_{m+1}$ as the last region to be dealt with. Noting that we still have $\partial_w(S)-\partial_w(A)$ has total charge $0$ as $w$ is even, the number of parity change regions that are needed to be placed on $E_S$ is predetermined, and so when our construction reaches $W_{m+1}$, it automatically will have total charge $0$. So the above algorithm conneting $L_1$ with $L_2$ can be performed on three regions simultaneously with $W_{m+1}$ in the middle. The resulting lining satisfies the subinduction hypothesis.

Now we have completed the subinduction and have obtained a lining for the entire buffered $k$-atom $A$ except for a final set $E$ of parity change regions coming from the $(k-1)$-atom
corresponding to the root of the tree $T$. Superficially, it seems that we have as many as 
$C'' K$ many parity changes required for the set $E$. However, by Lemma~\ref{lem:bounded_charge} again, the total charge of $A$ is still at most $K$.
Since we have now defined a partital lining of $A$ except for the points in $E$,
we in fact have that the number of parity change regions in $E$ is bounded by $K$. This is the key point, that the number of parity change regions
required as we exit to the higher level atom (the $k$-atom $A$ in our notation) is bounded by a 
fixed constant independent of the level $k$ of the construction (which is necessary as the width 
$w$ of the buffers is a fixed number).

This established the inductive hypothesis ${\mathbf H}_k$ for all $k$. Since  
$$F(2^{\Z^2})=\bigcup_k \bigcup\{A(w)\,:\, \mbox{ $A$ is a $k$-atom}\},$$ 
we have constructed a Borel lining for $F(2^{\Z^2})$.
This completes the proof of Theorem~\ref{thm:cls}.

\end{document}